\pdfoutput=1
\documentclass[a4paper]{amsart}
\usepackage[hmarginratio={1:1},vmarginratio={1:1},lmargin=85pt,tmargin=110pt]{geometry}

\usepackage{calligra}
\usepackage{fontenc}

\usepackage{mathtools}
\mathtoolsset{mathic}
\usepackage[final]{microtype}
\usepackage{multicol}
\usepackage{tabularx}
\usepackage{latexsym,exscale,enumitem,amsfonts,amssymb,mathtools}
\usepackage{amsmath,amsthm,amsfonts,amssymb,amscd,textcomp,xcolor,bbm}
\usepackage{stmaryrd}
\SetSymbolFont{stmry}{bold}{U}{stmry}{m}{n}
\usepackage[normalem]{ulem}
\usepackage{thmtools}
\usepackage{etex}
\usepackage{young,youngtab}
\usepackage{fancybox}

\usepackage{caption} 
\usepackage[all]{xy}
\SelectTips{cm}{}

\usepackage{tikz}
\tikzset{anchorbase/.style={baseline={([yshift=-0.5ex]current bounding box.center)}}}
\usetikzlibrary{decorations.markings}
\usetikzlibrary{decorations.pathreplacing}
\usetikzlibrary{arrows,shapes,positioning}
\tikzstyle directed=[postaction={decorate,decoration={markings,
    mark=at position #1 with {\arrow{>}}}}]
\tikzstyle rdirected=[postaction={decorate,decoration={markings,
    mark=at position #1 with {\arrow{<}}}}]
\tikzset{
    partial ellipse/.style args={#1:#2:#3}{
        insert path={+ (#1:#3) arc (#1:#2:#3)}
    }
}

\newcommand{\C}{\mathbb{C}}
\newcommand{\Z}{\mathbb{Z}}
\newcommand{\R}{\mathbb{R}}
\newcommand{\Q}{\mathbb{Q}}

\newcommand{\sym}{\mathrm{Sym}}
\newcommand{\symz}{\mathrm{Sym}_{\Z}}

\newcommand{\trace}{\mathrm{tr}}

\newcommand{\LR}[3]{\mathrm{c}_{#1#2}^{#3}}
\newcommand{\schur}[1]{\mathrm{s}_{#1}}
\newcommand{\fullsym}[1]{\mathrm{h}_{#1}}
\newcommand{\elsym}[1]{\mathrm{e}_{#1}}
\newcommand{\partition}[2]{\mathrm{P}(#1,#2)}
\newcommand{\partitions}[1]{\mathrm{P}(#1)}
\newcommand{\partitionall}{\mathrm{P}}

\newcommand{\Ybox}[1]{\mathrm{box}_{#1}}
\newcommand{\complementnorm}[1]{{#1}^{\mathrm{c}}}
\newcommand{\complementtrans}[1]{\widehat{#1}}

\newcommand{\transpose}[1]{\overline{#1}}

\newcommand{\A}{\mathbb{A}}
\newcommand{\B}{\mathbb{B}}
\newcommand{\X}{\mathbb{X}}
\newcommand{\Eq}{\mathbb{S}}
\newcommand{\roots}{\lambda}
\newcommand{\alphS}{\mathtt{\Sigma}}
\newcommand{\Webs}{V}

\newcommand{\Kom}{\boldsymbol{\mathrm{Kom}}}
\newcommand{\Komh}{\boldsymbol{\mathrm{K}}}
\newcommand{\Kar}{\boldsymbol{\mathrm{Kar}}}

\DeclareSymbolFont{extraup}{U}{zavm}{m}{n}
\DeclareMathSymbol{\varheart}{\mathalpha}{extraup}{86}

\newcommand{\Hgen}[1]{\left\llbracket #1 \right\rrbracket}
\newcommand{\HgenS}[1]{\left\llbracket #1 \right\rrbracket^{\alphS}}

\newcommand{\AB}{AB}

\newcommand{\BC}{BC}

\newcommand{\AandC}{Y}
\newcommand{\BA}{BA}
\newcommand{\CA}{CA}
\newcommand{\CB}{CB}
\newcommand{\wcolor}[1]{\mathtt{#1}}

\newcommand{\MM}[1]{\mathrm{MM}#1}

\newcommand{\coltanglesfree}{\boldsymbol{\mathrm{TD}}^{\ast}}
\newcommand{\coltangles}{\boldsymbol{\mathrm{TD}}}

\newcommand{\scalar}[1]{r(\mathtt{#1})}

\newcommand{\idem}[1]{\boldsymbol{\mathbbm{1}}_{\!\mathtt{#1}}}

\newcommand{\freefoamSet}{\boldsymbol{\mathrm{Foam}}^{\ast}} 
\newcommand{\freefoam}{\Z\boldsymbol{\mathrm{Foam}}^{\ast}} 
\newcommand{\freefoameq}{\Eq\boldsymbol{\mathrm{Foam}}^{\ast}} 
\newcommand{\freefoameqcl}{\Eq\boldsymbol{\mathrm{Foam}}^{\ast}(\mathrm{cl})}
\newcommand{\foam}[1]{\Eq\boldsymbol{\mathrm{Foam}}}
\newcommand{\foamcl}[1]{\Eq\boldsymbol{\mathrm{Foam}}(\mathrm{cl})} 
 
\newcommand{\foamS}[1]{\alphS\boldsymbol{\mathrm{Foam}}} 
\newcommand{\FoamSH}{\alphS\boldsymbol{\overline{\mathrm{Foam}}}}  
\newcommand{\Foam}[1]{\foam{#1}} 
\newcommand{\FoamS}[1]{\foamS{#1}} 
\newcommand{\fcd}{phase diagram }
\newcommand{\fcds}{phase diagrams }
\newcommand{\spec}{\mathrm{sp}^{\alphS}}

\newcommand{\slnn}[1]{\mathfrak{sl}_{#1}}
\newcommand{\id}{\mathrm{id}}
\newcommand{\modu}[1]{#1\text{-}\mathrm{mod}}
\newcommand{\Hom}{\mathrm{Hom}}
\newcommand{\End}{\mathrm{End}}
\newcommand{\eval}{\mathrm{ev}}
\newcommand{\grass}{\mathrm{Gr}}

\theoremstyle{definition}
\newtheorem{theoremm}{Theorem}[section]

\declaretheorem[style=definition,name=Theorem,qed=$\square$,numberlike=theoremm]{theorem}
\declaretheorem[style=definition,name=Lemma,qed=$\square$,numberlike=theoremm]{lemma}
\declaretheorem[style=definition,name=Proposition,qed=$\square$,numberlike=theoremm]{proposition}

\declaretheorem[style=definition,name=Corollary,qed=$\blacksquare$,numberlike=theoremm]{corollary}

\declaretheorem[style=definition,name=Lemma,qed=$\blacksquare$,numberlike=theoremm]{lemmab}

\declaretheorem[style=definition,name=Example,qed=$\blacktriangle$,numberlike=theorem]{example}
\declaretheorem[style=definition,name=Definition,qed=$\blacktriangle$,numberlike=theorem]{definition}
\declaretheorem[style=definition,name=Remark,qed=$\blacktriangle$,numberlike=theorem]{remark}
\declaretheorem[style=definition,name=Convention,qed=$\blacktriangle$,numberlike=theorem]{convention}

\declaretheorem[style=definition,name=Definition,numberlike=theorem]{definitionn}
\declaretheorem[style=definition,name=Example,numberlike=theorem]{examplen}
\declaretheorem[style=definition,name=Remark,numberlike=theorem]{remarkn}

\declaretheorem[style=definition,name=Convention,numberlike=theorem]{conventionn}
\declaretheorem[style=definition,name=Lemma,numberlike=theoremm]{lemman}

\declaretheorem[style=definition,name=Theorem,numberlike=theoremm]{theoremn}

\declaretheorem[style=definition,name=Remark,qed=$\blacktriangle$,numbered=no]{remarkintro}

\usepackage{chngcntr}
\counterwithout{figure}{section}

\def\notation#1#2#3{\rlap{\hyperref[#1]{{\color{myblue}#2}}}\hspace*{8.2mm} \hbox to 47mm{#3\hfill}}

\newcommand{\makeqed}{\hfill\ensuremath{\square}}
\newcommand{\makeqedtri}{\hfill\ensuremath{\blacktriangle}}
\newcommand{\qedmake}{\hfill\ensuremath{\blacksquare}}

\definecolor{somecolor}{rgb}{0,0.8,0}
\definecolor{somecolor2}{rgb}{0,0.6,0.6}
\definecolor{somecolor3}{rgb}{0.4,0,0.6}
\definecolor{somecolor3b}{rgb}{0.65,0,0.85}
\definecolor{somecolor3c}{rgb}{0.85,0.65,0}
\definecolor{somecolor3d}{rgb}{0,0.85,0.65}
\definecolor{myblue}{rgb}{0,0,0.9}

\definecolor{mydarkblue}{RGB}{10,10,170}
\definecolor{orchid}{RGB}{143,40,194}
\definecolor{lava}{RGB}{207,16,32}

\newcommand{\comm}[1]{}
\allowdisplaybreaks

\newcommand{\purplebox}{\,\tikz[baseline=-.05,scale=0.25]{\draw[somecolor3b,fill=somecolor3b] (0,0) rectangle (1,1);}\,}
\newcommand{\goldenbox}{\,\tikz[baseline=-.05,scale=0.25]{\draw[somecolor3c,fill=somecolor3c] (0,0) rectangle (1,1);}\,}
\newcommand{\cyanbox}{\,\tikz[baseline=-.05,scale=0.25]{\draw[somecolor3d,fill=somecolor3d] (0,0) rectangle (1,1);}\,}

\numberwithin{equation}{section}
\usepackage[final]{hyperref, bookmark}
\usepackage{cleveref}
\hypersetup{
    pdftoolbar=true,        
    pdfmenubar=true,        
    pdffitwindow=false,     
    pdfstartview={FitH},    
    pdftitle={Functoriality of colored link homologies},    
    pdfauthor={Michael Ehrig, Daniel Tubbenhauer and Paul Wedrich},
    pdfsubject={},   
    pdfcreator={Michael Ehrig, Daniel Tubbenhauer and Paul Wedrich},   
    pdfproducer={Michael Ehrig, Daniel Tubbenhauer and Paul Wedrich}, 
    pdfkeywords={}, 
    pdfnewwindow=true,      
    colorlinks=true,       
    linkcolor=mydarkblue,          
    citecolor=teal,        
    filecolor=magenta,      
    urlcolor=orchid,          
    linkbordercolor=lava,
    citebordercolor=teal,
    urlbordercolor=orchid,  
    linktocpage=true
}
\let\fullref\autoref
\def\makeautorefname#1#2{\expandafter\def\csname#1autorefname\endcsname{#2}}
\makeautorefname{equation}{Equation}%
\makeautorefname{footnote}{footnote}%
\makeautorefname{item}{item}%
\makeautorefname{figure}{Figure}%
\makeautorefname{table}{Table}%
\makeautorefname{part}{Part}%
\makeautorefname{appendix}{Appendix}%
\makeautorefname{chapter}{Chapter}%
\makeautorefname{section}{Section}%
\makeautorefname{subsection}{Section}%
\makeautorefname{subsubsection}{Section}%
\makeautorefname{paragraph}{Paragraph}%
\makeautorefname{subparagraph}{Paragraph}%
\makeautorefname{theorem}{Theorem}%
\makeautorefname{theo}{Theorem}%
\makeautorefname{thm}{Theorem}%
\makeautorefname{addendum}{Addendum}%
\makeautorefname{addend}{Addendum}%
\makeautorefname{add}{Addendum}%
\makeautorefname{maintheorem}{Main theorem}%
\makeautorefname{mainthm}{Main theorem}%
\makeautorefname{corollary}{Corollary}%
\makeautorefname{corol}{Corollary}%
\makeautorefname{coro}{Corollary}%
\makeautorefname{cor}{Corollary}%
\makeautorefname{lemma}{Lemma}%
\makeautorefname{lemm}{Lemma}%
\makeautorefname{lem}{Lemma}%
\makeautorefname{sublemma}{Sublemma}%
\makeautorefname{sublem}{Sublemma}%
\makeautorefname{subl}{Sublemma}%
\makeautorefname{proposition}{Proposition}%
\makeautorefname{proposit}{Proposition}%
\makeautorefname{propos}{Proposition}%
\makeautorefname{propo}{Proposition}%
\makeautorefname{prop}{Proposition}%
\makeautorefname{property}{Property}
\makeautorefname{proper}{Property}
\makeautorefname{scholium}{Scholium}%
\makeautorefname{step}{Step}%
\makeautorefname{conjecture}{Conjecture}%
\makeautorefname{conject}{Conjecture}%
\makeautorefname{conj}{Conjecture}%
\makeautorefname{question}{Question}
\makeautorefname{questn}{Question}
\makeautorefname{quest}{Question}
\makeautorefname{ques}{Question}
\makeautorefname{qn}{Question}
\makeautorefname{definition}{Definition}%
\makeautorefname{defin}{Definition}%
\makeautorefname{defi}{Definition}%
\makeautorefname{def}{Definition}%
\makeautorefname{dfn}{Definition}%
\makeautorefname{notation}{Notation}
\makeautorefname{nota}{Notation}
\makeautorefname{notn}{Notation}
\makeautorefname{remark}{Remark}%
\makeautorefname{rema}{Remark}%
\makeautorefname{rem}{Remark}%
\makeautorefname{rmk}{Remark}%
\makeautorefname{rk}{Remark}%
\makeautorefname{remarks}{Remarks}%
\makeautorefname{rems}{Remarks}%
\makeautorefname{rmks}{Remarks}%
\makeautorefname{rks}{Remarks}%
\makeautorefname{example}{Example}%
\makeautorefname{examp}{Example}%
\makeautorefname{exmp}{Example}%
\makeautorefname{exam}{Example}%
\makeautorefname{exa}{Example}%
\makeautorefname{algorithm}{Algorith}%
\makeautorefname{algo}{Algorith}%
\makeautorefname{alg}{Algorith}%
\makeautorefname{axiom}{Axiom}%
\makeautorefname{axi}{Axiom}%
\makeautorefname{ax}{Axiom}%
\makeautorefname{case}{Case}%
\makeautorefname{claim}{Claim}%
\makeautorefname{clm}{Claim}%
\makeautorefname{assumption}{Assumption}%
\makeautorefname{assumpt}{Assumption}%
\makeautorefname{conclusion}{Conclusion}%
\makeautorefname{concl}{Conclusion}%
\makeautorefname{conc}{Conclusion}%
\makeautorefname{condition}{Condition}%
\makeautorefname{condit}{Condition}%
\makeautorefname{cond}{Condition}%
\makeautorefname{construction}{Construction}%
\makeautorefname{construct}{Construction}%
\makeautorefname{const}{Construction}%
\makeautorefname{cons}{Construction}%
\makeautorefname{criterion}{Criterion}%
\makeautorefname{criter}{Criterion}%
\makeautorefname{crit}{Criterion}%
\makeautorefname{exercise}{Exercise}%
\makeautorefname{exer}{Exercise}%
\makeautorefname{exe}{Exercise}%
\makeautorefname{problem}{Problem}%
\makeautorefname{problm}{Problem}%
\makeautorefname{probm}{Problem}%
\makeautorefname{prob}{Problem}%
\makeautorefname{solution}{Solution}%
\makeautorefname{soln}{Solution}%
\makeautorefname{sol}{Solution}%
\makeautorefname{summary}{Summary}%
\makeautorefname{summ}{Summary}%
\makeautorefname{sum}{Summary}%
\makeautorefname{operation}{Operation}%
\makeautorefname{oper}{Operation}%
\makeautorefname{observation}{Observation}%
\makeautorefname{observn}{Observation}%
\makeautorefname{obser}{Observation}%
\makeautorefname{obs}{Observation}%
\makeautorefname{ob}{Observation}%
\makeautorefname{convention}{Convention}%
\makeautorefname{convent}{Convention}%
\makeautorefname{conv}{Convention}%
\makeautorefname{cvn}{Convention}%
\makeautorefname{warning}{Warning}%
\makeautorefname{warn}{Warning}%
\makeautorefname{note}{Note}%
\makeautorefname{fact}{Fact}%
\begin{document}
\vbadness=10001
\hbadness=10001
\title[Functoriality of colored link homologies]{Functoriality of colored link homologies}
\author{Michael Ehrig}
\address{M.E.: School of Mathematics \& Statistics, Carslaw Building, University of Sydney, NSW 2006, Australia}
\email{michael.ehrig@sydney.edu.au}

\author{Daniel Tubbenhauer}
\address{D.T.: Mathematisches Institut, Universit\"at Bonn, Endenicher Allee 60, Room 1.003, 53115 Bonn, Germany. \newline
Current Address: Institut f\"ur Mathematik, Universit\"at Z\"urich, Winterthurerstrasse 190, Campus Irchel, Office Y27J32, CH-8057 Z\"urich, Switzerland, \href{www.dtubbenhauer.com}{www.dtubbenhauer.com}}
\email{daniel.tubbenhauer@math.uzh.ch}

\author{Paul Wedrich}
\address{P.W.: Imperial College London, Department of Mathematics, 6M50 Huxley Building, South Kensington Campus London SW7 2AZ, United Kingdom. \newline Current Address: Mathematical Sciences Institute, The Australian National University, John Dedman Building, Union Lane, Canberra  ACT  2601, Australia, \href{http://paul.wedrich.at}{paul.wedrich.at}}
\email{p.wedrich@gmail.com}

\subjclass{57M25, 81T45}

\begin{abstract}
We prove that
the bigraded, colored Khovanov--Rozansky 
type A link and tangle invariants are functorial 
with respect to link and tangle cobordisms. 
\end{abstract}
\maketitle
\tableofcontents
\renewcommand{\theequation}{\thesection-\arabic{equation}}
%
\section{Introduction}\label{sec:intro}
\addtocontents{toc}{\protect\setcounter{tocdepth}{1}}

Building on Khovanov's categorification of 
the Jones polynomial \cite{Kho}, Khovanov--Rozansky \cite{KR} introduced a
link homology theory categorifying the 
$\slnn{N}$ Reshetikhin--Turaev 
invariant. Their homology theory associates 
bigraded vector spaces 
to link diagrams, two of which are isomorphic whenever the diagrams 
differ only by Reidemeister moves. 
Hence, their homological link invariant 
takes values in isomorphism classes of bigraded vector spaces.

\subsection*{The first question\texorpdfstring{\nopunct}{}}\! 
posed by 
this construction is whether there is a natural 
choice of Reidemeister move isomorphisms, such that 
any isotopy of links in $\mathbb{R}^3$ gives rise to 
an explicit isomorphism between the Khovanov--Rozansky 
invariants, which only depends on the isotopy class of 
the isotopy. A positive answer to this question provides 
a functor:
\begingroup
\renewcommand*{\arraystretch}{1.35}
\[
\begin{Bmatrix}
\text{link embeddings in } \R^3
\\
\text{isotopies modulo isotopy}
\end{Bmatrix}
\longrightarrow 
\begin{Bmatrix} 
\text{bigraded vector spaces}
\\ 
\text{isomorphisms} 
\end{Bmatrix}
\]
\endgroup

\subsection*{The second question,\texorpdfstring{\nopunct}{}}\! building on the 
first, is whether this functor can be extended to a functor:
\begingroup
\renewcommand*{\arraystretch}{1.35}
\[
\begin{Bmatrix}
\text{link embeddings in } \R^3
\\
\text{link cobordisms in } 
\R^3\times [0,1] \text{ modulo isotopy}
\end{Bmatrix}
\longrightarrow 
\begin{Bmatrix} 
\text{bigraded vector spaces}
\\ 
\text{homogeneous linear maps} 
\end{Bmatrix}
\]
\endgroup

\subsection*{The goal of this paper\texorpdfstring{\nopunct}{}}\! is to answer 
both questions affirmatively, i.e. to prove 
the functoriality of Khovanov--Rozansky link homologies for $N\geq 2$ under link cobordisms. 

This result is new to the best of our knowledge, except in low-rank cases. For $N=2$, functorial theories were obtained by Caprau \cite{Cap}, Clark--Morrison--Walker \cite{CMW} and Blanchet \cite{Bla} after Khovanov's original theory turned out to have a sign ambiguity, see Jacobsson \cite{Jac} and Bar-Natan \cite{BN1}. However, with a bit more care, functoriality can also be achieved in the original construction, see Vogel \cite{Vog} and joint work with Stroppel \cite{EST2,EST1}. Functoriality for $N=3$ was proved by Clark \cite{Cla}.

We prove the general functoriality statement in a framework that is different to and more general than Khovanov--Rozansky's 
construction in \cite{KR}, as we will now explain.

\subsection{Foams}
Khovanov--Rozansky link homology theories 
can be defined, or at least described, in many different languages, ranging from those of algebraic geometry, (higher) 
representation theory and symplectic geometry to those of 
various incarnations of string theory. The construction most 
suitable for this paper is combinatorial and uses a graphical 
calculus of \textit{webs} and \textit{foams}.

Foams play precisely the same role for Khovanov--Rozansky's link homologies 
that Bar-Natan's cobordisms \cite{BN1} play 
for the original Khovanov homology. They can be seen---in a straightforward way---as a categorification of the Murakami--Ohtsuki--Yamada 
state-sum model \cite{MOY} for the 
Reshetikhin--Turaev link invariants of type A, or, equivalently, the graphical calculus for the 
corresponding representation category, see Cautis--Kamnitzer--Morrison \cite{CKM}. 

Foams were first used by Khovanov for the 
definition of an $\slnn{3}$ homology \cite{Kho3}.
Khovanov--Rozansky \cite{KR} made the observation 
that foams should be useful for understanding higher rank link homologies as well, and they extended foams accordingly \cite{KR3}. These were then used by  
Mackaay--Sto{\v{s}}i{\'c}--Vaz \cite{MSV} in the construction of the corresponding link homologies.

Later on, Queffelec--Rose \cite{QR}, building 
on joint work with Lauda \cite{LQR} and earlier work 
of Mackaay and his coauthors \cite{Ma1,MPT} on low-rank 
cases, placed foams in the context of an instance of 
categorical skew Howe duality. This led to a better 
understanding of foams as well as a comparison of foam-based 
link homologies with those obtained via other constructions, 
see the introduction of Mackaay--Webster \cite{MW} for a summary. 

A disadvantage of the approach of \cite{QR} is that their 2-categories of foams only describe a certain part of the web and foam calculus envisioned by Khovanov--Rozansky and Mackaay--Sto{\v{s}}i{\'c}--Vaz. However, Robert--Wagner \cite{RoW} have closed this gap while maintaining backward compatibility with \cite{QR}, and we 
will use their approach after explaining its 
essential features in \fullref{subsec:main-player}.

\subsection{Tangles and canopolises}
Proving the functoriality of a link homology theory 
essentially amounts to checking coherence relations between 
various ways of composing maps associated to Reidemeister 
moves and other basic link cobordisms that represent isotopic 
cobordisms. These relations are the \textit{movie moves} as presented 
by Carter--Saito \cite[Chapter 2]{CSbook}. As 
demonstrated by Bar-Natan \cite{BN1}, it is extremely useful 
to be able to perform the required computations locally, 
i.e. in a small portion of the link diagram. 

Fortunately, the construction of link homologies via 
foams immediately extends to the case of tangles and the 
invariants have essentially tautological planar composition 
properties. This is captured in Bar-Natan's notions of 
canopolises (a.k.a. planar algebras in categories) and 
canopolis morphisms, which we recall in \fullref{subsection:canopolis}. 

A key idea for our paper is Bar-Natan's 
insight \cite[Section 1.1.1]{BN1} that the 
canopolis framework allows a low-effort proof 
of the fact that Khovanov--Rozansky link homologies are 
functorial up to scalars. This amounts to a 
significant proof shortcut, as it then only 
remains to ensure that these scalars are equal to one. 
However, this is still a formidable challenge.

\subsection{Colors}
The Murakami--Ohtsuki--Yamada state-sum model in fact 
determines the $\slnn{N}$ Reshe\-tikhin--Turaev invariants 
of links whose components are colored by fundamental 
representations of quantum $\slnn{N}$, i.e. the quantum exterior powers of 
the usual vector representation. One categorical level up, 
foams immediately provide an analogous extension of Khovanov--Rozansky's 
original \textit{uncolored}---i.e. colored only 
with the vector representation---construction 
to the colored case. We will work in this 
generality and record the coloring of tangle components with exterior 
power representations by remembering only the exterior exponents and 
placing them as labels next to the respective strands. 

Colored Khovanov--Rozansky homologies have been 
first constructed by Wu \cite{Wu1} and Yonezawa \cite{Yon} 
in a technically demanding generalization 
of the approach in \cite{KR}. That their homological
invariants can be recovered via foams was 
proven in full generality by Queffelec--Rose \cite{QR}.

\subsection{Deformations}
Arguably the most important tool available for 
studying Khovanov homology is Lee's deformation \cite{Lee}. 
While producing uninteresting link invariants, this deformation naturally 
appears as the limit of a spectral sequence 
starting at Khovanov homology, which has been 
used to extract hidden topological information about knots, 
see e.g. Rasmussen \cite{Ras}.

A key idea for our paper is 
Blanchet's use of a Lee-type deformation 
for proving the functoriality of a modified 
version of Khovanov homology. In \cite{Bla} he first 
proved functoriality up to scalars along Bar-Natan's strategy and then computed these scalars---and showed them to be equal to one---by working in the much 
simpler setting of the deformation. This approach draws 
on the combinatorial interpretation of Lee's deformation 
provided by Bar-Natan--Morrison \cite{BNM}. 

The Khovanov--Rozansky link homologies of higher rank are subject to an even richer deformation theory, the exploration of which started with the work of Gornik \cite{Gor} and Mackaay--Vaz \cite{MV2} in the uncolored case, and Wu \cite{Wu2} in the colored case. These deformations were classified in joint work with Rose \cite{RW} and used in the proof of a family of physical conjectures about link homologies \cite{Wed3} by the third named author.

In this paper we draw on the deformed foam technology 
as developed in \cite{RW} to provide a functoriality 
proof for colored Khovanov--Rozansky invariants following Blanchet's strategy.

\subsection{Equivariance}
A central feature in Blanchet's argument 
is that, if functoriality holds up to scalars, then 
these scalars can be computed in a Lee-type deformation. 
To avoid arguing that these scalars are preserved along 
Lee-type spectral sequences, we will let our invariants take 
values in a homotopy category of chain complexes and not
take homology right away. One advantage of this is that the 
undeformed, colored link invariant, as well as all its deformations, can be 
obtained as specializations of a unifying \textit{equivariant} theory. 
We then prove that the equivariant theory is functorial up to scalars 
and that all its specializations inherit this property with the same scalars. 
It then only remains to compute these scalars in the 
Lee-type deformation, which is significantly simpler.

The first equivariant, or \textit{universal}, link 
homology can be traced back to Bar-Natan \cite{BN1} and 
Khovanov \cite{Kho2} and it encapsulated 
both Khovanov homology and Lee's deformation. The case $N=3$ 
was treated by Mackaay--Vaz \cite{MV2}. The extension to 
higher rank is due to Krasner \cite{Kra} in the uncolored 
case, and due to Wu \cite{Wu2} in the colored case. The 
adjective \textit{equivariant} refers to the fact that these link 
homologies associate $\mathrm{GL}_N$-equivariant 
cohomology rings of Grassmannians to colored unknots. 

Fortunately, the foam technology of Robert--Wagner \cite{RoW} 
is already formulated in the necessary generality to be compatible 
with the equivariant, colored Khovanov--Rozansky link homologies. 

\subsection{Integrality}
The foam-based link invariants of Mackaay--Sto{\v{s}}i{\'c}--Vaz \cite{MSV} can be defined 
integrally and they give rise to integral versions of Khovanov--Rozansky homologies. 
These invariants take values in bigraded abelian groups rather than 
in bigraded vector spaces. Similarly, 
in the colored case, Queffelec--Rose 
\cite[Proposition 4.10]{QR} have observed that their 
foam-based construction can also be defined over the 
integers. The same is true for the construction using Robert--Wagner's foams \cite{RoW}.

We have decided to present the results of this paper using the ground ring $\C$. This is for notational convenience as it allows us to treat the canopolis of equivariant foams and all its specializations in the same framework and using the same language. However, with minimal adjustments our proof of functoriality also works over $\Z$, and we will comment on these variations when necessary.

\subsection{Module structure}
The Khovanov--Rozansky homologies of links are modules over cohomology rings of Grassmannians that appear as the invariants of colored unknots. These actions have a natural interpretation using functoriality: Given a basepoint on a colored link, one can place a small unknot of matching color next to the basepoint. This has the effect of tensoring the link invariant with the corresponding cohomology ring. The action of this ring is then determined by the map associated to the cobordism that merges the unknot with the link at the basepoint.

These module structures carry additional information. For example, 
Hedden--Ni \cite{HN} have shown that they enable Khovanov homology to detect unlinks 
(it was later proven that the module structure is determined by the bigraded 
vector space structure in the case of the unlink \cite{BS})
, see also Wu \cite{Wu3} for a related result for the triply-graded homology. Module structures 
are also important for the comparison with Floer-theoretic link invariants, see 
e.g. Baldwin--Levine--Sarkar \cite{BLS}, and the 
construction of \textit{reduced}, colored Khovanov--Rozansky homologies \cite{Wed3}.

\subsection{Outlook}
We finish by commenting on some interesting 
topics that we will not investigate further in this paper.

First, it is necessary to emphasize that we prove functoriality 
under link cobordisms modulo isotopies in $\R^3\times [0,1]$. 
It is tempting to view Khovanov--Rozansky homology as an invariant 
of links in $S^3=\R^3\cup\{\infty\}$ and cobordisms in $S^3\times I$. 
Indeed, links generically miss the point $\infty$ and cobordisms between 
them generically miss $\infty\times [0,1]$. However, 
the same is no longer true for isotopies between link cobordisms.

As we have learned from Scott Morrison, proving the 
functoriality of Khovanov--Rozansky homology in $S^3$ 
would require checking only a single type of additional 
movie move, which 
however is non-local. Functoriality in $S^3$ is also 
the only missing ingredient for upgrading Khovanov--Rozansky 
link homologies to invariants of 4-manifolds with links in their 
boundary: it should allow the Khovanov--Rozansky 
invariants to be assembled into a \textit{disk-like 4-category} 
as in \cite{MoW}, which then gives rise to a $(4+\epsilon)$-dimensional TQFT.

A second question concerns the functoriality
properties of the extension of colored 
Khovanov--Rozansky link homologies to invariants of tangled webs, 
embedded in $\R^3$, under foams embedded in $\R^3\times [0,1]$ modulo 
isotopy. Tangled webs can naturally be accommodated in the
framework that we use in this paper, and foams embedded in 4-space can be encoded as movies 
of tangled webs. Isotopies of foams are represented by a finite collection of movie moves, see e.g. Carter \cite{Car}. We expect that the 
question of functoriality under foams can be 
investigated similar as in \fullref{sec:functorial}.

Next, Khovanov--Rozansky homologies 
can be extended to the case of colorings by other irreducible representations. Using the framework in the present 
paper, there are at least two distinct ways of doing this. One uses 
finite resolutions of these representations by fundamental representations 
and is analogous to Khovanov's construction for $N=2$ in \cite{Kho4}, 
see also Robert \cite{Rob} for the case $N=3$. In the other approach, the invariants for other colors are computed by inserting 
\textit{infinite twists} into fundamentally colored cables of the original 
link, see Rozansky \cite{Roz} for the case of $N=2$ and Cautis \cite{Cau2} 
for the general case. For both constructions it is an interesting question 
whether they satisfy functoriality properties and admit module structures. 
For another construction of symmetric link homologies, see Robert--Wagner \cite{RoW2}.

Finally, it is tempting to speculate about 
foam-based constructions of link homology theories 
that categorify the Reshetikhin--Turaev invariants 
outside of type A. Our wish list for such constructions 
includes that they should allow deformations along splittings 
of the corresponding Dynkin diagram, as it is the case in 
type A \cite{RW}. This might eventually help to prove functoriality 
properties for link homologies in other types, but more importantly, 
it gives hints how to construct the necessary foam theories. 

Very preliminary results in this direction have been obtained in \cite{ETW} 
where a kind of type D foams were constructed using the foams that 
appear in the present paper for $N=2$. However, it is not clear whether 
these foams can be used to define type D link homologies. Even 
one categorical level down, there are different web calculi 
outside of type A \cite{KUP,MO,ST}, not all of which 
are compatible with Reshetikhin--Turaev invariants.

\subsection{Structure of the paper} 
In \fullref{sec:glnfoams} we introduce
the canopolis of $\mathrm{GL}_N$-equivariant foams by first 
defining a free canopolis of foams in \fullref{subsection:canopolis} 
and then imposing additional relations in \fullref{subsec:main-player}. 
\makeautorefname{subsection}{Sections}
In \fullref{subsec:equi-rel} and \ref{subsec:special} we collect 
a number of foam relations, which we need in the proof of functoriality. 

\makeautorefname{subsection}{Section}

In \fullref{sec:foam} we recall the construction of the 
categorical tangle invariant and study the chain maps induced 
by Reidemeister moves. 

Finally, the functoriality proof in \fullref{sec:functorial} 
is split into the up-to-scalar check in \fullref{subsec:uptounits} 
followed by the computation of the scalars in \fullref{subsec:checkmm}.

\begin{remarkintro}\label{remark:colors-in-this-paper}
For the figures in this paper we have chosen 
colors that can also be distinguished in grayscale print. 
An important role in calculations will be played by 
the colors purple $\purplebox$, golden $\goldenbox$ and cyan $\cyanbox$, 
which will appear differently shaded.
\end{remarkintro}

\addtocontents{toc}{\protect\setcounter{tocdepth}{2}}

\noindent \textbf{Acknowledgements:} 
We are indebted to Christian Blanchet for proposing this project to P.W. during his thesis defense, to Hoel Queffelec for encouraging us to pursue the question of functoriality over the integers, and to Louis-Hadrien Robert and Emmanuel Wagner for sharing and explaining a draft of their paper about a combinatorial formula for evaluating closed foams. The latter two exchanges took place at the Isaac Newton Institute for Mathematical Sciences during the programme ``Homology theories in low dimensional topology'', which was supported by the UK Engineering and Physical Sciences Research Council [Grant Number EP/K032208/1].

We would also like to thank Nils Carqueville, Marco Mackaay and Pedro Vaz for their explanations of signs appearing 
in $2$-categories of matrix factorizations, as well as Mikhail Khovanov, Skype, Catharina Stroppel and Joshua Sussan for many helpful conversation 
about foams and categorification, and 
an anonymous referee and Louis-Hadrien Robert for comments 
on a draft of this paper. D.T. also thanks a still nameless toilet 
paper roll for explaining the concept of 
a canopolis to him.

\noindent
M.E. was partially supported by the Australian Research Council Grant DP150103431.
P.W. was supported by a Fortbildungsstipendium at 
the Max Planck Institute for Mathematics in Bonn and 
by the Leverhulme Research Grant RP2013-K-017 to Dorothy 
Buck, Imperial College London.

\section{\texorpdfstring{$\mathrm{GL}_N$}{GLN}-equivariant foams}\label{sec:glnfoams}
\subsection{Symmetric polynomials}\label{subsec:sym-poly}
We start by briefly recalling several notions regarding symmetric polynomials which we will need in this paper. 
Throughout, we fix $N\in\Z_{\geq 1}$.

Let $\A$ be an alphabet of size $|\A|=a$.
We denote by $\sym(\A)$ the $\C$-algebra of 
symmetric polynomials in the alphabet $\A$. Similarly, 
for alphabets $\A_i$, we denote by 
$\sym(\A_1|\cdots|\A_r)$ the $\C$-algebra of 
polynomials which are symmetric in each alphabet 
separately. We view these symmetric polynomial rings as being 
graded by assigning each variable the degree $2$.

We denote the $j^{\mathrm{th}}$ elementary symmetric polynomials in $\A$ by $\elsym{j}(\A)$, 
and the $j^{\mathrm{th}}$ complete symmetric polynomial in $\A$ by $\fullsym{j}(\A)$. Recall that 
the first $a$ of either give an algebraically independent set of 
generators for $\sym(\A)$, i.e.
\[
\sym(\A) 
\cong 
\C[\elsym{1}(\A),\dots,\elsym{a}(\A)]
\cong 
\C[\fullsym{1}(\A),\dots,\fullsym{a}(\A)],
\]
and that $\elsym{j}(\A)=0$ for $j>a$. 

\begin{definitionn}\label{definition:schurdiff} 
Let $\A$ and $\B$ be two disjoint 
alphabets of sizes $a$ and $b$. The complete 
symmetric polynomials in the difference of 
these two alphabets are elements of 
$\sym(\A|\B)$ given by the generating 
function:
\[
\sum_{k\geq 0}\,\fullsym{k}(\A-\B) x^k 
= 
\frac{\prod_{B\in \B}(1-B x)}{\prod_{A\in \A}(1-A x)}.
\] 
The complete symmetric polynomials in $\A$ 
are recovered in the special case 
$\B=\emptyset$, or 
under the homomorphism setting the 
variables in $\B$ to zero: 
\[
\fullsym{k}(\A)=\fullsym{k}(\A-\emptyset)=
\fullsym{k}(\A-\B)|_{\B\mapsto 0}.
\]

\end{definitionn}

\begin{example}\label{example:polynomials}
For $\A=\{A_1,A_2\}$ and $\B=\{B\}$ we get the following generating function:
\[
\frac{(1-Bx)}{(1-A_1x)(1-A_2x)}
=
\underbrace{1}_{\fullsym{0}(\A-\B)}
+ \underbrace{(A_1 + A_2 - B)}_{\fullsym{1}(\A-\B)} x
+ \underbrace{(A_1^2 + A_1 A_2 +A_2^2 -(A_1 + A_2)B)}_{\fullsym{2}(\A-\B)} x^2 
+ \mathfrak{O}(x^3)
\]
For $B=0$ we get the generating function for the complete symmetric polynomials in $\{A_1,A_2\}$.
\end{example}

\begin{conventionn}\label{convention:partitions}
Let $\partitionall$ denote the set of all integer partitions. For $a,b\in\Z_{\geq 0}$ 
we write $\partitions{a}$ for the set of partitions 
with at most $a$ parts, and $\partition{a}{b}$ for the 
set of partitions with at most $a$ parts, all of which are of size at most $b$.

Partitions $\alpha=(\alpha_1,\dots,\alpha_a)$ can be 
identified with Young diagrams with $\alpha_i$ specifying 
the numbers of boxes in the $i^{\mathrm{th}}$ row. Using this 
identification, let the total number of boxes be denoted 
by $|\alpha|$ and let $\transpose{\alpha}$ denote the 
Young diagram obtained by 
reflecting $\alpha$ along its main diagonal. If it is 
understood that $\alpha\in \partition{a}{b}$, we write 
$\complementnorm{\alpha}\in\partition{a}{b}$ for the Young 
diagram of the complement of $\alpha$, whose rows are given by $\complementnorm{\alpha}_{a+1-j} = 
b -\alpha_j$ and $\Ybox{a,b}=\complementnorm{\emptyset}$ for the full 
box Young diagram. Further, we write 
$\complementtrans{\alpha} = \transpose{\complementnorm{\alpha}}$ 
for the transpose of the complement of $\alpha$. 

Here is an 
illustrative example for $a=3, b=4$:
\[
\Ybox{3,4}=
\xy
(0,0)*{
\begin{Young}
 &  &  &  \cr
 &  &  &  \cr
 &  &  &  \cr
\end{Young}
} 
\endxy
\;,\quad
\alpha=
\xy
(0,0)*{
\begin{Young}
& & &\cr
& \cr
\cr
\end{Young}
}
\endxy
\;,\quad
\complementnorm{\alpha}=
\xy
(0,0)*{
\begin{Young}
& & \cr
&\cr
\end{Young}
}
\endxy
\;,\quad
\complementtrans{\alpha}=
\xy
(0,0)*{
\begin{Young}
& \cr
& \cr
\cr
\end{Young}
}
\endxy
\hspace{1.55cm}
\raisebox{-.55cm}{\makeqedtri}
\hspace{-1.55cm}
\]
\end{conventionn}

\begin{definitionn}\label{definition:schur-poly} 
Let $\alpha=(\alpha_1,\dots,\alpha_a)\in \partitions{a}$ be 
a partition. Then the Schur polynomial corresponding to 
$\alpha$ in the difference of alphabets $\A-\B$ is given by:
\begin{equation*}
\schur{\alpha}(\A-\B) = \det((\fullsym{\alpha_i+j-i}(\A-\B))_{1\leq i,j\leq a}).
\end{equation*}
The Schur polynomials in $\A$ 
are recovered in the special case $\B=\emptyset$ or 
under the homomorphism setting the variables in $\B$ to zero: 
\[
\schur{\alpha}(\A)=\schur{\alpha}(\A-\emptyset)=\schur{\alpha}(\A-\B)|_{\B\mapsto 0}.
\hspace{4.55cm}
\raisebox{-.05cm}{\makeqedtri}
\hspace{-4.55cm}
\]
\end{definitionn}

Recall that the Schur 
polynomials $\schur{\alpha}$ for $\alpha\in\partitions{a}$ 
form a basis for the $\C$-algebra $\sym(\A)$ 
with structure constants for the multiplication given by 
the \textit{Lit\-tle\-wood-Ri\-chard\-son coefficients} $\LR{\alpha}{\beta}{\gamma}$, 
see e.g. \cite[Section I.5]{MD}. This means for 
$\alpha, \beta\in \partitions{a}$ that
\[
\schur{\alpha}\schur{\beta}=
{\textstyle\sum_{\gamma\in\partitions{a}}}\,
\LR{\alpha}{\beta}{\gamma}
\schur{\gamma}.
\]

\begin{example}\label{example:unknot}(\cite[Section 2.3]{Wu2}, and also \cite{Las}.) 
Let $\grass_a$ denote the Grassmannian of 
$a$-dimensional subspaces in $\C^N$. This carries 
an action of $\mathrm{GL}_N$ and its $\mathrm{GL}_N$-equivariant 
cohomology can be presented as follows, see e.g. \cite[Lectures 6 and 7]{Ful}:
\begin{equation*}
H^*_{\mathrm{GL}_N}(\grass_a)
\cong
\frac{
\sym(\A|\Eq)}
{
\langle h_{N-a+i}(\A-\Eq)\mid i> 0
\rangle
}.
\end{equation*}
In fact, $H^*_{\mathrm{GL}_N}(\grass_a)$ is 
a graded, free module of rank $\binom{N}{a}$ 
over the equivariant cohomology of a point $H^*_{\mathrm{GL}_N}(pt)$, which is isomorphic to the symmetric polynomial ring $\sym(\Eq)$ in 
an alphabet $\Eq$ of size $N$. Moreover, $H^*_{\mathrm{GL}_N}(\grass_a)$ has a 
homogeneous bases over $\sym(\Eq)$ given by the Schur polynomials 
$\schur{\alpha}(\A)$ (or, alternatively, $\schur{\alpha}(\A-\Eq)$) 
with $\alpha\in\partition{a}{N-a}$.  

Let $\trace_a\colon H^*_{\mathrm{GL}_N}(\grass_a) \to \sym(\Eq)$ 
denote the $\sym(\Eq)$-linear projection onto the $\sym(\Eq)$-span of
the Schur polynomial corresponding to $\Ybox{a,N-a}(\A)$, 
adjusted by a sign that we will explain in 
\fullref{remark:weird-signs}.
Explicitly, this \textit{trace} satisfies
\begin{gather}\label{eq:trace-sign}
\trace_a(\schur{\alpha}(\A)\schur{\beta}(\A - \Eq)) = (-1)^{\binom{a}{2}} \delta_{\alpha,\complementnorm{\beta}},
\quad
\text{ for any }\alpha,\beta\in\partition{a}{N-a}.
\end{gather}
This means that it defines a non-degenerate $\sym(\Eq)$-bilinear form and, 
thus, a Frobenius algebra structure on $H^*_{\mathrm{GL}_N}(\grass_a)$. 
The corresponding (Poincar\'{e}) duality maps a basis 
element $\schur{\alpha}(\A)$ to 
$(-1)^{\binom{a}{2}} \schur{\complementnorm{\alpha}}(\A-\Eq)$. 
\end{example}

\begin{examplen}(\cite[Section 2.3]{Wu2}, and also \cite{Las}.)\label{example:digon} 
Let again $\A$ and $\B$ denote alphabets of size $a$ and $b$. Note that $\sym(\A|\B)$ is a free 
$\sym(\A\cup \B)$-module of rank $\binom{a{+}b}{a}$ 
with homogeneous bases given by the 
Schur polynomials $\schur{\alpha}(\A)$ (or, alternatively, 
$\schur{\alpha}(\A-\B))$ with $\alpha\in\partition{a}{b}$. 

Let $\zeta\colon \sym(\A|\B)\to\sym(\A\cup \B)$ 
denote the $\sym(\A\cup \B)$-linear projection 
onto the $\sym(\A\cup \B)$-span of the Schur 
polynomial corresponding to $\Ybox{a,b}(\A)$. This map, 
which is known as the Sylvester operator, 
is of degree $-2 a b$ and satisfies
\[
\zeta(\schur{\alpha}(\A)\schur{\beta}(\B)) 
= (-1)^{|\complementtrans{\alpha}|} \delta_{\alpha,\complementtrans{\beta}},
\quad 
\text{ for any }\alpha\in \partition{a}{b}\text{ and }\beta\in \partition{b}{a}.
\hspace{2.18cm}
\raisebox{-.00cm}{\makeqedtri}
\hspace{-2.18cm}
\] 
\end{examplen}

\subsection{Canopolises of webs and foams}\label{subsection:canopolis}
One main toolkit used in this paper is the canopolis formalism introduced 
by Bar-Natan \cite[Section 8.2]{BN1}, which can be 
seen as a categorification of Jones' planar 
algebra formalism \cite{Jon}. 
We will now describe a canopolis of webs and foams between them.
Throughout this section, let $S$ be a compact, planar surface, 
i.e. a surface with a fixed embedding into $\R^2$.

\begin{definitionn}\label{definition:web}
A web in $S$ is a 
finite, oriented, trivalent graph, 
properly embedded in $S$, together with a labeling of edges by elements 
of $\Z_{>0}$ satisfying a flow condition
at every trivalent vertex. Vertices 
with incoming labels $a,b$ and outgoing 
label $a+b$ are called merge vertices, the others 
split vertices. We illustrate 
them, and a trivial web, in the case where $S$ is a disk:
\begin{gather}\label{eq:mergesplit}
\xy
(0,0)*{\includegraphics[scale=1.15]{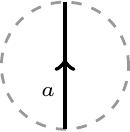}};
\endxy
\quad,\quad
\xy
(0,0)*{\includegraphics[scale=1.15]{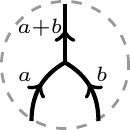}};
\endxy
\quad,\quad 
\xy
(0,0)*{\includegraphics[scale=1.15]{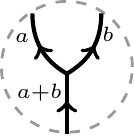}};
\endxy
\hspace{4.25cm}
\raisebox{-.5cm}{\makeqedtri}
\hspace{-4.25cm}
\end{gather}
\end{definitionn}

For the following definition, 
let $V$ and $W$ be two webs in $S$ 
with identical boundary data, i.e. 
they agree in a collar 
neighborhood of $\partial S$.

\begin{definition}\label{definition:foam}
A foam $F$ from $V$ to $W$ is a compact, two-dimensional 
CW-complex (the particular CW structure 
on $F$ is irrelevant in the following)
with finitely many cells, 
properly embedded in 
$S \times [0,1]$, such that
every interior point $x\in F$ 
has a neighborhood of one of the following 
three types, which are also illustrated in 
\fullref{fig:foamslocally}.
\smallskip
\begin{enumerate}[label=(\Roman*)]

\setlength\itemsep{.15cm}

\item A smoothly embedded two-dimensional manifold. The connected 
components of the set of such points are called the 
\textit{facets} of $F$. Each facet is required 
to be oriented, and we label it with 
an element of $\Z_{>0}$.

\item The letter $\mathrm{Y}$ (the union of three
distinct rays in $\R^2$, meeting in the origin) times $[0,1]$. The connected 
components of the set of such singular points are called the 
\textit{seams} of $F$. Every seam carries an orientation, which 
agrees with the orientation induced by two of the 
adjacents facets, say of label $a$ and $b$. Then we 
also require that the third adjacent facet is labeled 
by $a+b$ and that it induces the opposite orientation on the seam.

\item The cone on the one-skeleton of a tetrahedron. 
We call the cone points \textit{singular vertices}, and 
they are contained in the boundary of precisely four 
seams and six facets. 
\end{enumerate}

\begin{figure}[ht]
\begin{equation*}
\text{(I)}\colon
\xy
(0,0)*{\includegraphics[scale=1.15]{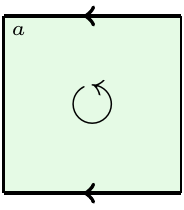}};
\endxy
\quad
,
\quad
\text{(II)}\colon
\xy
(0,0)*{\includegraphics[scale=1.15]{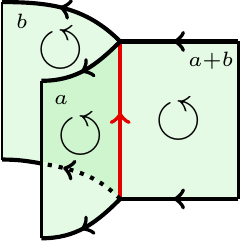}};
\endxy
\quad
,
\quad
\text{(III)}\colon
\xy
(0,0)*{\includegraphics[scale=1.15]{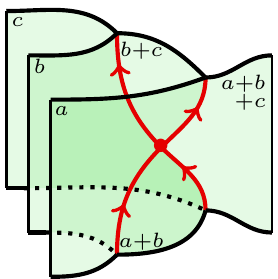}};
\endxy
\end{equation*}
\caption{Local models and orientation conventions.
}\label{fig:foamslocally}
\end{figure}

Furthermore, $F\colon V\to W$ is required to satisfy 
the following conditions:
\smallskip
\begin{itemize}

\setlength\itemsep{.15cm}

\item The (bottom) boundary of $F$ in $S\times\{0\}$
agrees with $V$ with matching 
labels and reversed orientations on the edges. 
The (top) boundary of $F$ in $S\times\{1\}$ agrees with $W$ with labels and 
induced orientations. In particular, singular 
seams are oriented down through merge 
vertices and up through split vertices in the boundary webs. 

\item $F$ restricted to $S \times [0,1]$ over 
a collar neighborhood of $\partial S$ is 
cylindrical, i.e. it agrees with the restriction 
of $V\times [0,1]$ to the same set.
\end{itemize}
\smallskip
We identify foams which differ by an 
ambient isotopy relative to the boundary in $S\times [0,1]$. Then the set of webs 
in $S$ with fixed boundary data assemble into a category 
with morphisms given by foams mapping from their bottom 
boundary web to their top boundary web. 
If $S$ is a disk, we call these \textit{disk categories}.
\end{definition}

\begin{remark}\label{remark:general-position}
Using classical Morse theory 
inductively on the skeletons of foams, one 
can show that one can isotope foams $F$ into generic 
position so that seams and facets have finitely 
many non-degenerate critical points for the height 
function. Then the slices $F\cap S\times \{z\}$ 
are webs for all but finitely many $z\in [0,1]$. 
See \fullref{rem:can-gen} for 
the local foams around these singularities.
\end{remark}

Webs in disks can be glued in a 
planar way reminiscent of Jones' planar algebras \cite{Jon}.
These \textit{planar algebra} operations are given by gluing disks with embedded webs into 
holed disks with embedded 
arcs, such that
web boundaries are glued in a way that 
respects the labelings and orientations. An example of a holed disk is
given in \fullref{fig:planalg}.

\begin{figure}[ht]
\begin{gather*}
\xy
(0,0)*{\includegraphics[scale=1.15]{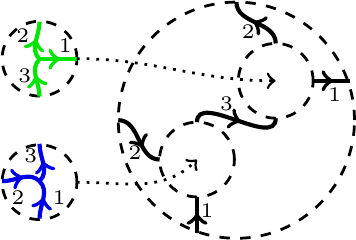}};
\endxy
\quad \longrightarrow \quad
\xy
(0,0)*{\includegraphics[scale=1.15]{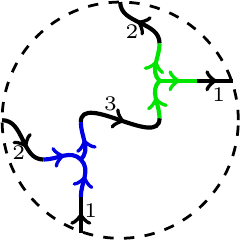}};
\endxy
\end{gather*}
\caption{Planar composition of webs.
}\label{fig:planalg}
\end{figure}

These planar algebra operations are associative, 
but there are no identity operations because we consider webs as embedded and not up to isotopy.
Note that every web in a disk is generated, via 
planar algebra operations, by trivial, merge and split 
webs as in \eqref{eq:mergesplit}.

More generally, we will call any collection of 
sets parametrized by boundary data of planar disks, together 
with associative composition operations given by gluing disks into 
decorated holed disks, a \textit{planar algebra}. If the collection of sets is 
replaced by a collection of categories and the planar algebra operations are 
multi-place functors, then we call this structure a \textit{canopolis}, c.f. Bar-Natan \cite[Section 8.2]{BN1}.

The planar algebra operations for webs via holed disks 
immediately extend to planar algebra 
operations on foams via holed cylindrical foams. 
These operations are functors since they commute with the 
categorical composition of stacking foams (see e.g. 
\fullref{fig:saddles}), and they are strictly associative, so the disk categories assemble into a canopolis.

\begin{figure}[ht]
\begin{gather*}
\xy
(0,.75)*{\includegraphics[scale=1.15]{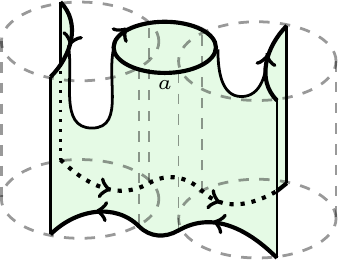}};
\endxy
\end{gather*}
\caption{Planar algebra composition of two saddles.}\label{fig:saddles}
\end{figure}

\begin{definition}\label{definition:foamcan}
We denote by $\freefoamSet$ the canopolis assembled 
from the categories of foams over disks, as described above.
\end{definition}

\begin{remark}\label{rem:can-gen}
For a foam 
in general position, the critical points for the height 
function have neighborhoods modeled on the following foams. 

If the critical point is contained in 
the interior of a facet, then we get Morse-type handle attachments
\begin{gather}\label{eq:foamgen1}
\xy
(0,0)*{\includegraphics[scale=1.15]{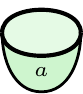}};
\endxy
\quad,\quad
\xy
(0,0)*{\includegraphics[scale=1.15]{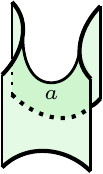}};
\endxy
\quad,\quad
\xy
(0,0)*{\includegraphics[scale=1.15]{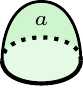}};
\endxy
\end{gather}
each with two possible orientations. If the 
critical point is contained in the interior of a 
seam, then we get the digon creation and annihilation, 
and zip and unzip foams
\begin{gather}\label{eq:foamgen2}
\xy
(0,0)*{\includegraphics[scale=1.15]{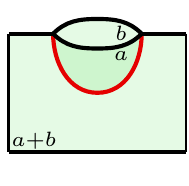}};
\endxy
\quad,\quad
\xy
(0,0)*{\includegraphics[scale=1.15]{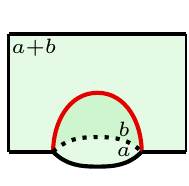}};
\endxy
\quad,\quad
\xy
(0,0)*{\includegraphics[scale=1.15]{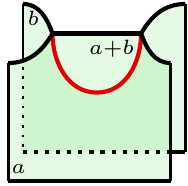}};
\endxy
\quad,\quad
\xy
(0,0)*{\includegraphics[scale=1.15]{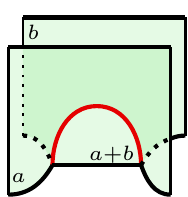}};
\endxy
\\
\label{eq:foamgen2b}
\xy
(0,0)*{\includegraphics[scale=1.15]{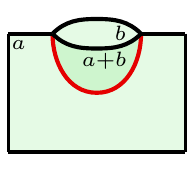}};
\endxy
\quad,\quad
\xy
(0,0)*{\includegraphics[scale=1.15]{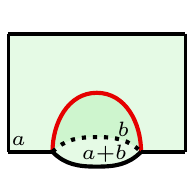}};
\endxy
\quad,\quad
\xy
(0,0)*{\includegraphics[scale=1.15]{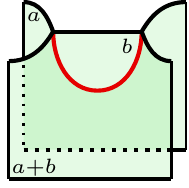}};
\endxy
\quad,\quad
\xy
(0,0)*{\includegraphics[scale=1.15]{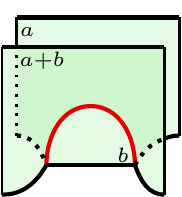}};
\endxy
\end{gather}
each of which admits two orientations. Finally, 
singular vertices of the following types appear: 
\begin{gather}\label{eq:foamgen3}
\xy
(0,0)*{\includegraphics[scale=1.15]{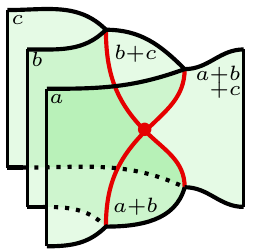}};
\endxy
\quad,\quad
\xy
(0,0)*{\includegraphics[scale=1.15]{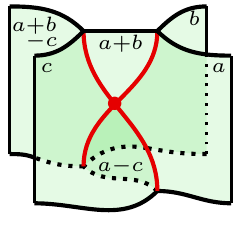}};
\endxy
\end{gather}
For the rightmost foam we assume that $a>c$. These 
foams furthermore exist in upside-down versions, 
and each one admits two orientations.
\end{remark} 

We let 
$\freefoam$ denote the $\Z$-linear extension of 
$\freefoamSet$ where, additionally, every 
facet of foams may be decorated by a partition. 
If gluing two foams results in a 
foam having two partitions on one facet, it is to be 
replaced by a $\Z$-linear combination 
of decorated foams according to a rule modeled on the 
multiplication of Schur polynomials: 
\begin{gather*}
\xy
(0,.8)*{\includegraphics[scale=1.15]{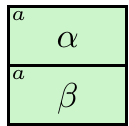}};
\endxy
=
\sum_{\gamma\in\partitions{a}}
\LR{\alpha}{\beta}{\gamma}
\xy
(0,0)*{\includegraphics[scale=1.15]{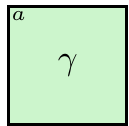}};
\endxy
\end{gather*}
Here $\LR{\alpha}{\beta}{\gamma}$ are the 
Littlewood-Richardson coefficients. Using this rule, we will abuse notation and place $\Z$-linear combinations of partitions on foam facets. We also use the interpretation of such linear combinations as symmetric polynomials in an $a$-element alphabet associated to the facet. This is especially useful when several facets are involved, in which case decorations can be encoded as partially symmetric polynomials.

\begin{definition}\label{definition:bidegree}
The morphisms in $\freefoam$ admit
a $\Z\times\Z$-grading. 
The bidegrees of the foam generators in
\fullref{rem:can-gen} are as follows.
\smallskip
\begin{itemize}

\setlength\itemsep{.15cm}

\item Cups and caps with label $a$ are of bidegree $(a^2,-a)$.

\item Saddles with label $a$ are of bidegree $(-a^2,a)$.

\item Digon creation and annihilation 
foams of label $a$ and $b$ as 
in \eqref{eq:foamgen2} are of bidegree $(-a b,0)$; 
the ones in \eqref{eq:foamgen2b} are of bidegree $(b(a+b),-b)$.

\item Zip and unzip foams of label $a$ and $b$ 
as in \eqref{eq:foamgen2} are of bidegree $(a b,0)$;
the ones in \eqref{eq:foamgen2b} are 
of bidegree $(-b(a+b),b)$.

\item The first foam in \eqref{eq:foamgen3} is of bidegree $(0,0)$, the 
other one of bidegree $(a b,0)$.

\item A decoration by a partition $\alpha$ is 
of bidegree $(2|\alpha|,0)$.
\end{itemize}
\smallskip
We define the degree $\deg{F}$ of a 
foam $F$ in $\freefoam$ to be the integer 
obtained by collapsing its bidegree from $(k,l)$ 
to $k+N l$.
\end{definition}

Alternatively, the bidegree of a foam 
can be defined as a weighted Euler 
characteristic, compare e.g. with \cite[Definition 2.3]{RoW}.

For the following definition we consider 
$H^*_{\mathrm{GL}_N}(pt)\cong\sym(\Eq)$ as our ground ring. Recall that this ring acts on $\mathrm{GL}_N$-equivariant cohomology in the natural way, see 
also \fullref{example:unknot}.

\begin{definitionn}\label{definition:eq-alph}
Let $\freefoameq$ be the additive 
closure of the $\sym(\Eq)$-linear, $\Z$-graded canopolis determined by the following data.
\smallskip
\begin{itemize}

\setlength\itemsep{.15cm}

\item Objects are formal $q$-degree shifts
of webs $q^s \Webs$, where $\Webs$ ranges over the objects of $\freefoam$, and $s\in\Z$.

\item Morphisms are $\sym(\Eq)$-linear 
combinations of foams in $\freefoam$ 
such that
\[
F\colon q^l V \to q^k W \Rightarrow \deg{F}=k-l.
\]
Hereby we consider the variables of $\Eq$ to be of degree two.

\item Categorical composition is given by 
the $\sym(\Eq)$-bilinear extension of the 
composition in $\freefoam$.

\item Planar composition is given by the 
$\sym(\Eq)$-multilinear extension of the planar 
composition in $\freefoam$.\makeqedtri
\end{itemize}
\end{definitionn}

The process of taking the additive 
closure of a canopolis amounts to 
allowing formal direct sums of objects as well as matrices of morphisms 
between them, with composition given by matrix multiplication. 

\subsection{The canopolis of \texorpdfstring{$\mathrm{GL}_N$}{GLN}-equivariant foams}\label{subsec:main-player}
In this section we will describe how to 
take a quotient of $\freefoameq$ that allows 
the construction of equivariant type A link homologies. 
The construction uses the foam evaluation of Robert--Wagner \cite{RoW} 
to run the \textit{universal construction} as presented by Blanchet--Habegger--Masbaum--Vogel \cite{BHMV} for categories of 
closed webs and foams between them. Finally, this is extended 
to the canopolis framework.

\subsubsection{Closed foam evaluation and the universal construction}
\label{sec:univconstr}
The main ingredient that we need is a 
way of evaluating closed foams. This gives a $\sym(\Eq)$-linear \textit{evaluation}
\[
\eval \colon \End_{\freefoameq}(\emptyset)\to \sym(\Eq).
\]
As our evaluation we choose the $\mathrm{GL}_N$-equivariant version of 
the explicit and combinatorial evaluation described by Robert--Wagner in \cite[Section 2.2]{RoW} where the reader 
can find the details. Actually, Robert--Wagner work with torus-equivariant cohomology and thus, with $\C[\Eq]$ as the ground ring. But this simply amounts to an extension of scalars from $\sym(\Eq)$.

\begin{remark}\label{remark:foam-evaluation}
We believe that Robert--Wagner's evaluation coincides 
(up to minor renormalization details) with the evaluation provided by the Kapustin--Li formula as 
formulated by Khovanov--Rozansky \cite{KR3} and used by Mackaay--Sto{\v{s}}i{\'c}--Vaz \cite{MSV} in the construction 
of foam categories and link homologies. We choose to
refer to Robert--Wagner because their evaluation is explicit, 
combinatorial and is already formulated in the equivariant setting.
\end{remark}

First, we focus on the disk category with empty 
boundary in $\freefoameq$, i.e. the category of 
closed webs in the disk and foams between them, which 
we denote by $\freefoameqcl$. Since this 
category is $\sym(\Eq)$-linear, we can consider the following 
representable functor to the monoidal category $\modu{\sym(\Eq)}$ of free 
$\sym(\Eq)$-modules:
\begin{align*}
\overline{\mathcal{F}} \colon  \freefoameqcl &\to \modu{\sym(\Eq)},\\
\Webs &\mapsto \Hom_{\freefoameq}(\emptyset,\Webs),\\
\left(V\xrightarrow{F} W\right) &\mapsto \left(\Hom_{\freefoameq}(\emptyset,V) \xrightarrow{F\circ -} \Hom_{\freefoameq}(\emptyset,W)\right).
\end{align*}

The $\sym(\Eq)$-modules $\overline{\mathcal{F}}(\Webs)$
associated to webs $\Webs$ are too large to be useful. 
However, 
given any $G\in \Hom_{\freefoameq}(\Webs,\emptyset)$, one considers 
the map
\[
\phi_G\colon \overline{\mathcal{F}}(\Webs) \to \sym(\Eq),
\quad
\phi_G(\emptyset\xrightarrow{F}\Webs)=\eval(G\circ F).
\]
This map is well-defined since the evaluation 
depends only on the combinatorial data of the foams and is, thus, invariant under 
isotopy. 
The intersection $\mathcal{I}(\Webs)=\bigcap_G \ker(\phi_G)$ taken 
over all $G\in \Hom_{\freefoameq}(\Webs,\emptyset)$ gives a submodule of 
$\overline{\mathcal{F}}(\Webs)$. One then sets 
\begin{gather}\label{eq:q-space}
\mathcal{F}(\Webs)=\overline{\mathcal{F}}(\Webs)/\mathcal{I}(\Webs).
\end{gather}

We think of $\mathcal{F}(\Webs)$ as the space of all embedded foams with boundary 
$\Webs$ modulo such $\sym(\Eq)$-linear 
combinations that evaluate to zero under 
arbitrary closures. The functor $\mathcal{F}$ 
extends to a $\sym(\Eq)$-linear functor
\[
\mathcal{F} \colon  \freefoameqcl \to \modu{\sym(\Eq)}.
\]
This might be regarded as a (singular) TQFT for the category 
of closed webs and foams between them.

In the following we collect some useful properties, 
which are implicit in \cite[Section 3]{RoW}:
\smallskip
\begin{itemize}

\setlength\itemsep{.15cm}

\item $\mathcal{F}(\emptyset) \cong \sym(\Eq)$.

\item $\mathcal{F}(V\sqcup W) \cong \mathcal{F}(V)\otimes_{\sym(\Eq)}\mathcal{F}(W)$.

\item The modules $\mathcal{F}(\Webs)$ are free over $\sym(\Eq)$ 
with graded rank computed by the MOY-evaluation of webs as in \cite{MOY}.

\item $\mathcal{F}$ restricts to a $1{+}1$-dimensional 
TQFT with values in $\sym(\Eq)$-mod on the subcategory 
of $\freefoameqcl$ consisting of $a$-labeled 
$1$-manifolds and $a$-labeled cobordisms between them. This 
TQFT is determined by the Frobenius algebra given by the 
$\mathrm{GL}_N$-equivariant cohomology of the Grassmannian 
$\grass_a$ from \fullref{example:unknot}.

\item The decorations on foam facets are 
related to basis elements of 
$H^*_{\mathrm{GL}_N}(\grass_a)$ via the $\sym(\Eq)$-linear composite isomorphism
\begin{gather*}
\begin{aligned}
H^*_{\mathrm{GL}_N}(\grass_a) \cong \frac{\sym(\A|\Eq)}{\langle h_{N-a+i}(\A-\Eq)\mid i> 0 \rangle}   
&\xrightarrow{\cong} 
\mathcal{F}(\bigcirc_a),
\\
\schur{\alpha}(\A)&\mapsto
\xy
(0,0)*{\includegraphics[scale=1.15]{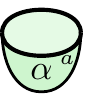}};
\endxy,
\quad
\alpha\in \partitions{a,N-a},
\end{aligned}
\end{gather*}
where $\bigcirc_a$ denotes the $a$-labeled circle object and 
$\mathcal{F}(\bigcirc_a)$ denotes the corresponding quotient space as in 
\eqref{eq:q-space}, in which the illustrated cup foam lives.
\end{itemize}

\subsubsection{Foams modulo relations}
\label{sec:foamsmodrel}
Instead of continuing to work with the functor $\mathcal{F}$, one can form a quotient 
category of webs and foams, by identifying foams that 
are sent to equal $\sym(\Eq)$-linear maps under $\mathcal{F}$. To this end, 
we define $\foamcl{N}$ to be the category 
with the same objects as $\freefoameqcl$, but with morphism spaces
\[
\Hom_{\foamcl{N}}(V,W) = \Hom_{\freefoameqcl}(V,W)/ \ker(\mathcal{F}).
\]
We denote the functor induced by $\mathcal{F}$ on $\foamcl{N}$ by the same symbol.

The category $\foamcl{N}$ is obtained from 
$\freefoameqcl$ by imposing the relations in $\ker(\mathcal{F})$. Many of these relations are of a local nature and some important ones are listed in the next section. Before that, however, we would like to extend this quotient 
to the canopolis framework. 

\subsubsection{Canopolization}

Let $F\colon V \to W$ be an arbitrary foam in $\freefoameq$ between 
webs with boundary. Abstractly, $F$ can be considered as a foam from the empty web 
to the closed web formed by its boundary 
$\partial F =  W\cup_\phi\overline{V}$, which we consider as embedded 
in a disk. Here $\overline{V}$ denotes the web $V$ reflected in a line 
and $\cup_\phi$ stands for the planar algebra operation which connects 
the appropriate boundary points of the webs $W$ and $\overline{V}$. 
Accordingly, there exist invertible canopolis operations:

\begin{gather}\label{eq:bending-trick}
\mathrm{bend}\colon \Hom_{\freefoameq}(V,W) 
\xrightarrow{\cong} 
\Hom_{\freefoameq}(\emptyset,W\cup_\phi\overline{V}).
\end{gather}

Informally, such an operation is given by \textit{bending} (or \textit{clapping}) the entire boundary 
of the foam to the top of the cylinder. In the canopolis, this can 
be achieved by planar composition with the identity foam on 
$\overline{V}$ and then pre-composing with a \textit{cup foam} that 
is obtained by rotating $V$ along a half-circle. The inverse 
operation is given by bending the part of the boundary down again. 
These operations are inverse because their composites transform foams only by isotopies.
Here is a prototypical example:
\[
\xy
(0,0)*{\includegraphics[scale=1.15]{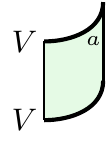}};
\endxy
\;
\xrightarrow{\text{bend}}
\;
\xy
(0,0)*{\includegraphics[scale=1.15]{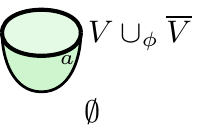}};
\endxy
\]

We define 
\[
\mathcal{I}(V,W)= \mathrm{bend}^{-1}(\mathcal{I}(W\cup_\phi\overline{V}))\subset \Hom_{\freefoameq}(V,W).
\]
It is easy to see that the collection of submodules $\mathcal{I}(V,W)$ is preserved under arbitrary canopolis operations in $\freefoameq$. 

\begin{definition}\label{definition:foam-canopolis-rels}
Let $\foam{N}$ denote the $\sym(\Eq)$-linear, $\Z$-graded canopolis 
obtained as a quotient of the canopolis 
$\freefoameq$ by the canopolis ideal determined 
by the collection $\mathcal{I}(V,W)\subset \Hom_{\freefoameq}(V,W)$.
\end{definition}

The subcategory of webs and foams without boundary in $\foam{N}$ is naturally identified with $\foamcl{N}$ as introduced above. 

\begin{remarkn}\label{rem:gradhom} 
The graded $\sym(\Eq)$-rank of the morphism space 
$\Hom_{\foam{N}}(V,W)$ can be 
computed via the \textit{bending trick} \eqref{eq:bending-trick} as a rescaled MOY-evaluation 
of the closed web $W\cup_\phi\overline{V}$. The rescaling 
depends on the number and labels of critical points introduced 
in the bending. From this, one can deduce the complex dimensions of the graded components of the morphism spaces in 
$\Hom_{\foam{N}}(V,W)$. We give some examples that are used later:
\smallskip
\begin{itemize}

\setlength\itemsep{.15cm}

\item The endomorphism space of a web without vertices, i.e. a trivial tangle, is one-dimensional in degree zero, spanned by the identity foam.

\item The morphism spaces between the empty web and a circle of label $a$ are one-dimensional in degree $-a(N-a)$, spanned by cup and cap foams respectively.

\item The morphism spaces between the two distinct webs that consist of two anti-parallel strands of label $a$ are one-dimensional in degree $a(N-a)$, spanned 
by the saddle foam.\makeqedtri
\end{itemize}
\end{remarkn}

\begin{remark}\label{remark:integrality-1}
(Integrality.) In fact, the closed foam evaluation of 
Robert--Wagner takes values in the ring of 
symmetric polynomials in $\Eq$ with integer 
coefficients $\symz(\Eq)$, see \cite[Main Theorem]{RoW}.
Accordingly, all constructions in this section can 
be performed over the integers as well, 
and all morphism spaces are free $\symz(\Eq)$-modules.
\end{remark}

\subsection{\texorpdfstring{$\mathrm{GL}_N$}{GLN}-equivariant foam relations}\label{subsec:equi-rel}

Below we collect a number of relations 
that hold in $\foam{N}$ and that we will need for our main result. In writing these relations 
we shall decorate foam facets by symmetric 
polynomials. These correspond to $\sym(\Eq)$-linear combinations 
of Schur polynomials, which in turn 
correspond to partitions. We will switch freely between these conventions.

\makeautorefname{subsubsection}{Sections}

We start with the following two lemmas, which follow directly from the discussion in 
\fullref{sec:univconstr} and \ref{sec:foamsmodrel}.

\makeautorefname{subsubsection}{Section}

\begin{lemmab}\label{lemma:neck-cut}
The \textit{neck-cutting} and \textit{sphere relations} hold in $\foam{N}$:
\begin{gather}\label{eq:neckcut}
\xy
(0,0)*{\includegraphics[scale=1.15]{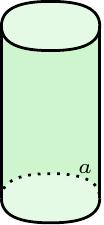}};
\endxy
\!
=
(-1)^{\binom{a}{2}}
\!\!\!\!\!\!\!\!
\sum_{\alpha\in\partition{a}{N-a}}
\xy
(0,0)*{\includegraphics[scale=1.15]{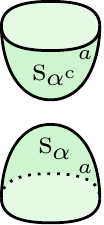}};
\endxy
\quad,\quad
\xy
(0,0)*{\includegraphics[scale=1.15]{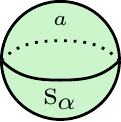}};
\endxy
= \begin{cases}
(-1)^{\binom{a}{2}}, \text{ if } \alpha=\Ybox{a,N-a} ,\\
0, \text{ otherwise.}
\end{cases}
\end{gather}
In the neck-cutting relations we mean $\schur{\complementnorm{\alpha}}=\schur{\complementnorm{\alpha}}(\A-\Eq)$, 
but $\schur{\alpha}=\schur{\alpha}(\B)$, where 
$\A$ and $\B$ are the alphabet on the cup and cap respectively. 
The sphere relations are true 
both for Schur polynomials in $\A$ as well as those in $\A-\Eq$.\qedhere
\end{lemmab}

We explain below in \fullref{remark:weird-signs} why the 
sign $(-1)^{\binom{a}{2}}$ is necessary.

\begin{lemman}\label{lem:neckcutsphere} 
The following relations hold in $\foam{N}$.
\begin{gather}
\xy
(0,0)*{\includegraphics[scale=1.15]{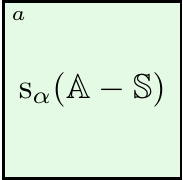}};
\endxy\;=
0,\text{ if }\alpha\notin\partition{a}{N-a} \text{ or if } a>N.
\hspace{3.4cm}
\raisebox{-.9cm}{\qedmake}
\hspace{-3.4cm}
\end{gather}
\end{lemman}

\begin{proposition}\label{proposition:QRholds}
The defining relations \cite[(3.8) to (3.20)]{QR} of the foam 
$2$-category considered 
by Queffelec--Rose hold in $\Foam{N}$. 
\end{proposition}

\begin{proof} 
See \cite[Proof of Proposition 4.2]{RoW}.
\end{proof}

\begin{example}\label{example:QRholds}
We give three examples of relations from 
\fullref{proposition:QRholds} that we use later:
\begin{equation}\label{eq:zipoff}
\xy
(0,0)*{\includegraphics[scale=1.15]{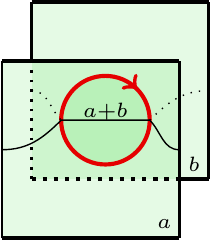}};
\endxy
\;\;
= \sum_{\alpha\in \partition{a}{b}} (-1)^{|\complementtrans{\alpha}|}
\;\;
\xy
(0,0)*{\includegraphics[scale=1.15]{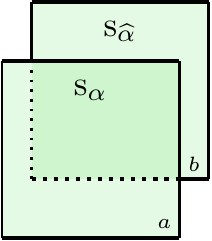}};
\endxy
\end{equation}
This relation is a special case of \cite[(3.14)]{QR} 
in which two facets carry the label zero. 
This means the corresponding
facets are to be erased and their boundary 
seams to be smoothed out.

The second example illustrates the family of Matveev--Piergallini (MP) relations \cite[(3.8)]{QR}:
\begin{equation}\label{eq:mp1}
\xy
(0,0)*{\includegraphics[scale=1.15]{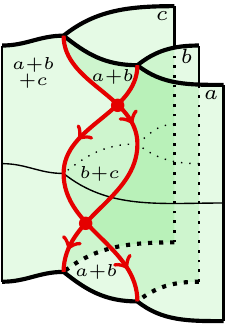}};
\endxy
\;\; = \;\;
\xy
(0,0)*{\includegraphics[scale=1.15]{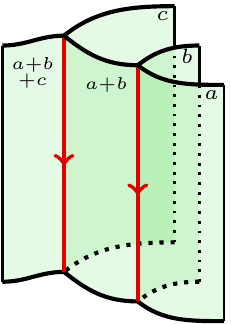}};
\endxy
\end{equation}
Finally, the \textit{decoration migration relations} \cite[(3.9)]{QR}, which involve the Littlewood 
Richardson coefficients:
\begin{equation}\label{eq:dot-mig1}
\xy
(0,0)*{\includegraphics[scale=1.15]{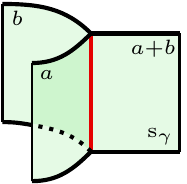}};
\endxy
\;
=
\;
\sum_{\alpha,\beta}
\LR{\alpha}{\beta}{\gamma}
\;
\xy
(0,0)*{\includegraphics[scale=1.15]{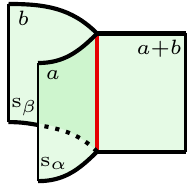}};
\endxy
\end{equation}
If $\A$, $\B$ and $\X$ are the alphabets on the facets of label $a$, $b$ and $a+b$ respectively, then this relation identifies symmetric polynomials in $\A \cup \B$ and $\X$. 
\end{example}

The Sylvester operator from \fullref{example:digon} allows a 
compact description of the blister removal relations \cite[(3.10)]{QR}:

\begin{examplen}\label{example:blister1}
Let $p\in \sym(\A|\Eq)$ and 
$q\in \sym(\B|\Eq)$ be decorations on the front (rear) and rear (front)
facets of the left (right) blister below. Then we have the following 
relations in $\Foam{N}$:
\begin{equation}\label{eq:blister1}
\xy
(0,0)*{\includegraphics[scale=1.15]{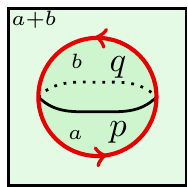}};
\endxy
\;=\;
\xy
(0,0)*{\includegraphics[scale=1.15]{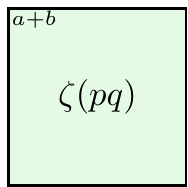}};
\endxy
\;=(-1)^{a b}
\xy
(0,0)*{\includegraphics[scale=1.15]{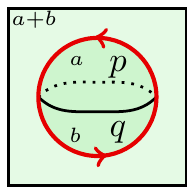}};
\endxy
\hspace{3.0cm}
\raisebox{-.8cm}{\makeqedtri}
\hspace{-3.0cm}
\end{equation}
\end{examplen}

\begin{corollary}\label{corollary:blister}
Consider the theta foam obtained from the 
left-hand side of \eqref{eq:blister1} by 
quotienting the boundary of the square to a 
point. Suppose that this foam carries 
decorations $p\in \sym(\A|\Eq)$, 
$q\in \sym(\B|\Eq)$ and $r\in \sym(\X|\Eq)$ 
on the facets of label $a$, $b$ 
and $a+b$. Then this theta foam evaluates to the 
scalar $\trace_{a+b}(r \zeta(p q))\in\sym(\Eq)$ in $\Foam{N}$ 
if we identify symmetric polynomials in $\A\cup \B$ 
and $\X$ via \eqref{eq:dot-mig1}.
\comm{Suppose a theta foam carries 
decorations $p\in \sym(\A|\Eq)$, 
$q\in \sym(\B|\Eq)$ and $r\in \sym(\X|\Eq)$ 
on the facets of label $a$, $b$ 
and $a+b$. Then it evaluates to the 
scalar $\trace_{a+b}(r \phi(\zeta(p q)))\in\sym(\Eq)$ in $\Foam{N}$ }
\end{corollary}

\begin{remark}\label{remark:weird-signs}
\fullref{corollary:blister} and \eqref{eq:zipoff} explain 
why the sign 
in \eqref{eq:trace-sign} and \eqref{eq:neckcut} 
is necessary. Using the notation from \fullref{corollary:blister}, 
we can evaluate a theta foam in two different ways:
\[
\sum_{\alpha \in \partition{a}{b}}(-1)^{|\complementtrans{\alpha}|}\;
\xy
(0,0)*{\includegraphics[scale=1.15]{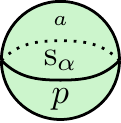}};
\endxy
\quad
\xy
(0,0)*{\includegraphics[scale=1.15]{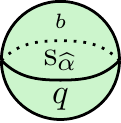}};
\endxy
\;
\stackrel{\eqref{eq:zipoff}}{=}
\;
\xy
(0,0)*{\includegraphics[scale=1.15]{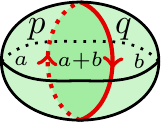}};
\endxy
\;
\stackrel{\eqref{eq:blister1}}{=}
(-1)^{ab}
\;
\xy
(0,0)*{\includegraphics[scale=1.15]{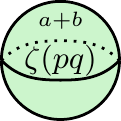}};
\endxy
\]
Thus, taking $p=\emptyset$ 
and $q=\Ybox{b,N-b}$ gives
\[
\trace_a(\Ybox{a,N-a})\trace_b(\Ybox{b,N-b})=(-1)^{ab}\trace_{a+b}(\Ybox{a+b,N-a-b}).
\]
Hence, fixing $\trace_1(\Ybox{1,N-1})=1$ determines the rest to be 
as in \eqref{eq:neckcut}.
\end{remark}

\begin{lemma}\label{lemma:scalar-the-second}
The following relations hold in $\Foam{N}$ for $a\geq b$:
\begin{gather}\label{eq:idem-disc}
\xy
(0,0)*{\includegraphics[scale=1.15]{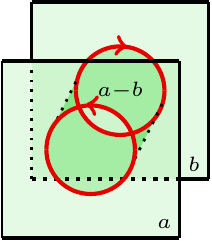}};
\endxy
\;
= (-1)^{b(a-b)}\;\;
\xy
(0,0)*{\includegraphics[scale=1.15]{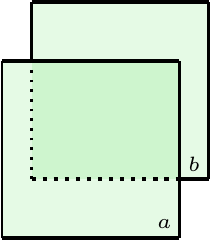}};
\endxy
\end{gather}
Here the middle annulus carries the label $a-b$. 
Similarly in case $b\geq a$, but with swapped orientation on the seams.
\end{lemma}

\begin{proof} 
We detach the middle facet from the 
rear facet via \eqref{eq:zipoff} at 
the expense of a decoration 
$\sum_{\alpha\in \partition{a-b}{b}} (-1)^{|\complementtrans{\alpha}|}\schur{\alpha}(\X)\schur{\complementtrans{\alpha}}(\B)$, 
where $\X$ and $\B$ are the alphabets on the middle and 
rear facets. Removing the blister 
in the front facet gives zero, unless $\alpha=\Ybox{a-b,b}$, 
see \eqref{eq:blister1}.
Hence, only the coefficient 
$(-1)^{|\complementtrans{\emptyset}|}=(-1)^{b(a-b)}$ survives.
\end{proof}

\subsection{Specializations and their relations}\label{subsec:special}
From now on, we let $\alphS=\{\roots_1,\dots,\roots_N\}$ be a set of $N$ pairwise different complex numbers and consider the specialization homomorphism of $\C$-algebras 
\begin{gather*}
\spec\colon\sym(\Eq)\to \C,\quad p(\Eq)\mapsto p(\alphS).
\end{gather*}
Define the canopolis $\FoamS{N}$ as the $\C$-linear 
canopolis obtained from $\Foam{N}$ by specializing 
variables via $\spec$. The objects in $\FoamS{N}$ still 
carry $q$-degree shifts, but on the level of morphism spaces, 
the grading is demoted to a filtration. In the following, we 
abuse notation by writing $\spec$ for the induced $\C$-linear 
canopolis morphism from $\Foam{N}$ to $\FoamS{N}$, which 
respects the filtrations. It is clear that $\FoamS{N}$ 
satisfies the $\alphS$-specialized versions 
of the foam relations already listed for $\Foam{N}$. 

\begin{example}\label{example:Blanchet}
In the case $N=2$ and $\alphS=\{1,-1\}$ the foams in 
$\FoamS{N}$ have been used by Blanchet \cite[Section 4]{Bla}. 
They give rise to a deformed link homology theory that is analogous to Lee's deformation \cite{Lee} of Khovanov homology.
\end{example}

\begin{remark}\label{remark:spez-to-zero}
Specializing all variables in $\Eq$ not to distinct numbers, 
but to zero instead, one recovers a canopolis of webs and foams 
that can be used for the definition of the (non-equivariant) 
colored Khovanov--Rozansky link homologies. The canopolis 
$\FoamS{N}$ can be seen as its \textit{generic} deformation. 
Deformations obtained as other specializations of the variables in $\Eq$ 
have been studied in \cite{RW}.
\end{remark}

\begin{lemma}\label{lemma:idem}
The algebra of 
decorations on an $a$-labeled foam facet is the 
direct sum of one-dimensional summands indexed 
by $a$-element subsets $\wcolor{A}\subset\Sigma$. 
We denote the corresponding idempotents by $\idem{A}$ and display them as
\[
\xy
(0,0)*{\includegraphics[scale=1.15]{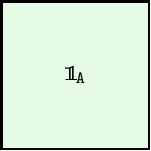}};
\endxy
\]
Additionally, decorations $p\in \sym(\A)$ satisfy
\begin{equation}\label{eq:action-schur}
p(\A)\cdot\idem{A}=p(\wcolor{A})\cdot\idem{A},
\end{equation}
i.e. they act on the $\wcolor{A}$-summand via evaluation at $\wcolor{A}$.
\end{lemma}

Here we use that we work over $\C$. However, $\Q$ would suffice as ground ring if $\alphS\subset \Q$.
In the following we use the idempotents $\idem{A}$ in 
\fullref{lemma:idem} to label the facets of foams in $\FoamS{N}$.

\begin{proof}
See \cite[Lemma 4.2]{RW} and also \cite[Proof of Theorem 3.11]{RW}.
\end{proof}

\begin{lemman}(\cite[Lemma 4.2]{RW}.)\label{lemma:admissibility} 
Let $\wcolor{A}$, $\wcolor{B}$ and $\wcolor{X}$ be
subsets of $\alphS$ with $|\wcolor{A}|=a$, $|\wcolor{B}|=a$ and 
$|\wcolor{X}|=a+b$ respectively.
Then the following relation holds in $\FoamS{N}$:
\[
\xy
(0,0)*{\includegraphics[scale=1.15]{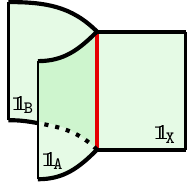}};
\endxy
=0,\quad  \text{unless } \wcolor{A} \cup \wcolor{B} = \wcolor{X}.
\hspace{4.4cm}
\raisebox{-.7cm}{\qedmake}
\hspace{-4.4cm}
\] 
\end{lemman}

\begin{lemma}\label{lemma:scalars}
Let $\A$, $\B$ be alphabets of size $a$ and $b$ respectively. Then we have: 
\[
r(\A,\B)=
{\textstyle\sum_{\alpha\in \partition{a}{b}}}\; (-1)^{|\complementtrans{\alpha}|} \schur{\alpha}(\A)\schur{\complementtrans{\alpha}}(\B) = 
{\textstyle\prod_{A\in \A, B\in \B}}\; (A-B).
\]
From this we get 
\[
r(\B,\A)=(-1)^{a b} r(\A,\B)
\quad\text{and}\quad
r(\A,\B)=0,\;\text{if }\A\cap\B \neq \emptyset.
\] 
If $\X$ is another alphabet, then
\[
r(\A,\B\cup \X)=r(\A,\B)r(\A,\X).
\]
If $\wcolor{A}\subset \alphS$, then we also have
\[
\scalar{A,\alphS\setminus A}=
{\textstyle\sum_{\alpha\in \partition{a}{N-a}}} \; \schur{\alpha}(\wcolor{A})\schur{\complementnorm{\alpha}}(\wcolor{A}-\alphS ) =
{\textstyle\sum_{\alpha\in \partition{a}{N-a}}}\; (-1)^{|\complementtrans{\alpha}|} \schur{\alpha}(\wcolor{A})\schur{\complementtrans{\alpha}}(\alphS \setminus \wcolor{A}),
\]
which is always non-zero.
\end{lemma}

\begin{proof} 
For the first 
equation, see e.g. \cite[Section I.4, Example 5]{MD}. For 
the second equation we use $\schur{\beta}(\wcolor{A}-\alphS ) = (-1)^{|\beta|} \schur{\transpose{\beta}}(\alphS \setminus \wcolor{A})$.
\end{proof}

We need to know how idempotent-colored 
foams behave with respect to the relations found 
above. 
We let $\scalar{A}=(-1)^{\binom{a}{2}}\scalar{A,\alphS\setminus A}$.

\begin{lemman}\label{lemma:scalar-the-first}
The following relations hold in $\FoamS{N}$:
\begin{gather}\label{eq:idem-bubble}
\xy
(0,0)*{\includegraphics[scale=1.15]{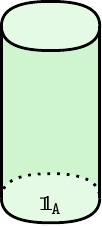}};
\endxy
\;=
\scalar{A}\; 
\xy
(0,0)*{\includegraphics[scale=1.15]{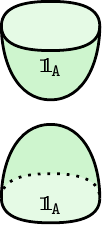}};
\endxy
\quad,\quad
\xy
(0,0)*{\includegraphics[scale=1.15]{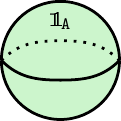}};
\endxy
=
\scalar{A}^{-1}.
\hspace{3.6cm}
\raisebox{-1.0cm}{\makeqed}
\hspace{-3.6cm}
\end{gather}
\end{lemman}

\begin{proof}
The first relation follows from the $\alphS$-specialization of 
the neck-cutting relation \eqref{eq:neckcut} 
combined with \eqref{eq:action-schur} and 
\fullref{lemma:scalars}. The relation for the sphere 
is obtained by composing both sides of the neck-cutting 
relation \eqref{eq:neckcut} with a cap.
\end{proof}

\begin{lemman}\label{lemma:scalar-the-third}
The following relation holds in $\FoamS{N}$:
\begin{gather*}
\xy
(0,0)*{\includegraphics[scale=1.15]{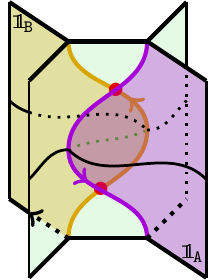}};
\endxy
=
\scalar{A,B}
\xy
(0,0)*{\includegraphics[scale=1.15]{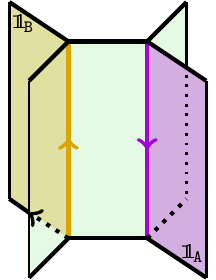}};
\endxy
\quad,\quad
\xy
(0,0)*{\includegraphics[scale=1.15]{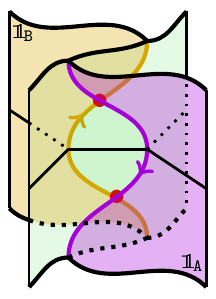}};
\endxy
=
\scalar{A,B}
\xy
(0,0)*{\includegraphics[scale=1.15]{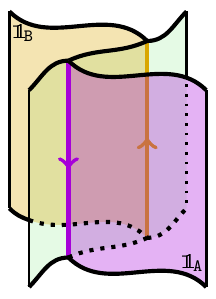}};
\endxy
\\
\label{eq:colblister}
\xy
(0,0)*{\includegraphics[scale=1.15]{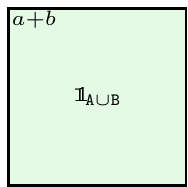}};
\endxy
\;=\;
\scalar{A,B}
\xy
(0,0)*{\includegraphics[scale=1.15]{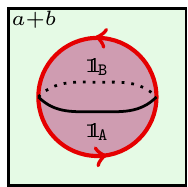}};
\endxy
\hspace{4.35cm}
\raisebox{-.8cm}{\makeqed}
\hspace{-4.35cm}
\end{gather*}
\end{lemman}

\begin{proof}
These relations follow immediately 
from \cite[Relations (3.13) and (3.14)]{QR} as 
well as \cite[Relation (4-7)]{RW}, imported 
via \fullref{proposition:QRholds}, 
and \eqref{eq:action-schur}.
\end{proof}

\begin{remark}\label{remark:flow-chart}
To reduce the complexity of 
computations appearing in \fullref{sec:functorial}, 
we will use \textit{\fcds}\!\! which show the interaction of foam 
facets with a chosen facet (or union of facets) 
in a foam, see \fullref{fig:flowchart}. 

\begin{figure}[ht]
\begin{gather*}
\xy
(0,0)*{\includegraphics[scale=1.15]{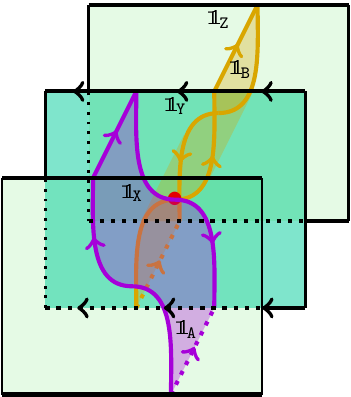}};
\endxy
\leftrightsquigarrow
\xy
(0,0)*{\includegraphics[scale=1.15]{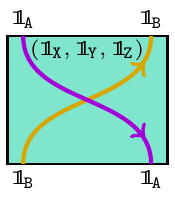}};
\endxy
=
\xy
(0,0)*{\includegraphics[scale=1.15]{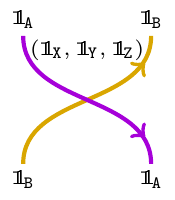}};
\endxy
\end{gather*}
\caption{A \fcd of a foam.
}\label{fig:flowchart}
\end{figure}

Here we adopt the convention that the facets lying in the drawing surface all carry the standard orientation of $\R^2$. The facets in front of the drawing surface are 
colored purple $\purplebox$ and those behind 
golden $\goldenbox$ (and later also: cyan $\cyanbox$). For 
readers familiar with the relationship 
between foams and categorified quantum 
groups (as developed in \cite{Ma1,MPT,LQR,QR}), we emphasize 
that the orientation conventions in phase diagrams 
are not identical to those used in the string diagrams 
of the skew Howe dual categorified quantum group.

The first equations in \fullref{lemma:scalar-the-third}
are simply
\begin{gather}\label{eq:idem-KLR}
\xy
(0,0)*{\includegraphics[scale=1.15]{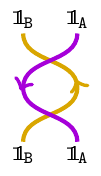}};
\endxy
=
\scalar{A,B}
\xy
(0,0)*{\includegraphics[scale=1.15]{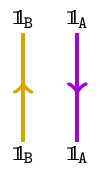}};
\endxy
\quad,\quad
\xy
(0,0)*{\includegraphics[scale=1.15]{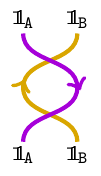}};
\endxy
=
\scalar{A,B}
\xy
(0,0)*{\includegraphics[scale=1.15]{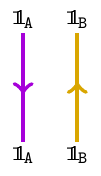}};
\endxy
\end{gather}
The following \fcds 
represent associativity and MP relations on 
foams, which hold in $\Foam{N}$ by 
\fullref{proposition:QRholds}:
\begin{gather}\label{eq:MP-flow}
\begin{aligned}
\xy
(0,0)*{\includegraphics[scale=1.15]{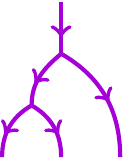}};
\endxy
&=
\reflectbox{
\xy
(0,0)*{\includegraphics[scale=1.15]{figs/fig-1-68}};
\endxy
}
\quad,\quad
\xy
(0,0)*{\includegraphics[scale=1.15]{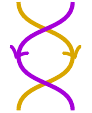}};
\endxy
\;=\;
\xy
(0,0)*{\includegraphics[scale=1.15]{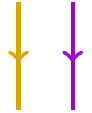}};
\endxy
\end{aligned}
\end{gather}
The relations obtained be 
reversing the orientation on all 
seams in \eqref{eq:MP-flow} also hold. 
Similarly, the pitchfork moves 
hold for all possible orientations on seams 
in $\Foam{N}$, see \cite[Relations (3.28) and (3.29)]{QR} 
(imported via \fullref{proposition:QRholds}):
\begin{gather}\label{eq:idem-pitchfork}
\xy
(0,0)*{\includegraphics[scale=1.15]{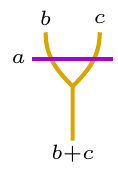}};
\endxy
=
\xy
(0,0)*{\includegraphics[scale=1.15]{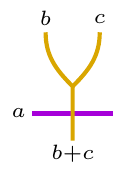}};
\endxy
\quad, \quad
\xy
(0,0)*{\includegraphics[scale=1.15]{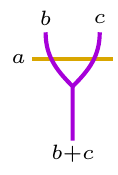}};
\endxy
=
\xy
(0,0)*{\includegraphics[scale=1.15]{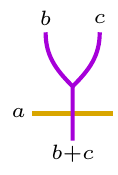}};
\endxy
\end{gather}
We will use several other versions of pitchfork relations, e.g.: \\
\noindent\begin{tabularx}{0.99\textwidth}{XX}
\begin{equation}\hspace{-8cm}\label{eq:idem-pitchfork2}
  \xy
(0,0)*{\includegraphics[scale=1.15]{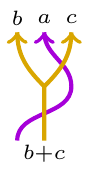}};
\endxy
=
\xy
(0,0)*{\includegraphics[scale=1.15]{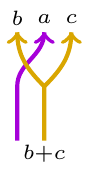}};
\endxy
  \end{equation} &
  \begin{equation}\hspace{-7cm}
    \label{eq:idem-pitchfork3}
 \xy
(0,0)*{\includegraphics[scale=1.15]{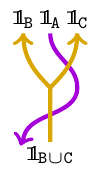}};
\endxy
= \scalar{A,C}
\xy
(0,0)*{\includegraphics[scale=1.15]{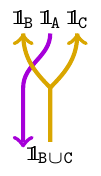}};
\endxy
  \end{equation}
\end{tabularx}\\
The relations of type \eqref{eq:idem-pitchfork2} hold in $\Foam{N}$ and can be deduced from \eqref{eq:idem-pitchfork} via \eqref{eq:MP-flow}. The relations of type \eqref{eq:idem-pitchfork3} hold in $\FoamS{N}$ and are obtained via \eqref{eq:idem-KLR}.
\end{remark}

\begin{lemman}\label{lemma:saddle-reverse} 
With $\wcolor{X}\subset \wcolor{A}$, the following 
\fcd relations hold in $\FoamS{N}$: 
\begin{gather}
\label{eq:idem-disc2}
\xy
(0,0)*{\includegraphics[scale=1.15]{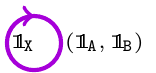}};
\endxy
\;=\frac{\scalar{X,B}}{\scalar{A\setminus X, X}}
\quad, \quad 
\xy
(0,0)*{\includegraphics[scale=1.15]{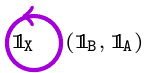}};
\endxy
\;=\frac{\scalar{B,X}}{\scalar{X,A\setminus X}}
\\
\label{eq:idem-KLR2}
\xy
(0,0)*{\includegraphics[scale=1.15]{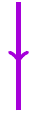}};
\endxy
=
\scalar{B,A}
\xy
(0,0)*{\includegraphics[scale=1.15]{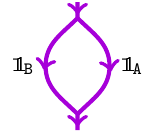}};
\endxy
\quad,\quad
\xy
(0,0)*{\includegraphics[scale=1.15]{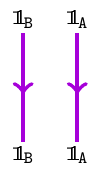}};
\endxy
=
\scalar{B,A}
\xy
(0,0)*{\includegraphics[scale=1.15]{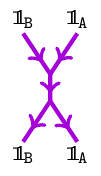}};
\endxy
\\
\label{eq:saddle-reverse}
\xy
(0,0)*{\includegraphics[scale=1.15]{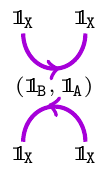}};
\endxy
= \frac{\scalar{B,X}}{\scalar{X,A\setminus X}}
\xy
(0,0)*{\includegraphics[scale=1.15]{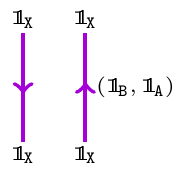}};
\endxy
\quad,\quad
\xy
(0,0)*{\includegraphics[scale=1.15]{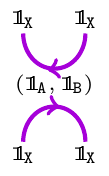}};
\endxy
= \frac{\scalar{X,B}}{\scalar{A\setminus X, X}}
\xy
(0,0)*{\includegraphics[scale=1.15]{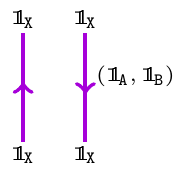}};
\endxy
\hspace{1.0cm}
\raisebox{-.7cm}{\makeqed}
\hspace{-1.0cm}
\end{gather}
\end{lemman}

\begin{proof}
Relation \eqref{eq:idem-disc2} is the shorthand notation for
\begin{gather*}
\xy
(0,0)*{\includegraphics[scale=1.15]{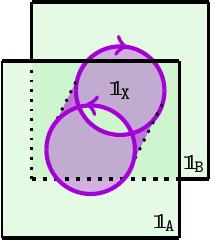}};
\endxy
\;=
\frac{\scalar{X,B}}{\scalar{A\setminus X, X}}
\;\;
\xy
(0,0)*{\includegraphics[scale=1.15]{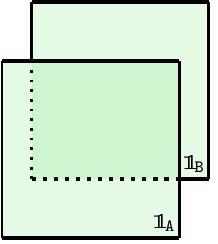}};
\endxy
\end{gather*}
which can be checked by detaching the annulus 
from the rear facet via \eqref{eq:zipoff} and 
collapsing the resulting blister in the front 
facet via \eqref{eq:colblister}. The relations 
in \eqref{eq:idem-KLR2} and \eqref{eq:saddle-reverse} 
follow analogously from the relations in 
\fullref{lemma:scalar-the-third} and MP relations.
\end{proof}

Note that planar isotopies of \fcds relative to their boundary 
correspond to isotopies of foams. Thus, relations 
obtained by isotoping any of the above 
relations continue to hold.

\section{Equivariant, colored link homology via foams and its specializations}\label{sec:foam}
\subsection{Colored link homology}\label{subsec:homology-stuff}
Given an additive, $\sym(\Eq)$-linear 
canopolis $\mathcal C$, 
we obtain from it an additive, $\sym(\Eq)$-linear 
canopolis of (bounded) complexes 
$\Kom(\mathcal C)$ and chain maps between them. 
On chain complexes, the planar algebra 
operation is defined 
to take the planar composites of all chain
groups, 
and the structure of differentials between 
them is modeled on the tensor product of chain complexes. More precisely, we require the inputs of all planar 
algebra operations to be ordered, which determines the 
order in which the formal tensor product is taken and 
thus, the \textit{Koszul signs} in the differentials 
of the resulting tensor product complex. 

\begin{convention}\label{convention:Koszul}
The tensor product of two complexes $(A^*,d_A)$ and 
$(B^*,d_B)$ is defined to be: 
\[
(A\otimes B)^* = 
{\textstyle\bigoplus_{a+b=*}} A^a \otimes B^b, 
\quad d_{(A\otimes B)^*}=  
{\textstyle\sum_{a+b =*}} (-1)^b d_{A^a}
\otimes \id_{B^b} + \id_{A^a}\otimes d_{B^b}
\]
The reordering isomorphism $A\otimes B \cong B\otimes A$ is 
defined to act as $(-1)^{a b}$ times the swap map 
$A^a \otimes B^b \to B^b\otimes A^a$. Note that in the setting 
of a canopolis of chain complexes, where the swap map is the 
identity, the reordering isomorphisms act as the identity, 
except on terms of 
doubly-odd homological degree. This convention agrees with \cite[Appendix A.6.1]{CMW}. 
\end{convention} 

Since null-homotopic chain maps form an ideal with respect to planar 
algebra operations, it also makes sense to 
consider the homotopy canopolis of (bounded) complexes $\Komh(\mathcal C)$ with morphisms given by chain maps up to homotopy. Forgetting the additional categorical structure in 
$\Kom(\mathcal C)$ or $\Komh(\mathcal C)$, 
we can also view both as planar algebras.
A detailed discussion of these constructions 
is given 
in e.g. \cite[Sections 3 and 4.1]{BN1}. All these notions extend to the graded setup as well.

\begin{convention}\label{convention:links}
All links and tangles, as well as their diagrams, are 
assumed to be labeled (or \textit{colored}) and oriented. Recall that two such tangle 
diagrams represent 
the same tangle if and only if the diagrams can be obtained 
from each other via a finite number of planar 
isotopies and Reidemeister moves 
as in \fullref{fig:movie1}, i.e. the ones 
displayed therein as well as their orientation and crossing variations.
\end{convention}

The tangles and their diagrams form planar algebras as presented in 
\fullref{subsection:canopolis}. Henceforth,
we consider the planar algebra of 
tangle diagrams $\coltangles$, which is 
generated by single strands, positive and negative crossings, 
as in \fullref{fig:tangle-gen}.

\begin{figure}[ht]
\begin{gather*}
\xy
(0,.75)*{\includegraphics[scale=1.15]{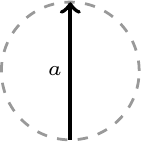}};
\endxy
\quad,\quad
+\colon
\xy
(0,0)*{\includegraphics[scale=1.15]{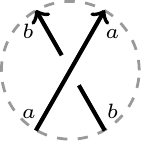}};
\endxy
\quad,\quad
-\colon
\xy
(0,0)*{\includegraphics[scale=1.15]{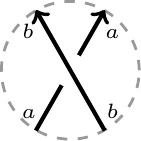}};
\endxy
\end{gather*}
\caption{Generators of $\coltangles$.
}\label{fig:tangle-gen}
\end{figure}

\begin{definitionn}\label{definition:basic-complexes}
Consider the map of planar algebras
\begin{equation}\label{eq:the-map-of-pas}
\Hgen{\cdot}\colon\coltangles\to\Kom(\freefoameq),	
\end{equation}
which is defined on generating 
tangle diagrams from \fullref{fig:tangle-gen} as follows.
\smallskip
\begin{enumerate}[label=$\bullet$]

\setlength\itemsep{.15cm}

\item It maps the $a$-labeled strand 
to the corresponding web, regarded as a 
complex concentrated in homological degree zero.

\item It maps positive crossings with an overstrand 
labeled $a$ and understrand labeled $b$ with $a\geq b$ 
to:
\begin{equation}\label{eq:crossingcx}
\begin{aligned}
\left\llbracket
\vphantom{\int^{\substack{1\\ 1\\1\\1}}}\right.\hspace*{-.15cm}
\xy
(0,0)*{\includegraphics[scale=1.15]{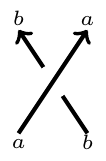}};
\endxy\hspace*{-.15cm}
\left.\vphantom{\int^{\substack{1\\ 1\\1\\1}}}\right\rrbracket
\stackrel{\raisebox{.05cm}{\tiny $a{\geq}b$}}{=}
q^{-x}
\underline{
\xy
(0,0)*{\includegraphics[scale=1.15]{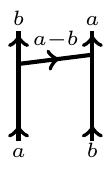}};
\endxy
}
&\xrightarrow{\phantom{.}d^+_{0}\phantom{.}}
q^{1-x}
\xy
(0,0)*{\includegraphics[scale=1.15]{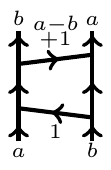}};
\endxy
\xrightarrow{\phantom{.}d^+_1\phantom{.}}
\cdots
\\
\cdots
&\xrightarrow{d_{b-2}^+} q^{b-1-x}
\xy
(0,0)*{\includegraphics[scale=1.15]{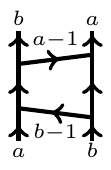}};
\endxy
\xrightarrow{d^+_{b-1}}
q^{b-x}
\xy
(0,0)*{\includegraphics[scale=1.15]{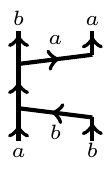}};
\endxy
\end{aligned}
\end{equation}
Again, powers of $q$ denote shifts in 
the $q$-degree, $x=b(N-b)$ and 
the underlined term is in homological 
degree zero. The differentials are given 
by the foams in \fullref{fig:differential}.
 
\item A positive crossing with labels $b>a$ is mapped to the complex obtained from 
the one in \eqref{eq:crossingcx} by reflecting 
webs in a vertical axis (and foams in a corresponding plane) 
and swapping labels $a$ and $b$.

\item The complexes for the negative crossings are 
obtained from the positive crossing complexes by 
inverting the $q$-degrees 
and the homological degrees, and 
reflecting the differential foams in a horizontal plane.
\makeqedtri
\end{enumerate}
\end{definitionn}

\begin{figure}[ht]
\begin{gather*}
d_k^+=
\xy
(0,0)*{\includegraphics[scale=1.15]{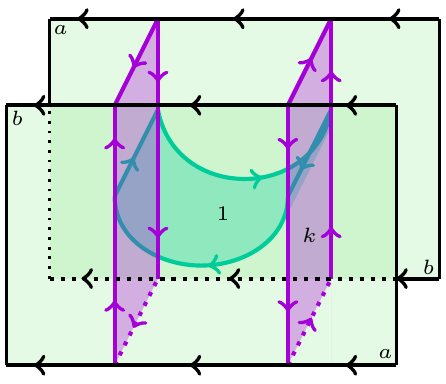}};
\endxy
\end{gather*}
\caption{The differentials $d_k^+$. 
(The other differentials are similar.)
}\label{fig:differential}
\end{figure}

\begin{example}\label{example-diffs}
In case $a=b=1$ the complexes assigned 
to positive and negative crossings are
of length two with differentials 
given by zip and unzip foams as in \eqref{eq:foamgen2}.
\end{example}

\begin{theorem}\label{theorem:wulink}
If $T_D,T^{\prime}_D$ are 
two tangle diagrams representing the 
same colored, oriented tangle, then $\Hgen{T_D}\cong\Hgen{T^{\prime}_D}$ holds in 
$\Komh(\Foam{N})$, i.e. $\Hgen{\cdot}$ is an 
invariant of colored, oriented tangles.
\end{theorem}

\begin{proof} 
Any reordering of the crossings in a 
tangle diagram induces an isomorphism between 
the respective invariants; so we disregard the 
ordering for the purpose of the following proof. 
By virtue of the planar algebraic construction, it suffices to show that the chain complexes associated to the tangles on both sides of each Reidemeister move are chain homotopy equivalent.

This result has already been proved 
in a variety of different categorifications of 
the MOY calculus, so we refer to the literature 
and only comment on the slight variations that we need. 

The proof closest to our endeavor as regards 
generality has been given by Wu 
in \cite[Proof of Theorem 1.1]{Wu2} in the context 
of his equivariant, colored 
link homology constructed 
via matrix factorizations. A proof purely in the language of webs 
and foams (although non-equivariant and with slightly different foam 
relations) has been first given by Mackaay--Sto{\v{s}}i{\'c}--Vaz \cite[Section 7]{MSV} 
for uncolored tangles. Queffelec--Rose \cite[Section 4.3]{QR} have 
provided proofs for the Reidemeister moves 
that we will call (R2+) and (R3+) 
(see \fullref{subsec:simple-complexes}) for all colors. 
They have also described the behavior of the invariant 
under \textit{fork slides} and \textit{fork twisting}, 
see \cite[Proofs of Theorem 4.7 and Proposition 4.10]{QR}.
These results immediately generalize to the equivariant 
setup with ground ring $\sym(\Eq)$ in 
$\foam{N}$ and---following Wu's strategy \cite{Wu2}---guarantee 
that the Reidemeister moves of type (R1), (R2--) and (R3--) hold for 
all colors if they 
hold in the uncolored case. The latter is easily 
checked by hand, e.g. analogous to the proof in \cite[Section 7]{MSV}.
\end{proof}

\begin{remark}\label{remark:integrality-2}
(Integrality.) \fullref{theorem:wulink} and its proof 
outlined above hold over $\Z$. In particular, 
the Reidemeister homotopy equivalences provided by 
Mackaay--Sto{\v{s}}i{\'c}--Vaz \cite[Section 7]{MSV} are 
integral. In the colored case, the Reidemeister 
homotopy 
equivalences used by Queffelec--Rose as well as their fork slides 
and fork twists are also integral, see \cite[Proposition 4.10]{QR}.
\end{remark}

\begin{definition}\label{definition:link-homology}
Let $L_D$ be a colored, oriented link diagram. Then the  
equivariant, colored Khovanov--Rozansky homology of $L_D$ is the 
bigraded,
finite-rank $\sym(\Eq)$-module defined as
\[
\mathrm{KhR}^\Eq(L_D) = H_*(\mathcal{F}(\Hgen{L_D})),
\]
where $\mathcal{F}$ denotes the TQFT from \fullref{sec:foamsmodrel}. 
\end{definition}

By \fullref{theorem:wulink}, $\mathrm{KhR}^\Eq$ is invariant 
under Reidemeister 
moves and planar isotopies up to explicit isomorphisms. 

If one sets the variables in the alphabet $\Eq$ equal 
to zero before applying an analog of the TQFT 
$\mathcal{F}$ in \fullref{definition:link-homology}, one obtains 
the (non-equivariant) bigraded, colored, Khovanov--Rozansky 
link homologies of Wu \cite{Wu1} and Yonezawa \cite{Yon}, 
see \cite[Proposition 4.2]{RoW} 
and \cite[Theorem 4.11]{QR}. In the uncolored case, 
these agree with the original Khovanov--Rozansky link homologies \cite{KR}.
More generally, the alphabet $\Eq$ can be specialized 
to any $N$-element multiset $\Sigma$ of complex numbers, 
and one obtains deformed Khovanov--Rozansky link homologies $\mathrm{KhR}^{\Sigma}$, 
see \cite{Wu2, RW}. In the next section, we will 
explain this in detail for $\Sigma=\alphS$.

\begin{remark}\label{remark:R1move-shift}
We have chosen a grading convention which results 
in a tangle invariant that respects the Reidemeister $1$ move. 
Another natural grading convention leads to shifts under the 
Reidemeister $1$ move, but has the advantage of allowing 
invariance under fork slide moves. We will not pursue this 
further in this paper.
\end{remark}

\begin{remark}\label{remark:no-fance-signs-yet}
For the purpose of proving \fullref{theorem:wulink} it is not 
necessary to specify a particular pair of inverse chain 
homotopy equivalences for each Reidemeister move. Indeed, any such pair can be 
rescaled by a pair of inverse units in $\C$. However, for the purpose 
of defining chain maps associated to link cobordisms and for the proof 
of functoriality, such a specification is necessary. Arguments as in 
\fullref{subsec:uptounits} imply that the ambiguity is limited to 
rescaling by a unit in $\C$, and we will choose a particular scaling 
in \fullref{subsec:simple-complexes}. 
\end{remark}

\subsection{Generic deformations and Karoubi envelope technology}\label{subsec:func-karoubi}
As before, let us denote by $\alphS=\{\roots_1,\dots,\roots_N\}$ 
a set of $N$ pairwise different complex numbers.

\begin{lemman}(\cite[Lemma 4.2]{RW}.)\label{lemma:idem2}
In $\FoamS{N}$, the algebra of decorations on the identity foam on a 
web $W$ is the direct sum of one-dimensional summands. 
The corresponding idempotents 
are given by colorings of all facets by idempotents 
as in \fullref{lemma:idem}, which are admissible in
the sense of 
\fullref{lemma:admissibility}, i.e. the associated 
subsets of $\alphS$ add up around every singular seam.\qedmake
\end{lemman}

Recall that the Karoubi envelope $\Kar(\mathcal{C})$ 
of a category $\mathcal{C}$ has as object pairs $(O,e)$, where 
$O$ is an object of $\mathcal{C}$ and $e\colon O\to O$ is an idempotent. 
The morphisms from $(O,e)$ to $(O^{\prime},e^{\prime})$ are morphisms of 
$\mathcal{C}$ compatible with the idempotents, i.e. triples $(e,f,e^{\prime})$ 
with $f\colon O\to O^{\prime}$ such that $f\circ e=e^{\prime}\circ f$ 
holds in $\mathcal{C}$. 
In case $\mathcal{C}$ is additive, one can split idempotents, 
i.e. by writing $O=(O,\mathrm{id})$ and $\mathrm{im}(e)=(O,e)$, we get
\begin{equation*}
O\cong\mathrm{im}(e)\oplus\mathrm{im}(\mathrm{id}-e)\quad
(\text{in }\Kar(\mathcal{C})).
\end{equation*}
We will use 
a full subcategory 
of $\Kar(\FoamS{N})$ which contains all 
idempotents identified in \fullref{lemma:idem2} above.

\begin{definition}\label{definition:kar-foams}
Let $\FoamSH$ denote the full additive, $\sym(\Eq)$-linear subcategory 
of $\Kar(\FoamS{N})$  
containing all objects of the form $(W,c(W)\mathrm{id}_W)$ as well as 
their $q$-degree shifts. 
Here $W$ is a web and $c(W)\mathrm{id}_W$ is an admissibly 
idempotent-colored identity foam on $W$ as in 
\fullref{lemma:idem2}.
\end{definition}

We shall think of the objects $(W,c(W)\mathrm{id}_W)$ as 
idempotent-colored webs. $\FoamS{N}$ embeds as a full 
subcategory of $\FoamSH$ since the identity foam on every 
web can be split into the sum of its idempotent-colorings. 
Correspondingly, uncolored webs can be split into direct 
sums of idempotent-colored webs. For example:
\begin{equation*}
\xy
(0,0)*{\includegraphics[scale=1.15]{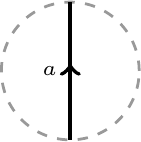}};
\endxy
\cong\bigoplus_{\wcolor{A}\subset\alphS}\;
\xy
(0,0)*{\includegraphics[scale=1.15]{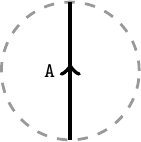}};
\endxy
\quad
(\text{in } \FoamSH).
\end{equation*}
Here the direct sum runs over all $a$-element 
subsets $\wcolor{A}$ of $\alphS$. Another example is:
\begin{gather}\label{eq:local-idems}
\xy
(0,0)*{\includegraphics[scale=1.15]{figs/fig-1-19}};
\endxy
\cong\bigoplus_{
\substack{
\wcolor{A},\wcolor{B}\subset\alphS, \\ \wcolor{A}\cap\wcolor{B}=\emptyset
}
}\;
\xy
(0,0)*{\includegraphics[scale=1.15]{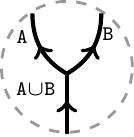}};
\endxy
\quad
(\text{in }\FoamSH).
\end{gather}
Note that the idempotent on the 
edge of the biggest label is determined by 
the other two due to the admissibility 
condition from \fullref{lemma:admissibility}.

We will denote by  
$\HgenS{\cdot}\colon \coltangles\to \Komh(\FoamS{N})\subset \Komh(\FoamSH)$
the $\alphS$-specialization of $\Hgen{\cdot}$, obtained 
via post composition with $\spec$ from 
\fullref{subsec:special}. From 
\fullref{theorem:wulink} we immediately obtain the following specialization.

\begin{theorem}\label{theorem:wulink2} 
If $T_D,T^{\prime}_D$ are 
two tangle diagrams representing the 
same colored, oriented tangle, then 
$\HgenS{T_D}{\cong}\HgenS{T^{\prime}_D}$ holds in 
$\Komh(\FoamS{N})$, i.e. $\!\HgenS{\cdot}\!$ is an 
invariant of colored, oriented tangles. 
\end{theorem}

The next lemma provides a 
decomposition of the complexes 
associated to crossings, which dramatically 
simplifies the computation of $\alphS$-deformed 
link invariants. By convention, we indicate the 
homological degree using powers of $t$.

\begin{lemman}\label{lemma:crossing-color}
The complex associated to a link in 
$\Komh(\FoamSH)$ is isomorphic to a complex 
with trivial differentials, e.g. for $a\geq b$ we have locally:
\begin{equation}\label{eq:crossing-color-generic}
\left\llbracket\vphantom{\int^{\substack{1\\ 1\\1\\1}}}\right.\hspace*{-.15cm}
\xy
(0,0)*{\includegraphics[scale=1.15]{figs/fig-2-4}};
\endxy
\hspace*{-.15cm}
\left.\vphantom{\int^{\substack{1\\ 1\\1\\1}}}\right\rrbracket^{\alphS}
\cong
\bigoplus_{\wcolor{A},\wcolor{B}\subset\alphS}
\xy
(0,0)*{\includegraphics[scale=1.15]{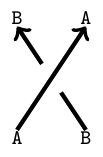}};
\endxy
\cong 
\bigoplus_{k=0,\dots,b}
\bigoplus_{
\substack{\wcolor{A},\wcolor{B}\subset\alphS, \\\ |\wcolor{B}\setminus\wcolor{A}|=k}
}\!\!
t^{k}
\xy
(0,0)*{\includegraphics[scale=1.15]{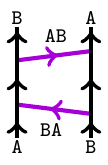}};
\endxy
\hspace{3.1cm}
\raisebox{-.6cm}{\makeqed}
\hspace{-3.1cm}
\end{equation}
\end{lemman}
Here and in the following, we 
display a tangle diagram $T_D$ with strands colored by idempotents
as a shorthand notation for the corresponding idempotent-colored 
subcomplex of $\HgenS{T_D}$ in $\Komh(\FoamSH)$. We also write $\wcolor{XY}=\wcolor{X}\setminus\wcolor{Y}$ for any two subsets 
$\wcolor{X},\wcolor{Y}\subset \alphS$. If a web edge is colored 
with $\wcolor{XY}=\emptyset$, then it is to be deleted from the diagram. The complexes 
associated to other crossings split analogously, with negative 
crossings receiving negative shifts in homological degree.

\begin{proof}
The left isomorphism is a special case of \cite[Lemma 5.9]{RW}, the right 
isomorphism follows by applying \eqref{eq:local-idems} to \eqref{eq:crossingcx}.
\end{proof}

In the case of $\alphS=\{1,-1\}$ and $a=b=1$, the decomposition \eqref{eq:crossing-color-generic} precisely recovers Blanchet's decomposition in \cite[Figure 17]{Bla}.

\subsection{Simple resolutions and Reidemeister foams}\label{subsec:simple-complexes}
In this section we study the complexes 
associated to link diagrams in the $\alphS$-deformation and the homotopy equivalences 
between them, which are induced by 
Reidemeister moves. This will also lead us to 
determine a particular scaling for the Reidemeister 
homotopy equivalences in $\Komh(\foam{N})$ that is necessary 
for the functoriality proof in \fullref{sec:functorial}.

As a first step, we apply 
\fullref{lemma:crossing-color} to 
all crossings in the link diagrams appearing in Reidemeister moves, and 
we immediately obtain:

\begin{lemman}\label{lemma:thecomplexes}
The following isomorphisms hold in $\Komh(\FoamSH)$.
\begin{gather*}
\HgenS{
\xy
(0,0)*{\includegraphics[scale=1.15]{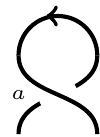}};
\endxy
\;
	}
\cong \bigoplus_{
\substack{\wcolor{A}\subset\alphS \\ |\wcolor{A}| =a}
}
t^{0}
\xy
(0,0)*{\includegraphics[scale=1.15]{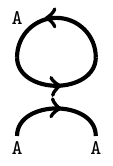}};
\endxy
\quad,\quad
\HgenS{
\xy
(0,0)*{\includegraphics[scale=1.15]{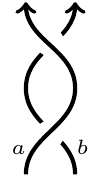}};
\endxy
}
\stackrel{\raisebox{.05cm}{\tiny $a{\geq}b$}}{\cong}
\!\!\bigoplus_{
\substack{\wcolor{A},\wcolor{B}\subset\alphS, \\ |\wcolor{A}|=a, |\wcolor{B}|=b}
}
\!\!\!\!
t^0
\xy
(0,0)*{\includegraphics[scale=1.15]{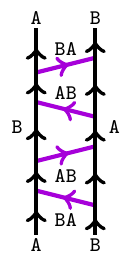}};
\endxy\\
\HgenS{
\xy
(0,0)*{\includegraphics[scale=1.15]{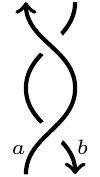}};
\endxy
}
\stackrel{\raisebox{.05cm}{\tiny $a{\geq}b$}}{\cong}
\!\!\bigoplus_{
\substack{\wcolor{A},\wcolor{B}\subset\alphS, \\ |\wcolor{A}|=a, |\wcolor{B}|=b}
}\!\!
t^0
\xy
(0,0)*{\includegraphics[scale=1.15]{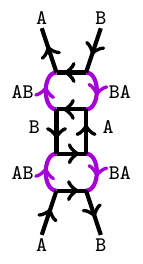}};
\endxy
\quad,\quad
\HgenS{
\xy
(0,0)*{\includegraphics[scale=1.15]{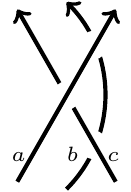}};
\endxy
}\!\!\!
\stackrel{\raisebox{.05cm}{\tiny $a{\geq}b{\geq}c$}}{\cong}\!\!\!\!\!\!\!\!\!
\bigoplus_{\substack{
\wcolor{A},\wcolor{B},\wcolor{C}\subset\alphS, \\ |\wcolor{A}|=a, |\wcolor{B}|=b,|\wcolor{C}|=c \\
k = |\wcolor{\BA}| + |\wcolor{\CA}| - |\wcolor{\CB}|
}
} \!\!\!t^{k} \! 
\xy
(0,0)*{\includegraphics[scale=1.15]{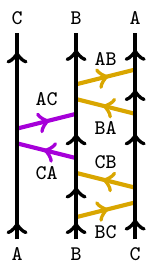}};
\endxy
\end{gather*}
The $\alphS$-deformed 
complexes associated to link diagrams appearing in other 
versions of Reidemeister moves have analogous simplifications.\qedmake
\end{lemman}

\begin{lemma}\label{lemma:thecomplexes2}
Under the direct sum decompositions of complexes 
from \fullref{lemma:thecomplexes}, the 
homotopy equivalences associated to Reidemeister 
moves are given by diagonal matrices whose non-zero 
entries are invertible, 
idempotent-colored foams in $\FoamS{N}$. The analogous result also holds for 
all possible colorings and orientations of 
Reidemeister moves, which do not appear 
in \fullref{lemma:thecomplexes}.
\end{lemma}

\begin{proof}
This follows from the facts that $\HgenS{\cdot}$ is a tangle invariant in 
$\Komh(\FoamSH)$ by \fullref{theorem:wulink2}, 
and that any non-zero 
foam between colored webs preserves the idempotent-colors on the boundary edges of such webs.
\end{proof}

For the following, we choose (once and for all) an 
ordering of $\alphS=\{\lambda_1,\dots \lambda_N\}$.

\begin{definition}\label{definition:nice-colors}
Our favourite idempotent-coloring of a tangle 
diagram $T_D$ is given by coloring every $a$-labeled 
strand with the idempotent $\idem{A}$ for
$\wcolor{A}=\{\lambda_1,\dots,\lambda_a\} \subset\alphS$. 
\fullref{lemma:crossing-color} implies 
that the subcomplex of $\HgenS{T_D}$ corresponding to 
this favourite idempotent-coloring simplifies
to a single colored web of minimal 
combinatorial complexity, which we will
call the \textit{simple resolution} of $T_D$. 
\end{definition}

\begin{example}\label{example:thecomplexes}
The simple resolutions of the 
diagrams from \fullref{lemma:thecomplexes} are obtained as follows.
\begin{gather*}
\xy
(0,.9)*{\includegraphics[scale=1.15]{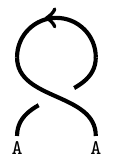}};
\endxy
\longmapsto 
\xy
(0,0)*{\includegraphics[scale=1.15]{figs/fig-2-15}};
\endxy
\quad,\quad
\xy
(0,0)*{\includegraphics[scale=1.15]{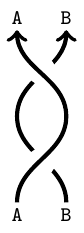}};
\endxy
\stackrel{\raisebox{.05cm}{\tiny $a{\geq}b$}}{\longmapsto}
\xy
(0,0)*{\includegraphics[scale=1.15]{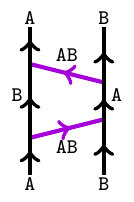}};
\endxy
\quad,\quad
\xy
(0,0)*{\includegraphics[scale=1.15]{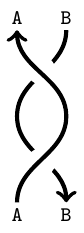}};
\endxy
\stackrel{\raisebox{.05cm}{\tiny $a{\geq}b$}}{\longmapsto}
\xy
(0,0)*{\includegraphics[scale=1.15]{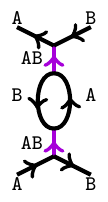}};
\endxy\\
\phantom{.}
\xy
(0,0)*{\includegraphics[scale=1.15]{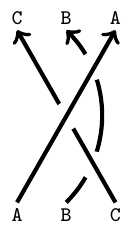}};
\endxy
\stackrel{\raisebox{.05cm}{\tiny $a{\geq}b{\geq}c$}}{\longmapsto}
\xy
(0,0)*{\includegraphics[scale=1.15]{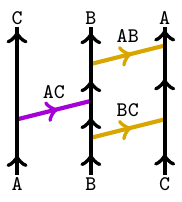}};
\endxy
\quad,\quad
\xy
(0,0)*{\includegraphics[scale=1.15]{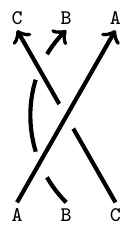}};
\endxy
\stackrel{\raisebox{.05cm}{\tiny $a{\geq}b{\geq}c$}}{\longmapsto}
\xy
(0,0)*{\includegraphics[scale=1.15]{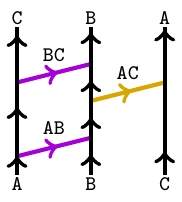}};
\endxy
\end{gather*} 
Here and in the following there are no 
homological degree shifts. Link diagrams which appear in other 
oriented versions of Reidemeister moves 
have similar simple resolutions. We will discuss 
this variety in more detail below.
\end{example}

\begin{lemman}\label{lemma:induced-RM}
The Reidemeister $1$ and $2$ homotopy equivalences in $\Komh(\FoamSH)$ are 
realized on simple resolutions by complex 
multiples of the foams in \fullref{fig:RMcomplexes}.
\smallskip
\begin{enumerate}

\setlength\itemsep{.15cm}

\item[(R1)] If a Reidemeister $1$ move in a strand 
of label $a$ increases the writhe of the 
link diagram, then the corresponding foams between simple resolutions has to 
be rescaled by $\scalar{A}$. For example:
\[
\xy
(0,0)*{\includegraphics[scale=1.15]{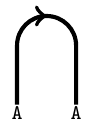}};
\endxy
\xrightarrow{\scalar{A} F_1}
\xy
(0,0)*{\includegraphics[scale=1.15]{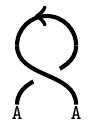}};
\endxy
\xrightarrow{G_1}
\xy
(0,0)*{\includegraphics[scale=1.15]{figs/fig-2-31}};
\endxy
\quad,\quad
\xy
(0,0)*{\includegraphics[scale=1.15]{figs/fig-2-31}};
\endxy
\xrightarrow{F_1}
\xy
(0,0)*{\includegraphics[scale=1.15]{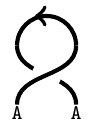}};
\endxy
\xrightarrow{\scalar{A} G_1}
\xy
(0,0)*{\includegraphics[scale=1.15]{figs/fig-2-31}};
\endxy 
\]

\item[(R2+)] For Reidemeister 
$2$ moves between strands of labels $a,b$ with parallel orientation, 
the foams $F_2^+$ and $G_2^+$ are 
normalized by a sign 
$\epsilon=(-1)^{\min(a,b)(a-b)}$ depending on
which strand is pushed over. For example:
\[
\xy
(0,0)*{\includegraphics[scale=1.15]{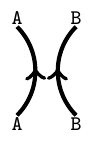}};
\endxy
\xrightarrow{\epsilon F^+_2}
\xy
(0,0)*{\includegraphics[scale=1.15]{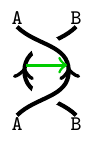}};
\endxy
\xrightarrow{G^+_2}
\xy
(0,0)*{\includegraphics[scale=1.15]{figs/fig-2-34}};
\endxy
\quad,\quad
\xy
(0,0)*{\includegraphics[scale=1.15]{figs/fig-2-34}};
\endxy
\xrightarrow{F^+_2}
\xy
(0,0)*{\includegraphics[scale=1.15]{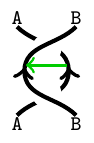}};
\endxy
\xrightarrow{\epsilon G^+_2}
\xy
(0,0)*{\includegraphics[scale=1.15]{figs/fig-2-34}};
\endxy
\]

\item[(R2--)] For Reidemeister $2$ moves between strands of 
labels $a,b$ and opposite orientations, the foams $F_2^-$ and $G_2^-$ are 
also normalized by $\epsilon$ as follows:
\[
\xy
(0,0)*{\includegraphics[scale=1.15]{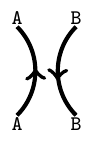}};
\endxy
\xrightarrow{ F^-_2}
\xy
(0,0)*{\includegraphics[scale=1.15]{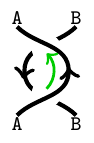}};
\endxy
\xrightarrow{\epsilon G^-_2}
\xy
(0,0)*{\includegraphics[scale=1.15]{figs/fig-2-37}};
\endxy
\quad,\quad
\xy
(0,0)*{\includegraphics[scale=1.15]{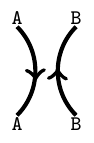}};
\endxy
\xrightarrow{\epsilon F^-_2}
\xy
(0,0)*{\includegraphics[scale=1.15]{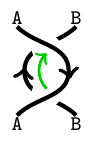}};
\endxy
\xrightarrow{G^-_2}
\xy
(0,0)*{\includegraphics[scale=1.15]{figs/fig-2-39}};
\endxy
\hspace{1.9cm}
\raisebox{-.55cm}{\makeqed}
\hspace{-1.9cm}
\]
\end{enumerate}
\end{lemman}

\begin{figure}[ht]
\begin{gather*}
F_1=
\xy
(0,0)*{\includegraphics[scale=1.15]{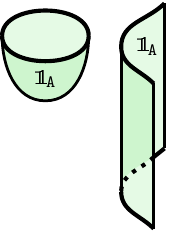}};
\endxy
\quad,\quad
F_2^{+}=
\xy
(0,0)*{\includegraphics[scale=1.15]{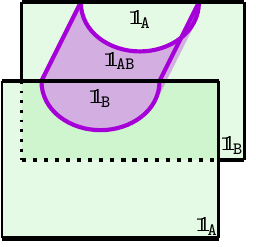}};
\endxy
\quad,\quad
F_2^{-}=
\xy
(0,0)*{\includegraphics[scale=1.15]{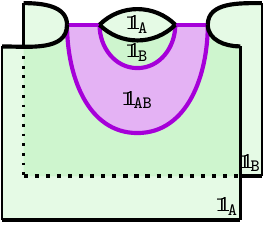}};
\endxy
\\
G_1=
\xy
(0,0)*{\includegraphics[scale=1.15]{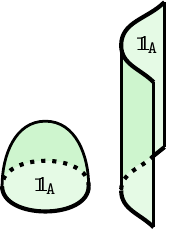}};
\endxy
\quad,\quad
G_2^{+}=
\xy
(0,0)*{\includegraphics[scale=1.15]{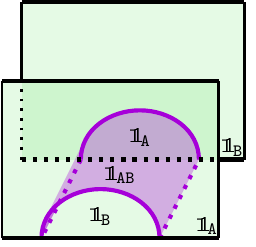}};
\endxy
\quad,\quad
G_2^{-}=
\xy
(0,0)*{\includegraphics[scale=1.15]{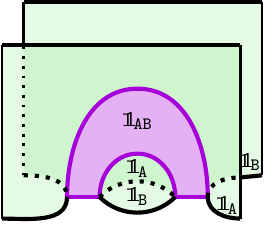}};
\endxy
\end{gather*}
\caption{The (R1), (R2+) and (R2--) foams.
}\label{fig:RMcomplexes}
\end{figure}

\begin{proof}
We first check the case 
of the Reidemeister $2$ moves. 
Using \fullref{rem:gradhom} and 
working in $\Komh(\foam{N})$, it is easy to see that the shown foams 
$F_2^\pm$ and $G_2^\pm$ (without their 
idempotent decoration) are uniquely determined up 
to a complex scalar by their degrees. 
This scalar is non-zero since these foams become 
invertible in $\FoamS{N}$. It follows that 
$F_2^+$ and $G_2^+$ are mutually inverse 
up to a scalar $\epsilon$, which is 
determined by Relation \eqref{eq:idem-disc}. This already holds in $\Foam{N}$. On the contrary,  
$F_2^-$ and $G_2^-$ only become mutually inverse up to 
the same sign $\epsilon$ when considered in $\FoamS{N}$ and decorated by idempotents as shown. This follows from the relations in \fullref{lemma:saddle-reverse}. Now we 
can rescale the Reidemeister $2$ homotopy 
equivalences in $\Foam{N}$ to obtain units as in the statement 
of the lemma.

Regarding Reidemeister $1$ moves, degree 
considerations in $\Foam{N}$ 
using \fullref{rem:gradhom} imply that the relevant 
foams are uniquely determined (up to a non-zero complex scalar) 
in the cases where we claim that no extra scalar appears. In 
the other cases, the foams carry a decoration and are determined 
(up to a unit) by the requirement to be a component of a chain map. 
More specifically, the decoration is the result of a 
neck-cut \eqref{eq:neckcut} and the scalar $\scalar{A}$ results 
from specializing it \eqref{eq:idem-bubble}. It follows that these 
decorated foams give mutual inverses in $\FoamS{N}$, and so we 
choose to scale the Reidemeister $1$ homotopy equivalences 
in $\Foam{N}$ accordingly. 
\end{proof}

Next, we describe the chain homotopy equivalences induced by Reidemeister $3$ moves of type (R3+), whose local model involves tangles with a cyclic sequence of boundary orientations given by three outward pointing arcs followed by three inward pointing arcs, see e.g. \eqref{eg:R3standard}. The remaining Reidemeister 3 moves, of type (R3--), have a boundary orientation sequence alternating between inward and outward pointing and will be dealt with at the end of this section.

\begin{lemma}\label{lemma:induced-RM3}
The homotopy equivalences induced 
by (R3+) moves are 
realized on simple resolutions by 
foams $F^+_3$ and $G^+_3$ 

\begin{equation}\label{eg:R3standard}
\xy
(0,0)*{\includegraphics[scale=1.15]{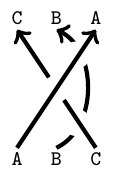}};
\endxy
\xrightarrow{F^+_3}
\xy
(0,0)*{\includegraphics[scale=1.15]{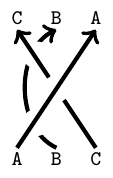}};
\endxy
\xrightarrow{G^+_3}
\xy
(0,0)*{\includegraphics[scale=1.15]{figs/fig-2-47}};
\endxy
\end{equation}
which can be represented by \fcds with 
at most two trivalent vertices. See 
\fullref{fig:RMcomplexes} for illustrations of these foams in the case $a\geq b\geq c$.
\end{lemma}

\begin{figure}[ht]
\begin{gather*}
F_3^{+}=
\xy
(0,0)*{\includegraphics[scale=1.15]{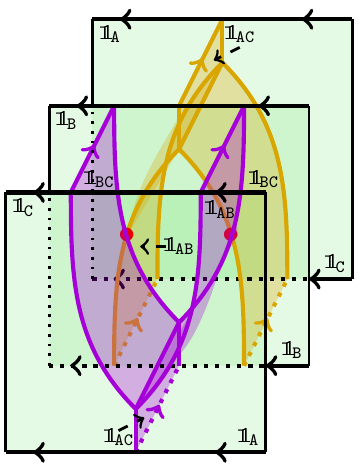}};
\endxy
\quad,\quad
G_3^{+}=
\xy
(0,0)*{\includegraphics[scale=1.15]{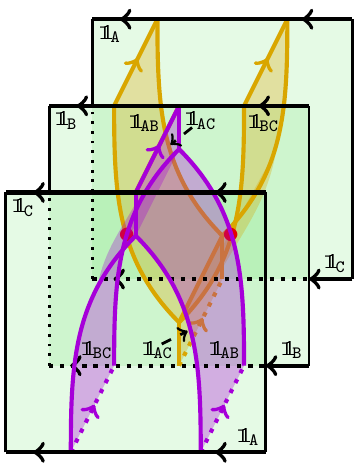}};
\endxy
\end{gather*}
\caption{Examples of (R3+) foams.}\label{fig:RM3standard}
\end{figure}

\begin{proof}
As in the proof of \fullref{lemma:induced-RM}, 
we deduce that Reidemeister $3$ homotopy equivalences 
restrict to invertible, idempotent-colored foams on simple 
resolutions. Degree considerations in $\Foam{N}$ determine 
these (up to a unit) to be foams that admit 
\fcds with at most two trivalent vertices. Note that 
they are the foams of lowest combinatorial complexity 
between these simple resolutions. Below we display \fcds 
for three labeling patterns and check that they represent 
mutually inverse foams in $\FoamS{N}$. As before, this 
determines our preferred scaling of the Reidemeister $3$ homotopy 
equivalences in $\Foam{N}$. 

For $a\geq b\geq c$, the composition $G_3^{+}\circ F_3^{+}$ simplifies to the identity foam as follows.
\[
\xy
(0,0)*{\includegraphics[scale=1.15]{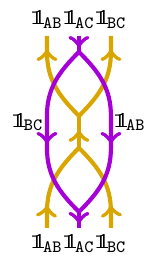}};
\endxy
=
\scalar{\AB,\BC}^{-1}
\xy
(0,0)*{\includegraphics[scale=1.15]{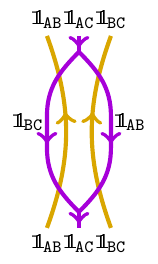}};
\endxy
=
\scalar{\BC,\AB}
\xy
(0,0)*{\includegraphics[scale=1.15]{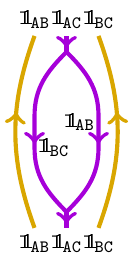}};
\endxy
=
\xy
(0,0)*{\includegraphics[scale=1.15]{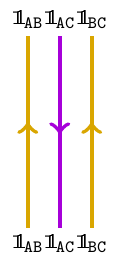}};
\endxy
\]
Here we have 
applied \eqref{eq:idem-KLR2} and \eqref{eq:idem-KLR}. In the case $a\geq c\geq b$ the foam $G_3^{+}\circ F_3^{+}$ simplifies as:
\begin{gather*}
\xy
(0,0)*{\includegraphics[scale=1.15]{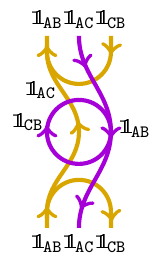}};
\endxy
=
\xy
(0,0)*{\includegraphics[scale=1.15]{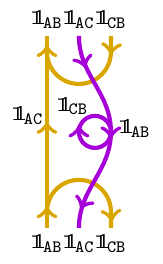}};
\endxy
=
(-1)^{b(c-b)}\!\!\!\!
\xy
(0,0)*{\includegraphics[scale=1.15]{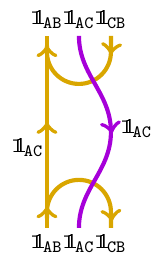}};
\endxy
=
\scalar{\wcolor{AC},\wcolor{CB}}
\xy
(0,0)*{\includegraphics[scale=1.15]{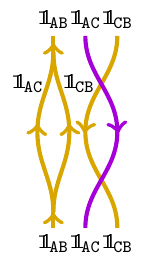}};
\endxy
=
\xy
(0,0)*{\includegraphics[scale=1.15]{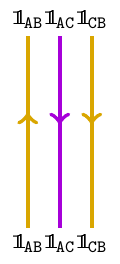}};
\endxy
\end{gather*}
Here we used \eqref{eq:MP-flow}, \eqref{eq:idem-KLR2}, \eqref{eq:idem-disc2} 
and \eqref{eq:saddle-reverse}. In the similar case $b\geq a\geq c$ we get:
\begin{gather*}
\xy
(0,0)*{\includegraphics[scale=1.15]{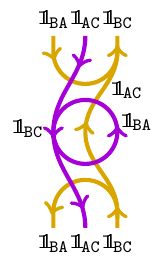}};
\endxy
=
\xy
(0,0)*{\includegraphics[scale=1.15]{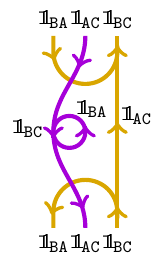}};
\endxy
=
(-1)^{c(b-a)}\!\!\!\!
\xy
(0,0)*{\includegraphics[scale=1.15]{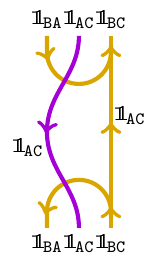}};
\endxy
=
\scalar{BA,AC}
\xy
(0,0)*{\includegraphics[scale=1.15]{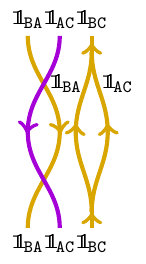}};
\endxy
=
\xy
(0,0)*{\includegraphics[scale=1.15]{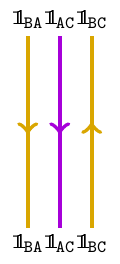}};
\endxy
\end{gather*}
The other compositions $F^+_3\circ G_3^+$ and all other 
cases can be checked to produce identity foams 
in a completely analogous way. Thus, we have shown that 
the (R3+) foams $F^+_3$ and $G^+_3$ 
are mutually inverse foams, which proves the statement.
\end{proof}

For Reidemeister $3$ moves with alternating 
boundary orientations, i.e. type (R3--), we use the following 
composite of (R2--), (R2+) and (R3+):
\begin{gather}\label{eq:R3hard}
\begin{aligned}
&\xy
(0,.75)*{\includegraphics[scale=1.2]{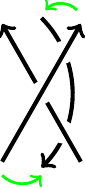}};
\endxy
\,
\leftrightsquigarrow
\,
\xy
(0,0)*{\includegraphics[scale=1.2]{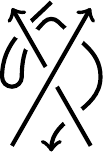}};
\endxy
\,
\leftrightsquigarrow
\,
\xy
(0,0)*{\includegraphics[scale=1.2]{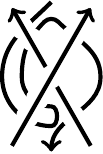}};
\endxy
\,
\leftrightsquigarrow
\,
\xy
(0,0)*{\includegraphics[scale=1.2]{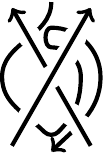}};
\endxy
\,
\leftrightsquigarrow
\,
\xy
(0,0)*{\includegraphics[scale=1.2]{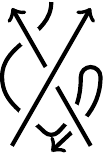}};
\endxy
\,
\leftrightsquigarrow
\,
\xy
(0,0)*{\includegraphics[scale=1.2]{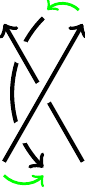}};
\endxy
\end{aligned}
\end{gather}

This is also to be interpreted as a 
template for variations of the (R3--) 
moves with different crossing types than 
the one shown. In all cases, the active 
strand (the one participating in all (R2) moves) 
is chosen to be the first strand encountered on 
the boundary when proceeding in the counterclockwise 
direction, starting from the boundary of the top 
strand. This is indicated by arrows in \eqref{eq:R3hard}. 
We denote the induced composite foams on simple resolutions by $F^-_3$ and $G^-_3$:
\begin{gather*}
\xy
(0,.75)*{\includegraphics[scale=1.15]{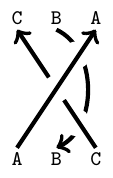}};
\endxy
\xrightarrow{F^-_3}
\xy
(0,0)*{\includegraphics[scale=1.15]{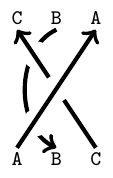}};
\endxy
\xrightarrow{G^-_3}
\xy
(0,0)*{\includegraphics[scale=1.15]{figs/fig-2-71}};
\endxy
\end{gather*}

\begin{example} In the case $a\geq b\geq c$, the foams $F^-_3$ and $G^-_3$ are given by the following compositions reading left-to-right and right-to-left respectively:
\begin{equation}
\label{eq:R3hard2}
\!\!\!
\xy
\xymatrixcolsep{1.12pc}
\xymatrix{
\xy
(0,0)*{\includegraphics[scale=1.15]{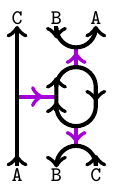}};
\endxy
\ar@<2pt>[r]^{F_2^-}
&
\xy
(0,0)*{\includegraphics[scale=1.15]{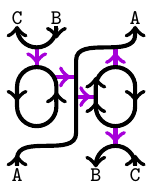}};
\endxy
\ar@<2pt>[r]^{F_2^+}
\ar@<2pt>[l]^{\tau G_2^-}
&
\xy
(0,0)*{\includegraphics[scale=1.15]{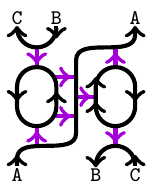}};
\endxy
\ar@<2pt>[r]^{F_3^+}
\ar@<2pt>[l]^{\varepsilon G_2^+}
&
\xy
(0,0)*{\includegraphics[scale=1.15]{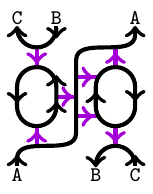}};
\endxy
\ar@<2pt>[r]^{G_2^+}
\ar@<2pt>[l]^{G_3^+}
&
\xy
(0,0)*{\includegraphics[scale=1.15]{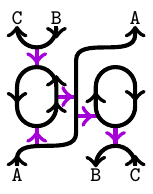}};
\endxy
\ar@<2pt>[r]^/.15cm/{G_2^-}
\ar@<2pt>[l]^{\varepsilon F_2^+}
&
\xy
(0,0)*{\includegraphics[scale=1.15]{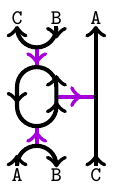}};
\endxy
\ar@<2pt>[l]^/-.15cm/{\tau F_2^-}
}
\endxy
\end{equation}
Here $\varepsilon$ and $\tau$ are signs
coming from our scaling conventions.
\end{example}

As for the other Reidemeister moves, the foams $F^-_3$ and $G^-_3$, which are defined on simple resolutions in $\FoamS{N}$, determine a particular scaling of the (R3--) homotopy equivalences in $\foam{N}$ that we henceforth adopt.

\begin{example}\label{example:this-is-easy!}
In the case where all labels on strands are equal, 
the simple resolutions of the tangle diagrams appearing 
in Reidemeister moves and the foams between them 
are especially simple: (R2+) and (R3+) moves 
take the form of identity foams between identity 
webs, and (R2--) are realized by cup- and cap-saddles. 
In the (R3--) move from \eqref{eq:R3hard2}, the purple $\purplebox$ edges disappear and the 
resulting foam is a \textit{monkey saddle} as in \cite[Figure 9]{BN1}. 
\end{example}

Above, we have determined a 
particular scaling for the Reidemeister move homotopy 
equivalences in $\Komh(\foam{N})$. Even though this process depends on the choice of a particular specialization $\alphS$, 
the result is quite rigid.

\begin{lemma}\label{lemma:integral}
(Integrality.) 
The rescaled Reidemeister move homotopy equivalences are integral.
\end{lemma}

\begin{proof}
Consider a particular Reidemeister move and denote by $f$ and $g$ the mutually inverse rescaled homotopy equivalences, which satisfy
\begin{equation}\label{eq:int-eq-1}
\id - f \circ g =  d \circ h + h \circ d,
\end{equation}
where $d$ is the differential in the domain of $g$, $h$ is a 
homotopy and all of these morphisms are built from foams with 
coefficients in $\C$. By the integral version of 
\fullref{theorem:wulink}, see \fullref{remark:integrality-2}, there 
also exist mutually 
inverse integral homotopy equivalences $f^{\prime}$ and $g^{\prime}$ as well 
as an integral homotopy $h^{\prime}$ such that 
\begin{equation}\label{eq:int-eq-2}
\id- f^{\prime} \circ g^{\prime} =  d \circ h^{\prime} + h^{\prime} \circ d.
\end{equation}
Above we have seen that the foams appearing in $f$ and $g$ 
on simple resolutions are already integral. This implies 
that $f^{\prime}$ and $g^{\prime}$ are integer multiples of $f$ and $g$, respectively.
 
Let $z_1, z_2\in \Z$ be such that $f^{\prime}= z_1 f$ and $g^{\prime} = z_2 g$. Substituting in \eqref{eq:int-eq-2} and subtracting from it a multiple of \eqref{eq:int-eq-1} gives
\[
(1-z_1 z_2) \id = d \circ (h^{\prime}-z_1 z_2 h) + (h^{\prime}-z_1 z_2 h)\circ d.
\] 
Since the domain of $g$ is not a contractible chain complex, 
this implies $z_1z_2=1$. Consequently, $z_1=z_2=\pm 1$ and so 
$f= \pm f^{\prime}$ and $g= \pm g^{\prime}$ are integral.
\end{proof}
	
\section{Functoriality}\label{sec:functorial}
\subsection{The canopolis of tangles and their cobordisms in 4-space}\label{subsec:fourdim}

In \fullref{subsec:homology-stuff} we have 
encountered the planar algebra $\coltangles$ of colored, oriented 
tangles. Now we extend it to a canopolis.

\begin{definitionn}\label{definition:coltangles-can}
Let $\coltangles$ be the canopolis determined 
by the following data. 
\smallskip
\begin{itemize}

\setlength\itemsep{.15cm}

\item The objects are 
given by colored, oriented tangle diagrams in disks $D$ with an 
ordering of the crossings. We regard such diagrams as actual colored, oriented tangles, embedded
in $D\times [0,1]$.

\item The morphisms (besides crossing reordering isomorphisms) are two-dimensional 
colored, oriented cobordisms between tangles, embedded 
in $D\times [0,1]\times [0,1]$. 
These are assumed to be cylindrical in a neighborhood 
of the boundary, with matching boundary orientation on the 
top boundary and opposite one on the bottom. We consider them modulo isotopy 
relative to the boundary.

\item The categorical composition is 
given by gluing cobordisms vertically.

\item The planar algebra 
operation by gluing horizontally and concatenating orders of crossings.\makeqedtri

\end{itemize}
\end{definitionn}

If a cobordism $C$ between tangles is in generic 
position, then the horizontal slices 
$C_z=C\cap (D\times [0,1]\times\{z\})$ give 
tangle diagrams, except for at most finitely 
many $z\in[0,1]$. At such a critical point, the tangle 
diagram $C_{z-\epsilon}$ transforms into the tangle 
diagram $C_{z+\epsilon}$ by a Morse
or a Reidemeister move, 
see \fullref{fig:RMsurfaces} for examples.

\begin{figure}[ht]
\begin{gather*}
\xy
(0,.5)*{\includegraphics[scale=1.15]{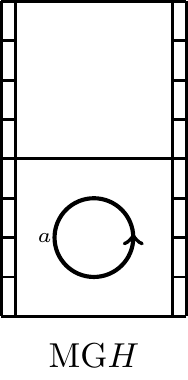}};
\endxy
\;,\;
\xy
(0,0)*{\includegraphics[scale=1.15]{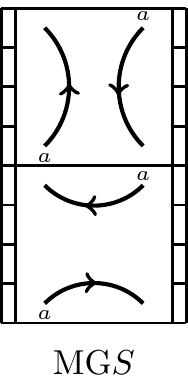}};
\endxy
\;,\;
\xy
(0,-.2)*{\includegraphics[scale=1.15]{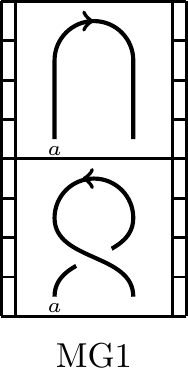}};
\endxy
\;,\;
\xy
(0,-.2)*{\includegraphics[scale=1.15]{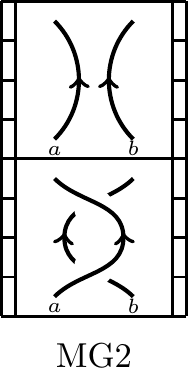}};
\endxy
\;,\;
\xy
(0,-.2)*{\includegraphics[scale=1.15]{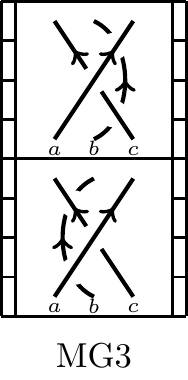}};
\endxy
\end{gather*}
\caption{Examples of movies of generating cobordisms.
}\label{fig:RMsurfaces}
\end{figure}

Between these critical values the 
diagrams $C_{z}$ differ 
only by planar isotopy. As a result, a cobordism $C$ 
in generic position can be 
represented by a 
\textit{movie} of tangle diagrams, whose consecutive 
frames show precisely the transformation of 
horizontal slices across a single critical value. Such movie 
presentations are not unique, but their 
ambiguity can be controlled, as we recall next. For 
this, we use a more rigid version of $\coltangles$ 
in which movie moves represent cobordisms uniquely:

\begin{definition}\label{definition:coltangle-free}
Let $\coltanglesfree$ be the canopolis with 
the same objects and canopolis operations as 
$\coltangles$, but with morphisms given by 
cobordisms rigidly built from Morse and 
Reidemeister type cobordism generators, 
without allowing isotopy.
\end{definition}

By definition, morphisms in $\coltanglesfree$ can be uniquely 
represented by movies of tangle diagrams, whose frames 
differ by a single Morse or Reidemeister type cobordism 
(together with a crossing reordering). The ambiguity 
of this presentation for cobordisms in $\coltangles$, 
where we allow isotopies, is described in the following proposition.

\makeautorefname{figure}{Figures}

\begin{proposition}\label{prop:kernel-mm}
The rigidly built cobordisms in $\coltanglesfree$, which are identified under the projection
$\coltanglesfree \to \coltangles$, are precisely those which can be related by finite sequences of the relations shown and explained in 
\fullref{fig:movie1}, \ref{fig:movie2} 
and \ref{fig:movie3} as well as their 
variations obtained from changing orientations and the height of strands.  
\end{proposition}

\makeautorefname{figure}{Figure}

These relations between cobordism movies
are 
colored, oriented versions the 
\textit{movie moves} as 
presented by Carter--Saito in \cite[Chapter 2]{CSbook}, but numbered as in \cite[Section 8]{BN1}. 

\begin{proof} 
By forgetting colors and orientations, any isotopy of cobordisms in 
$\coltangles$ is an isotopy of 
cobordisms as studied in \cite[Chapter 2]{CSbook}. 
Hence, it can be written as a finite sequence of the movie moves therein. 
Remembering the coloring data and the orientation again, we see that 
the original isotopy in $\coltangles$ can be written as a finite 
sequence of the colored, oriented movie moves.
\end{proof}

\begin{figure}[ht]
\begin{gather*}
\xy
(0,.75)*{\includegraphics[scale=1.15]{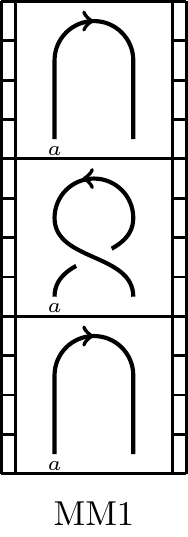}};
\endxy
\;,\;
\xy
(0,0)*{\includegraphics[scale=1.15]{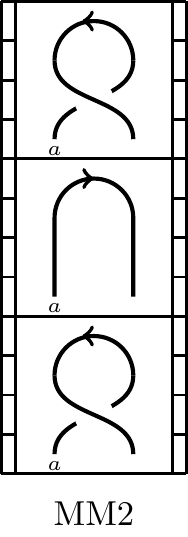}};
\endxy
\;,\;
\xy
(0,0)*{\includegraphics[scale=1.15]{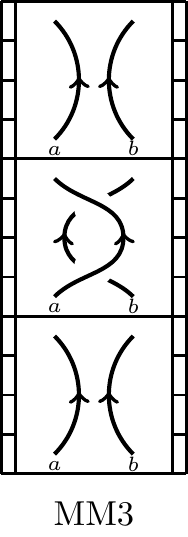}};
\endxy
\;,\;
\xy
(0,0)*{\includegraphics[scale=1.15]{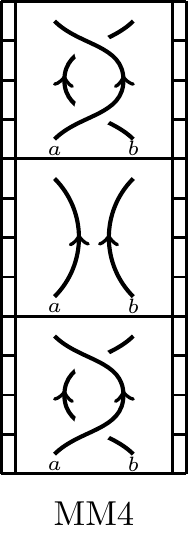}};
\endxy
\;,\;
\xy
(0,0)*{\includegraphics[scale=1.15]{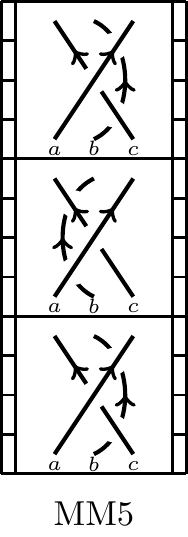}};
\endxy
\end{gather*}
\caption{The reversible 
movie moves, which say that doing and undoing Reidemeister moves is equivalent to doing nothing.
}\label{fig:movie1}
\end{figure}

\begin{figure}[ht]
\begin{gather*}
\xy
(0,.8)*{\includegraphics[scale=1.15]{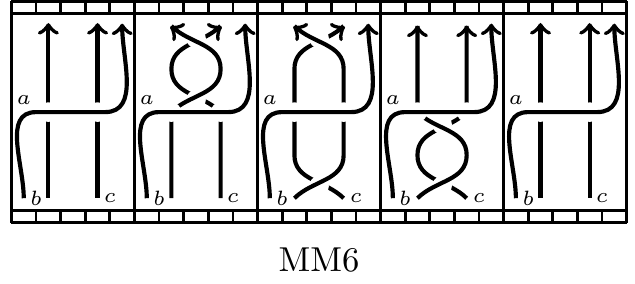}};
\endxy
\;,\;
\xy
(0,0)*{\includegraphics[scale=1.15]{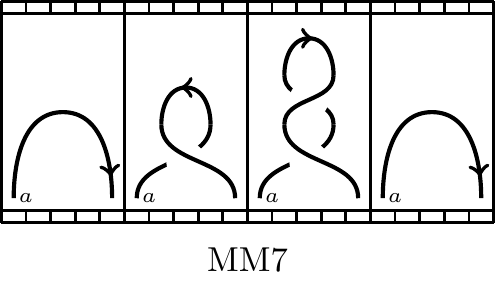}};
\endxy
\\
\xy
(0,.8)*{\includegraphics[scale=1.15]{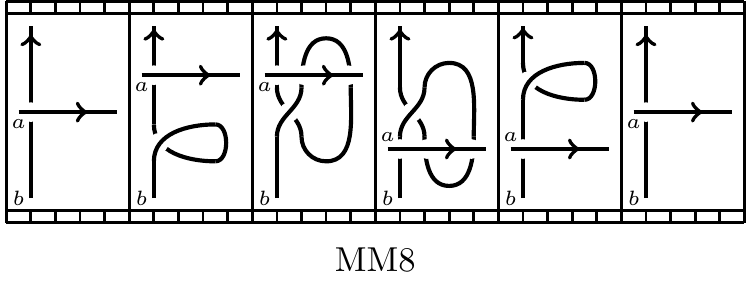}};
\endxy
\;,\;
\xy
(0,0)*{\includegraphics[scale=1.15]{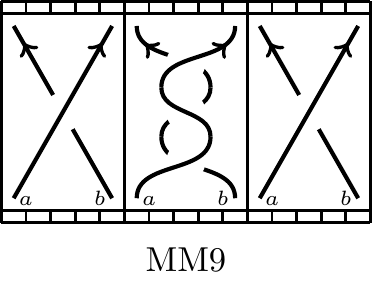}};
\endxy
\\
\xy
(0,0)*{\includegraphics[scale=1.15]{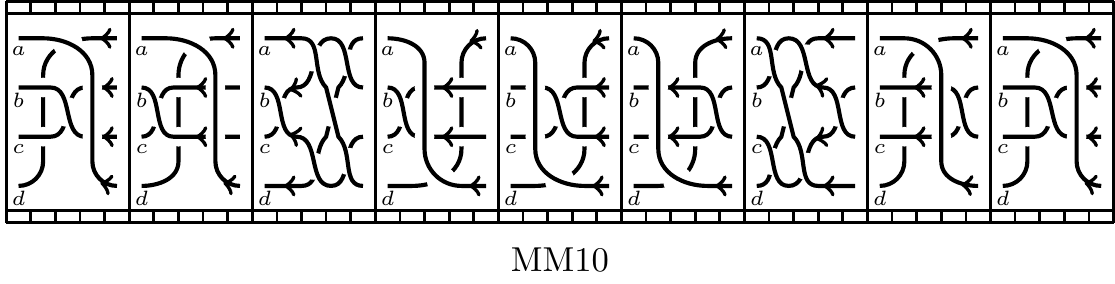}};
\endxy
\end{gather*}
\caption{The reversible movie moves, which show movies equivalent to doing nothing if read left- or rightwards.
}\label{fig:movie2}
\end{figure}

\begin{figure}[ht]
\begin{gather*}
\xy
(0,.8)*{\includegraphics[scale=1.15]{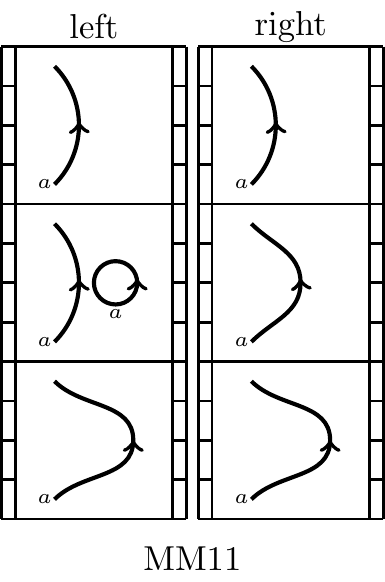}};
\endxy
\,,\,
\xy
(0,0)*{\includegraphics[scale=1.15]{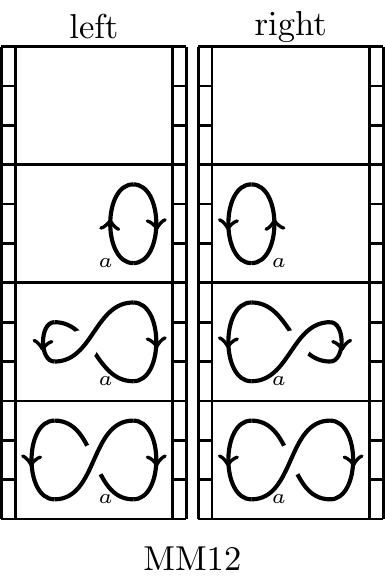}};
\endxy
\,,\,
\xy
(0,0)*{\includegraphics[scale=1.15]{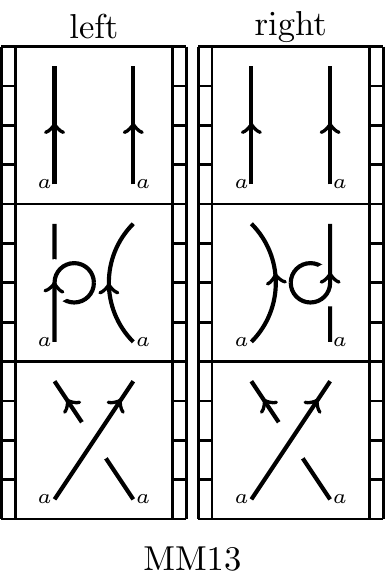}};
\endxy\\
\xy
(0,.8)*{\includegraphics[scale=1.15]{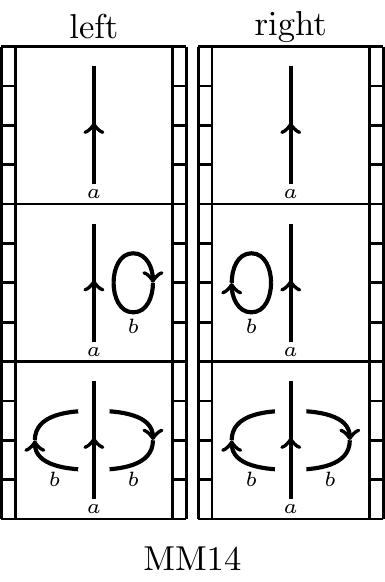}};
\endxy
\;,\;
\xy
(0,0)*{\includegraphics[scale=1.15]{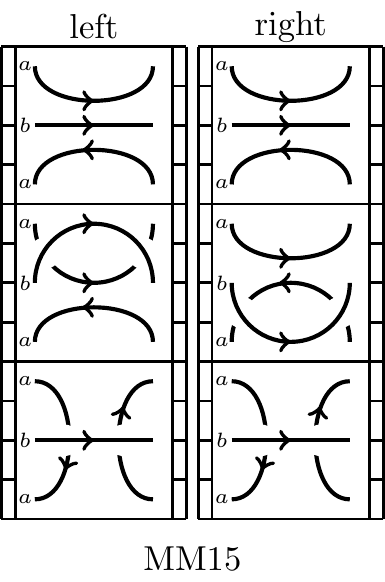}};
\endxy
\end{gather*}
\caption{The non-reversible movie moves. 
The columns show equivalent movies when read down- or upwards.
}\label{fig:movie3}
\end{figure}

Now, we extend the definition of $\Hgen{\cdot}$ 
to a canopolis functor
\[
\Hgen{\cdot}\colon\coltanglesfree\to\Komh(\foam{N})
\]
by assigning (homotopy classes of) chain 
maps to the generating cobordisms in 
\fullref{fig:RMsurfaces}. To cups, saddles and caps, we associate 
the chain maps given by acting with the corresponding foam 
from \eqref{eq:foamgen1} on every chain group. To Reidemeister 
cobordisms we assign the corresponding Reidemeister homotopy 
equivalences with the scaling identified in 
\fullref{subsec:simple-complexes}.
Finally, the crossing reordering isomorphisms are 
sent to the corresponding reordering isomorphisms in $\Komh(\foam{N})$.

\begin{remark}
By construction, $\Hgen{\cdot}$ assigns 
homotopy inverse chain maps to Reidemeister moves 
and their inverses. The chain maps assigned to 
movies from \fullref{fig:movie1} are 
thus, homotopic to the corresponding identity 
chain maps and we express this by saying that 
$\Hgen{\cdot}$ respects the movie moves 
$\MM{1}$--$\MM{5}$, see 
\fullref{theorem:wulink}.
\end{remark}

The main goal of this paper is to show 
that $\Hgen{\cdot}$ assigns equal (homotopy classes of) chain maps 
to cobordism movies which are related by one 
of the remaining movie moves $\MM{6}$--$\MM{15}$, and thus:

\begin{theoremn}\label{theorem:funclink}\textbf{(Functoriality.)}
The canopolis functor $\Hgen{\cdot}\colon\coltanglesfree\to\Komh(\foam{N})$ 
factors through a functor
\[
\Hgen{\cdot}\colon\coltangles\to\Komh(\Foam{N}).
\hspace{5.55cm}
\raisebox{-.01cm}{\makeqed}
\hspace{-5.55cm}
\]
\end{theoremn}

The proof of this theorem is contained in the following 
two sections. First we use an abstract argument to see that 
$\Hgen{\cdot}$ respects movie moves up to scalars in $\C$. 
Finally, in \fullref{subsec:checkmm}, we compute these 
scalars via the $\alphS$-deformation and find them to be equal to one.

\subsection{Functoriality up to scalars}\label{subsec:uptounits}
We follow Bar-Natan's strategy \cite[Section 8]{BN1}. 
The following Lemma contains the key idea of this strategy.

\begin{lemma}\label{lemma:BNsimple} 
Let $T_D$ be a diagram of a tangle which 
is isotopic (without fixing boundary points) 
to a trivial tangle. Then the space of  
degree zero endomorphisms 
of $\Hgen{T_D}\in\Komh(\Foam{N})$ is one-dimensional over $\C$.
\end{lemma}

\begin{proof} 
Since planar composition with crossings on the 
boundary is an invertible operation, the problem 
reduces to the case where $T_D$ is a trivial tangle 
diagram supported in homological degree zero, 
see \cite[Lemmas 8.7, 8.8 and 8.9]{BN1}. For such we have already observed in 
\fullref{rem:gradhom} that the space of 
degree zero endomorphisms is one-dimensional.
\end{proof}

\begin{proposition}\label{proposition:funclink}
The movie moves hold up to scalars in $\Komh(\Foam{N})$.
\end{proposition}

\begin{proof}
For the reversible movie moves 
from \fullref{fig:movie2}, 
it follows immediately from 
\fullref{lemma:BNsimple} that the chain map associated 
to the complicated movie agrees up to a complex scalar 
with the identity chain map. Additionally, since it is 
a composition of homotopy equivalences, it is invertible, 
and thus, non-zero.

For the non-reversible movies 
from \fullref{fig:movie3} we first check that the morphism spaces 
between the initial and final frames of the 
movies are one-dimensional over $\C$ in the relevant 
degrees. For $\MM{11}$ we have already 
seen this and $\MM{13}$ is treated analogously after 
expanding the crossing into a chain complex. For the others 
movie moves, one can cut scenes 
from the movie during which only homotopy equivalences 
happen. It remains to analyze frames differing by Morse moves. The corresponding morphisms spaces 
have been identified to be one-dimensional in 
\fullref{rem:gradhom}. This 
shows that the chain maps associated to the two sides 
of such a movie move agree up to a scalar. 
\end{proof}

The non-reversible 
movie moves might hold only trivially, i.e. 
both sides might represent the zero chain map. 
However, in the next section we shall see that these maps are never zero.

\begin{remark}\label{remark:integrality-3} 
(Integrality.) 
The proof of 
functoriality up to scalars over $\Z$ is completely analogous.
Here we use that all morphism spaces are free $\symz(\Eq)$-modules, 
see \fullref{remark:integrality-2}.
\end{remark}

\subsection{Computing the scalars}\label{subsec:checkmm}
It remains to compute the scalars by which the 
chain maps associated to the two sides of a movie 
move $\MM{6}$--$\MM{15}$ might differ. We check 
this on the $\alphS$-deformation. 

\begin{lemma}\label{lemma:from-def-to-real2}
If the movie moves hold non-trivially in $\Komh(\FoamSH)$, 
then they hold non-trivially in $\Komh(\Foam{N})$.
\end{lemma}

\begin{proof}
Since we already know that the movie moves hold in $\Komh(\Foam{N})$ 
up to scalars in $\C$, these scalars can be computed after 
specializing to $\Komh(\FoamS{N})$ and 
further embedding in $\Komh(\FoamSH)$. The assumption guarantees 
that all these scalars are all equal to one.
\end{proof}

\begin{lemma}\label{lemma:from-def-to-real}
If the movie moves hold non-trivially on simple resolutions, then they 
hold non-trivially in $\Komh(\FoamSH)$.
\end{lemma}

\begin{proof}
By specializing \fullref{proposition:funclink}, we see 
that, up scalars in $\C$, the movie moves hold in $\Komh(\FoamSH)$. 
Next, if two chain maps agree up to a scalar, this scalar can be computed by 
comparing the chain maps restricted to a subcomplex where one of 
them acts non-trivially. Here we 
choose the subcomplex in $\Komh(\FoamSH)$ given by a simple resolution. 
\end{proof}

In the following, we shall compute the 
chain maps appearing in the movie moves 
$\MM{6}$--$\MM{15}$ when restricted to simple 
resolutions. We will find that the chain maps 
on both sides of each movie move agree and are 
non-zero. This will satisfy the assumption in 
\fullref{lemma:from-def-to-real} and consequently 
\fullref{lemma:from-def-to-real2} and, thus, 
complete the proof of \fullref{theorem:funclink}.

\begin{remark}\label{remark:simple-res}
The simple resolutions of a tangle diagram are 
invariant under interchanging any number of 
positive and negative crossing (up to $q$-degree shifts which 
we ignore here). Moreover, thanks to 
\fullref{convention:Koszul} and the fact that 
simple resolutions of crossings are supported in even 
homological degree (namely zero), 
we do not need to consider reordering isomorphisms, 
since they all act by the identity.
Altogether, this allows us to reduce the number of 
different variants of movie moves which we need to check. 

However, the maps associated to Reidemeister 
cobordisms are usually dependent on the relative 
sizes of the labels on strands as well as the height 
ordering of strands. In checking the movie moves, we 
shall display one such variant for each choice and comment on the others.
\end{remark}

\begin{lemma}\label{lemma:MM6}
The movie move $\MM{6}$ holds on simple resolutions.
\end{lemma}

\begin{proof}
$\MM{6}$ has two essentially 
different versions depending on the relative 
orientation of the strands between 
which Reidemeister $2$ moves happen. 
The first version involves two (R2+) and (R3+) moves. 
For $a\geq b\geq c$ it is given on simple resolutions as:
\begin{gather*}
\xy
\xymatrix{
\xy
(0,1)*{\includegraphics[scale=1.15]{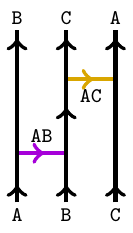}};
\endxy
\ar@<2pt>[r]^{\epsilon F_2^+}
&
\xy
(0,0)*{\includegraphics[scale=1.15]{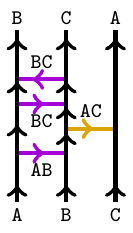}};
\endxy
\ar@<2pt>[r]^{F_3^{+}}
\ar@<2pt>[l]^{G_2^+}
&
\xy
(0,0)*{\includegraphics[scale=1.15]{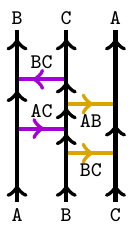}};
\endxy
\ar@<2pt>[r]^{G_3^{+}}
\ar@<2pt>[l]^{G_3^{+}}
&
\xy
(0,0)*{\includegraphics[scale=1.15]{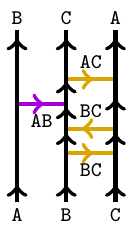}};
\endxy
\ar@<2pt>[r]^{G_2^+}
\ar@<2pt>[l]^{F_3^{+}}
&
\xy
(0,0)*{\includegraphics[scale=1.15]{figs/fig-3-22}};
\endxy
\ar@<2pt>[l]^{\epsilon F_2^+}
}
\endxy
\end{gather*}
Here $\epsilon=(-1)^{c(b-c)}$. The composite foam for reading left-to-right is as follows.
\[
(-1)^{c(b-c)}
\xy
(0,0)*{\includegraphics[scale=1.15]{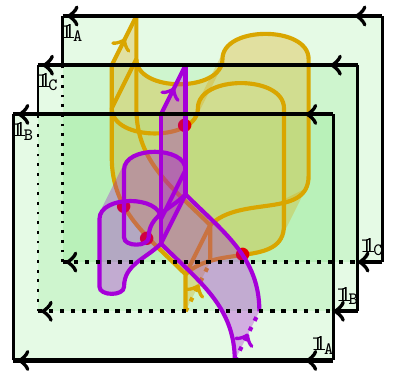}};
\endxy
\]
We compute that this foam is equal to the identity foam by simplifying its \fcd using the relations obtained in \fullref{subsec:special}.
\begin{gather*}
(-1)^{c(b-c)}\!\!
\xy
(0,0)*{\includegraphics[scale=1.15]{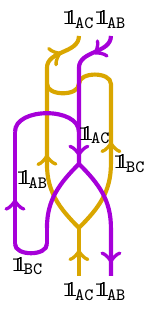}};
\endxy
\!\!=\! 
(-1)^{c(b-c)}
\xy
(0,0)*{\includegraphics[scale=1.15]{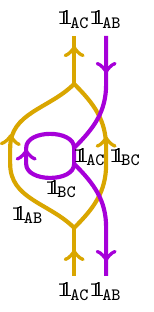}};
\endxy
\!\!=\!
\frac{(-1)^{c(b-c)}}{\scalar{BC,AB}}\!
\xy
(0,0)*{\includegraphics[scale=1.15]{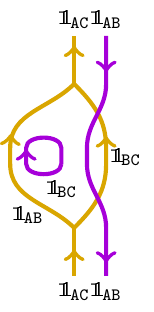}};
\endxy
\!\!=\!
\xy
(0,0)*{\includegraphics[scale=1.15]{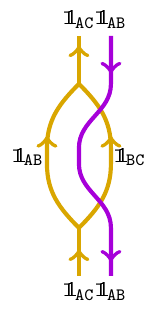}};
\endxy
\!\!=\!
\xy
(0,0)*{\includegraphics[scale=1.15]{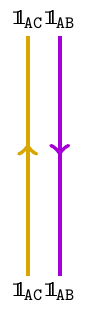}};
\endxy
\end{gather*}
Here we  have used \eqref{eq:MP-flow}, \eqref{eq:idem-KLR2}, \eqref{eq:idem-disc2} 
and \eqref{eq:idem-KLR}. Since this movie move is composed out of homotopy 
equivalences, its inverse right-to-left read version 
also gives the identity on simple resolutions. The 
cases of relative orderings of colors differing 
from $a\geq b\geq c$ and, more generally, all other 
versions of $\MM{6}$ which involve (R2+) moves are proved analogously.

The second version of the movie move $\MM{6}$ not 
only involves two (R2--) moves, but 
also an (R3--) move, whose associated chain 
map we have defined as a composition of 
one (R3+) and several (R2$\pm$) chain maps. The 
composition of these maps can be read off as 
the outer cycle of (R3+) and (R2$\pm$) chain 
maps in the following diagram, starting at the 
(marked) tangle diagram on the left.

\begin{align*}
\xy
(-10,0)*{
\begin{tikzpicture}[anchorbase,scale=1.15]
		\draw [dotted] (-.9,-1) to (.9,-1);
		\node at (0,0) {\includegraphics[scale=1.15]{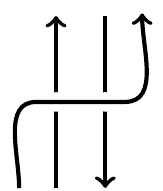}}; 
\end{tikzpicture}
};
(13,20)*{
\includegraphics[scale=1.15]{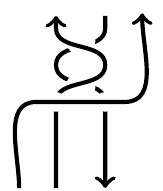}
};
(55,20)*{
\includegraphics[scale=1.15]{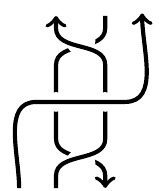}
};
(73,0)*{
\includegraphics[scale=1.15]{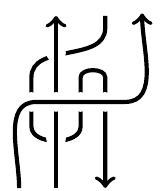}
};
(97,20)*{
\includegraphics[scale=1.15]{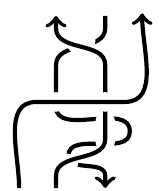}
};
(120,0)*{
\includegraphics[scale=1.15]{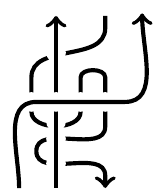}
};
(97,-20)*{
\includegraphics[scale=1.15]{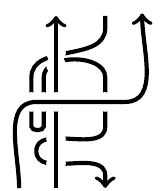}
};
(37,0)*{
\includegraphics[scale=1.15]{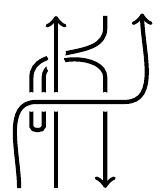}
};
(55,-20)*{
\includegraphics[scale=1.15]{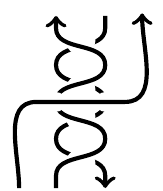}
};
(13,-20)*{
\includegraphics[scale=1.15]{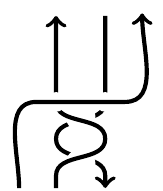}
};
(-2,17)*{
\begin{tikzpicture}[anchorbase,scale=.5]
		\draw [red,thick, ->] (-1,-1)to[out=90,in=180](1,1);
		\draw [thick, dashed, ->] (-1.35,-1)to[out=90,in=180](1,1.35);
		\draw [thick, dashed, <-] (-.65,-1)to[out=90,in=180](1,.65);
\end{tikzpicture}
};
(35.5,22)*{
\begin{tikzpicture}[anchorbase,scale=.5]
		\draw [red, thick, ->] (-1.5,0)to (1.5,0); 
\end{tikzpicture}
};
(75.5,22)*{
\begin{tikzpicture}[anchorbase,scale=.5]
		\draw [red, thick, ->] (-1.5,0)to (1.5,0); 
\end{tikzpicture}
};
(112,17)*{
\begin{tikzpicture}[anchorbase,scale=.5]
		\draw [red, thick, ->] (-1,1) to[out=0,in=90](1,-1);
\end{tikzpicture}
};
(112,-17)*{
\begin{tikzpicture}[anchorbase,scale=.5]
		\draw [red, thick, ->] (1,1)to[out=270,in=0](-1,-1);
\end{tikzpicture}
};
(75.5,-22)*{
\begin{tikzpicture}[anchorbase,scale=.5]
		\draw [red, thick, <-] (-1.5,0)to (1.5,0); 
\end{tikzpicture}
};
(35.5,-22)*{
\begin{tikzpicture}[anchorbase,scale=.5]
		\draw [red, thick, <-] (-1.5,0)to (1.5,0); 
\end{tikzpicture}
};
(-2,-17)*{
\begin{tikzpicture}[anchorbase,scale=.5]
		\draw [red, thick, ->] (1,-1)to[out=180,in=270](-1,1);
\end{tikzpicture}
};
(35.5,24)*{
\begin{tikzpicture}[anchorbase,scale=.5]
		\draw [thick, dashed, ->] (-1.5,0)to (1.5,0); 
\end{tikzpicture}
};
(75.5,24)*{
\begin{tikzpicture}[anchorbase,scale=.5]
		\draw [thick, dashed, ->] (-1.5,0)to (1.5,0); 
\end{tikzpicture}
};
(75.5,20)*{
\begin{tikzpicture}[anchorbase,scale=.5]
		\draw [thick, dashed, <-] (-1.5,0)to (1.5,0); 
\end{tikzpicture}
};
(64.5,11)*{
\begin{tikzpicture}[anchorbase,scale=.5]
		\draw [thick, dashed, ->] (-.7,.7)to (.7,-.7); 
\end{tikzpicture}
};
(55,0)*{
\begin{tikzpicture}[anchorbase,scale=.5]
		\draw [thick, dashed, <-] (-1.5,0)to (1.5,0); 
\end{tikzpicture}
};
(24.5,11)*{
\begin{tikzpicture}[anchorbase,scale=.5]
		\draw [thick, dashed, <-] (-.7,.7)to (.7,-.7); 
\end{tikzpicture}
};
\endxy
\end{align*}
Our task is to see that this cycle is 
equivalent to the identity chain map on 
the complex associated to the left tangle 
diagram. First note 
that the final (R2--) move in the cycle 
far-commutes with the four moves preceding 
it. The resulting composition is indicated 
by dashed arrows, which now involve the middle 
tangle diagrams. Next, observe that the first and 
last step in the dashed cycle are inverse (R2--) 
moves. The detour through the tangle diagram on 
the left can thus, be cut from the dashed cycle. 
Similarly, there is another detour through the 
tangle diagram on the top right, which can be cut 
out by $\MM{9}$ (which is proved 
independently in \fullref{lemma:MM9}). 
It remains to confirm that there is no monodromy around the dashed rhombus. 
After omitting the top left crossing from the diagrams, 
which does not participate in these moves, 
we recognize the required check as a comparison of 
two ways of performing a (R3+) move. 
\[
\xy
(0,0)*{\includegraphics[scale=1.15]{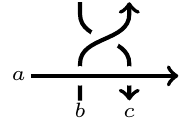}};
\endxy
\leftrightsquigarrow
\xy
(0,0)*{\includegraphics[scale=1.15]{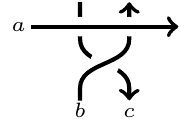}};
\endxy
\leftrightsquigarrow
\xy
(0,0)*{\includegraphics[scale=1.15]{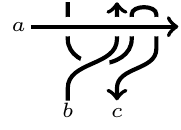}};
\endxy
\leftrightsquigarrow
\xy
(0,0)*{\includegraphics[scale=1.15]{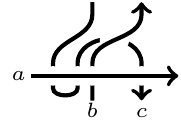}};
\endxy
\leftrightsquigarrow
\xy
(0,0)*{\includegraphics[scale=1.15]{figs/fig-3-42}};
\endxy
\]
The absence of monodromy can be checked 
on simple resolutions, e.g. for $a\geq b\geq c$:
\[
\xy
\xymatrix{
\xy
(0,1)*{\includegraphics[scale=1.15]{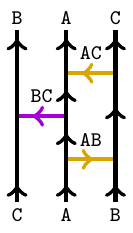}};
\endxy
\ar@<2pt>[r]^{F_3^+}&
\ar@<2pt>[l]^{G_3^+}
\xy
(0,0)*{\includegraphics[scale=1.15]{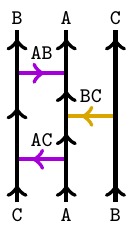}};
\endxy
\ar@<2pt>[r]^{\epsilon F_2^+}&
\ar@<2pt>[l]^{G_2^+}
\xy
(0,0)*{\includegraphics[scale=1.15]{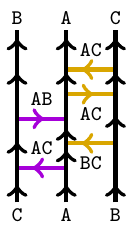}};
\endxy
\ar@<2pt>[r]^{F_3^+}&
\ar@<2pt>[l]^{G_3^+}
\xy
(0,0)*{\includegraphics[scale=1.15]{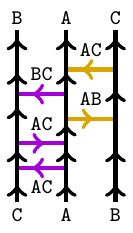}};
\endxy
\ar@<2pt>[r]^{\epsilon G_2^+}&
\ar@<2pt>[l]^{F_2^+}
\xy
(0,0)*{\includegraphics[scale=1.15]{figs/fig-3-46}};
\endxy
}
\endxy
\]
Again, $\epsilon=(-1)^{c(a-c)}$ 
(which appears twice and thus, does not contribute), and
we have rotated the webs to have boundary orientations 
pointing upward. A calculation
using \fcds
\begin{gather*}
\,
\xy
(0,0)*{\includegraphics[scale=1.15]{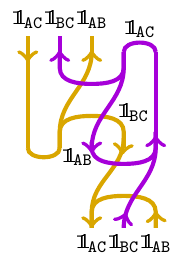}};
\endxy
=
\xy
(0,0)*{\includegraphics[scale=1.15]{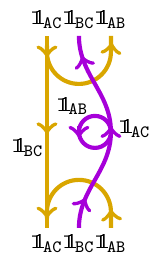}};
\endxy
=
(-1)^{c(a-b)}
\xy
(0,0)*{\includegraphics[scale=1.15]{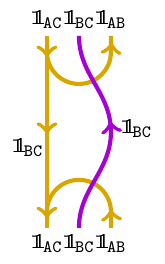}};
\endxy
=
\scalar{\AB,\BC}
\xy
(0,0)*{\includegraphics[scale=1.15]{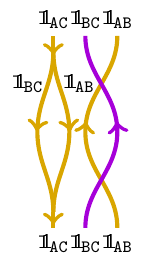}};
\endxy
=
\xy
(0,0)*{\includegraphics[scale=1.15]{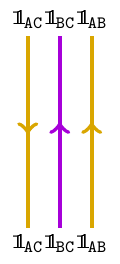}};
\endxy
\end{gather*}
verifies 
that the composite from left-to-right, respectively right-to-left, is the identity. 
Here we have used isotopy and \eqref{eq:MP-flow} first, then \eqref{eq:idem-KLR2} and 
\eqref{eq:idem-disc2}, then \eqref{eq:saddle-reverse} followed by \eqref{eq:idem-KLR2}, 
and finally \eqref{eq:MP-flow}. (We stress that some appearing scalars simplify due 
to \fullref{lemma:scalars}.)

Note that all other variants of $\MM{6}$ involving two (R2--) moves 
can be proven similarly using the short-cut strategy from above. 
Variants with a relative orderings of colors differing
from $a\geq b\geq c$ can be treated analogously.
\end{proof}

\begin{lemma}\label{lemma:MM7}
The movie move $\MM{7}$ holds on simple resolutions.
\end{lemma}

\begin{proof}
For the version displayed in \fullref{fig:movie2}, 
we get the following behavior on simple resolutions:
\[
\scalar{A}\;
\xy
(0,0)*{\includegraphics[scale=1.15]{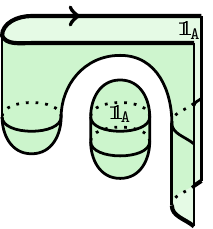}};
\endxy
\colon
\xy
\xymatrix{
\xy
(0,0)*{\includegraphics[scale=1.15]{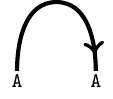}};
\endxy
\ar@<2pt>[r]^{F_1}
&
\xy
(0,0)*{\includegraphics[scale=1.15]{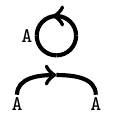}};
\endxy
\ar@<2pt>[r]^{\scalar{A}F_1}
\ar@<2pt>[l]^{\scalar{A}G_1}
&
\xy
(0,0)*{\includegraphics[scale=1.15]{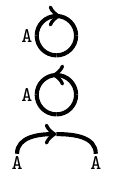}};
\endxy
\ar@<2pt>[r]^{G_2^-}
\ar@<2pt>[l]^{G_1}
&
\xy
(0,0)*{\includegraphics[scale=1.15]{figs/fig-3-51}};
\endxy
\ar@<2pt>[l]^{F_2^-}
}
\endxy
\]
The foam displayed on the left gives the 
left-to-right composition of chain maps on simple resolutions. 
The idempotent-colored sphere 
cancels with the scalar $\scalar{A}$ 
by \eqref{eq:idem-bubble}. The remaining foam is isotopic 
to the identity. The other 
variants have analogous proofs.
\end{proof}

\begin{lemma}\label{lemma:MM8}
The movie move $\MM{8}$ holds on simple resolutions.
\end{lemma}

\begin{proof}
The version of this movie move that is 
displayed in \fullref{fig:movie2} has the 
following behavior on simple resolutions if $a\geq b$:
\begin{gather*}
\xy
\xymatrix{
\xy
(0,0)*{\includegraphics[scale=1.15]{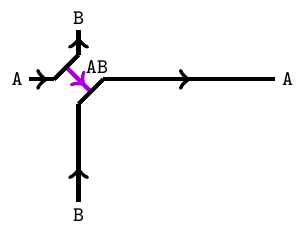}};
\endxy
\ar@<2pt>[r]^{\scalar{B}F_1}
&
\;\;
\xy
(0,0)*{\includegraphics[scale=1.15]{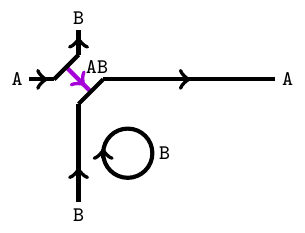}};
\endxy
\ar@<2pt>[r]^{\epsilon F_2^{+}}
\ar@<2pt>[l]^{G_1}
&
\xy
(0,0)*{\includegraphics[scale=1.15]{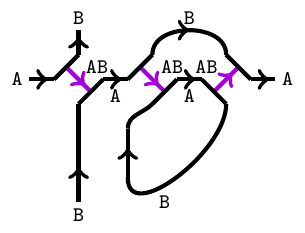}};
\endxy
\ar@<2pt>[l]^{G_2^{+}}
\ar@<2pt>[d]^{G_3^{+}}
\\
\xy
(0,0)*{\includegraphics[scale=1.15]{figs/fig-3-54}};
\endxy
\ar@<-2pt>[r]_{\scalar{B}F_1}
&
\xy
(0,0)*{\includegraphics[scale=1.15]{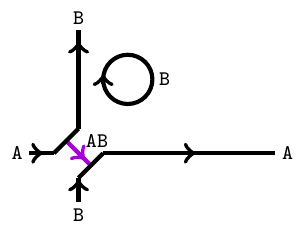}};
\endxy
\ar@<-2pt>[l]_{G_1}
\ar@<-2pt>[r]_{\epsilon F_2^{-}}
&
\xy
(0,0)*{\includegraphics[scale=1.15]{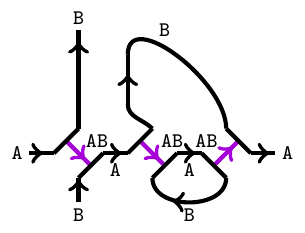}};
\endxy
\ar@<2pt>[u]^{F_3^{+}}
\ar@<-2pt>[l]_{G_2^{-}}
}
\endxy
\end{gather*}
Here $\epsilon=(-1)^{b(a-b)}$. Note that, due to the involvement of only two different labels, the foams $F_3^+$ and $G_3^+$ are trivial. The composite, thus, takes the following form.
\[
(-1)^{b(a-b)}\scalar{B}\;
\xy
(0,0)*{\includegraphics[scale=1.15]{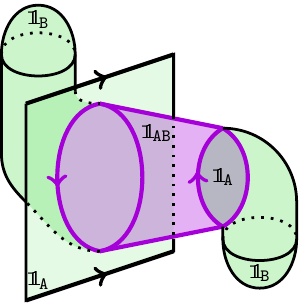}};
\endxy
\]
In this illustration we have omitted a trivial 
part to the foam. The sign $(-1)^{b(a-b)}$ and 
the scalar $\scalar{B}$ precisely cancel with 
the scalars appearing in \eqref{eq:idem-disc} 
and \eqref{eq:idem-bubble} when we delete the 
purple $\purplebox$ annulus and remove the resulting $\idem{B}$-colored 
sphere on the right hand side of the diagram. The result 
is isotopic to an identity foam. 

The reversed reading direction, the case $b\geq a$ and all other 
variations obtained by changing orientations are proved analogously. In each case, a 
foam of the above type (possibly with changed orientations and 
colorings) is simplified to give scalars as above. These 
precisely cancel with the normalizations of the Reidemeister 
foams since one always encounters an (R1) foam and its 
inverse, one essentially trivial (R3+) foam as well as 
(R2+) and (R2--) foams in inverse pairs.
\end{proof}

\begin{lemma}\label{lemma:MM9}
The movie move $\MM{9}$ holds on simple resolutions.
\end{lemma}

\begin{proof}
This move has essentially two variants, 
which consist of either two (R2+) or 
(R2--) moves. For $a\geq b$ they take the following form on simple resolutions:
\begin{align*}
\xy
\xymatrix{\!\!
\xy
(0,0)*{\includegraphics[scale=1.15]{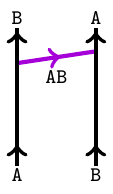}};
\endxy\!\!
\ar@<2pt>[r]^{\epsilon^{\prime} F_2^+}
&\!\!
\xy
(0,0)*{\includegraphics[scale=1.15]{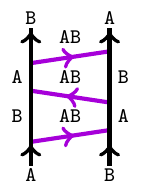}};
\endxy\!\!
\ar@<2pt>[r]^/.1cm/{\epsilon^{\prime} G_2^+}
\ar@<2pt>[l]^{\epsilon G_2^+}
&\!\!
\xy
(0,0)*{\includegraphics[scale=1.15]{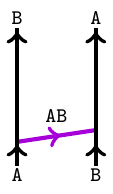}};
\endxy\!\!
\ar@<2pt>[l]^/-.05cm/{\epsilon F_2^+}
}
\endxy
\quad,\quad
\xy
\xymatrix{\!\!
\xy
(0,0)*{\includegraphics[scale=1.15]{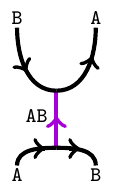}};
\endxy\!\!
\ar@<2pt>[r]^{\epsilon F_2^-}
&\!\!
\xy
(0,0)*{\includegraphics[scale=1.15]{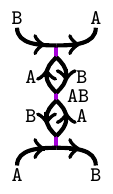}};
\endxy\!\!
\ar@<2pt>[r]^{\epsilon G_2^-}
\ar@<2pt>[l]^/-.05cm/{G_2^-}
&\!\!
\xy
(0,0)*{\includegraphics[scale=1.15]{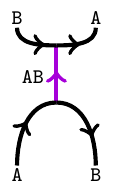}};
\endxy
\ar@<2pt>[l]^/-.05cm/{ F_2^-}
}
\endxy
\end{align*}
Here $\epsilon$ and $\epsilon^{\prime}$ are the
signs coming from the scaling of the Reidemeister foams, but which are irrelevant as they cancel 
in the composition. The composite foams are 
isotopic to identity foams, e.g. for the compositions 
from left to right we get:
\[
\xy
(0,0)*{\includegraphics[scale=1.15]{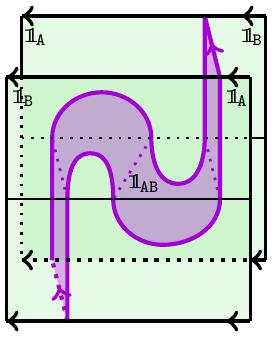}};
\endxy
\quad,\quad
\xy
(0,0)*{\includegraphics[scale=1.15]{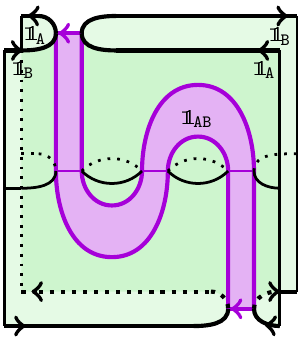}};
\endxy
\]
The case $a<b$ 
can be proven analogously.
\end{proof}

\begin{lemma}\label{lemma:MM10}
The movie move $\MM{10}$ holds on simple resolutions.
\end{lemma}

\begin{proof}
This movie move exists in a large number of variants, all of which 
are equivalent modulo far-commutation and the already 
established $\MM{6}$ by a beautiful argument 
of Clark--Morrison--Walker \cite[Proof of $\MM{10}$]{CMW}. This 
directly extends to the colored case and we only need to check the 
variant of $\MM{10}$ displayed in \fullref{fig:movie2} for 
$a\geq b\geq c\geq d$. It is given on simple resolutions as the 
following composite starting at the top left diagram:
\begin{gather*}
\xy
\xymatrix{
\xy
(0,.75)*{\begin{tikzpicture}[anchorbase,scale=1.15]
		\draw [dotted] (-.9,-1) to (.9,-1);
		\node at (0,0) {\includegraphics[scale=1.15]{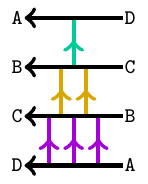}}; 
\end{tikzpicture}};
\endxy
\!
\ar@<2pt>[r]^{G_3^+}
&\!
\xy
(0,0)*{\includegraphics[scale=1.15]{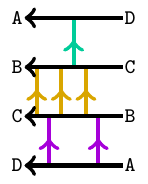}};
\endxy\!
\ar@<2pt>[r]^{G_3^+}
&\!
\xy
(0,0)*{\includegraphics[scale=1.15]{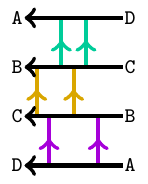}};
\endxy\!
\ar@<2pt>[r]^{G_3^+}
&\!
\xy
(0,0)*{\includegraphics[scale=1.15]{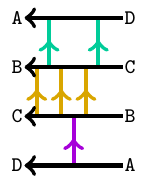}};
\endxy\!
\ar@<2pt>[dr]^{G_3^+}
&\!
\\
&\!
\xy
(0,0)*{\includegraphics[scale=1.15]{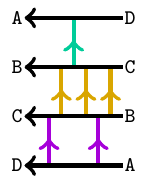}};
\endxy\!
\ar@<2pt>[ul]^{F_3^+}
&\!
\xy
(0,0)*{\includegraphics[scale=1.15]{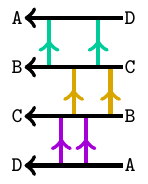}};
\endxy\!
\ar@<2pt>[l]^{F_3^+}
&\!
\xy
(0,0)*{\includegraphics[scale=1.15]{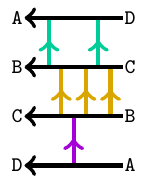}};
\endxy\!
\ar@<2pt>[l]^{F_3^+}
&\!
\xy
(0,0)*{\includegraphics[scale=1.15]{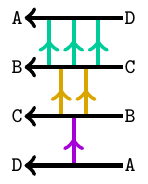}};
\endxy
\ar@<2pt>[l]^{F_3^+}
}
\endxy
\end{gather*}
The composite of the first four maps $G_3^+$ is given by the following foam:
\begin{equation}
\label{eq:MM10foam}
\xy
(0,0)*{\includegraphics[scale=1.15]{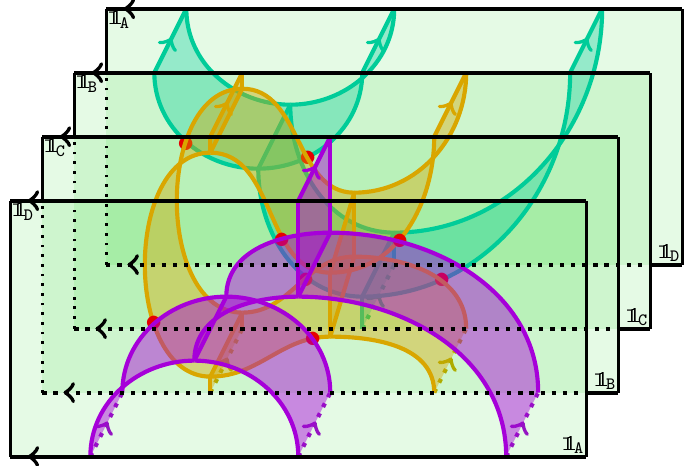}};
\endxy
\;,\;
\xy
(0,0)*{\includegraphics[scale=1.15]{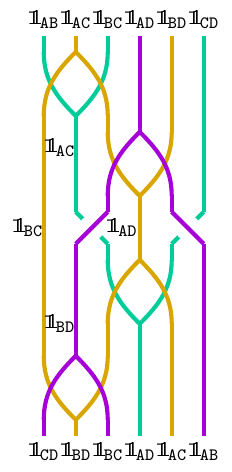}};
\endxy
\end{equation}
The composite of the remaining four maps $F_3^+$ is given by the foam obtained from the one above by reflecting in a line perpendicular to the green facets. The complete movie is thus represented by a \fcd\!\!, which is symmetric in the origin of $\R^2$, the lower half of which is displayed in \eqref{eq:MM10foam}. It remains to show that this \fcd represents the identity foam. 
As a first step, we focus on the interaction of 
purple $\purplebox$ and golden 
$\goldenbox$ facets on the second green 
upright plane. This allows us to draw orientations 
on the \fcd and the following simplification:

\begin{equation}
\label{eq:MM10purple}
\xy
(0,0)*{\includegraphics[scale=1.15]{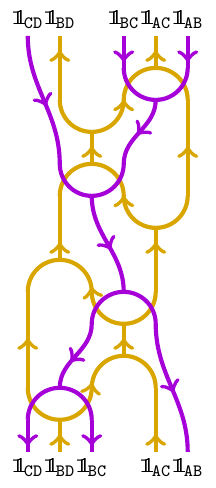}};
\endxy
\!\!\!\!=\!\!\!\!
\xy
(0,0)*{\includegraphics[scale=1.15]{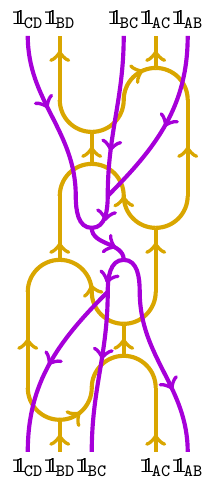}};
\endxy
\!\!\!\!\!\!\!\!= 
\xy
(0,0)*{\displaystyle\frac{1}{\scalar{BC,AB}}};
(0,-6.5)*{\scalar{CD,AC}};
\endxy
\!\!\!\!
\xy
(0,0)*{\includegraphics[scale=1.15]{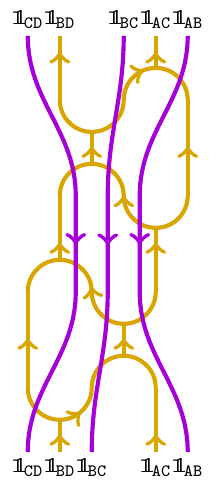}};
\endxy
\!\!\!\!\!\!\!\!
= \frac{\scalar{AB,BD}}{\scalar{BC,AB}}
\xy
(0,0)*{\includegraphics[scale=1.15]{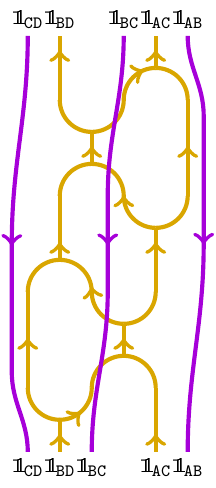}};
\endxy
\end{equation}
The outer two purple $\purplebox$ seams are completely separated 
from the rest of the diagram and will not be displayed in subsequent computations.

In the second step, we aim to simplify the golden $\goldenbox$ 
facets of the foam. For this we 
observe that can locally apply relations of the form:
\begin{equation}\label{eq:MM10-calc2}
\xy
(0,0)*{\includegraphics[scale=1.15]{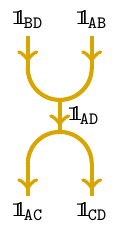}};
\endxy
= \scalar{BC, Y}
\xy
(0,0)*{\includegraphics[scale=1.15]{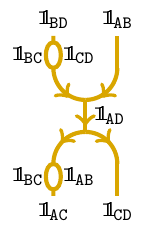}};
\endxy
= \scalar{BC, Y}
\xy
(0,0)*{\includegraphics[scale=1.15]{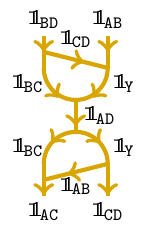}};
\endxy
=
\xy
(0,0)*{\includegraphics[scale=1.15]{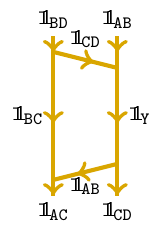}};
\endxy
\end{equation}
Here $\wcolor{\AandC}=\wcolor{AB}\cup\wcolor{CD}$.
After pushing the cyan $\cyanbox$ facets outwards in a manner analogous to the first step in \eqref{eq:MM10purple}, we can perform the final sequence of simplifications 
which we need. For this, we start with two applications of relations of type \eqref{eq:MM10-calc2} 
(with parallel running purple $\purplebox$ and cyan $\cyanbox$
strands in the middle):
\begin{gather*}
\frac{\scalar{AB,BD}}{\scalar{BC,AB}}
\xy
(0,0)*{\includegraphics[scale=1.15]{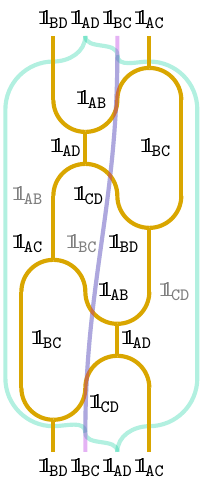}};
\endxy
= \frac{\scalar{AB,BD}}{\scalar{BC,AB}}
\xy
(0,0)*{\includegraphics[scale=1.15]{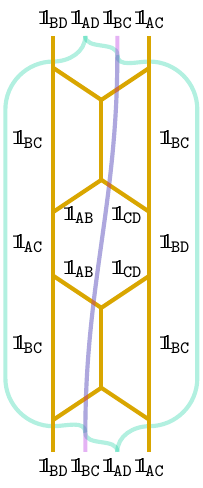}};
\endxy
\xy
(0,0)*{\displaystyle\frac{\scalar{AB,CD}}{\scalar{AB,BC}}};
(0,-6.5)*{\scalar{CD,BC}};
\endxy
\xy
(0,0)*{\includegraphics[scale=1.15]{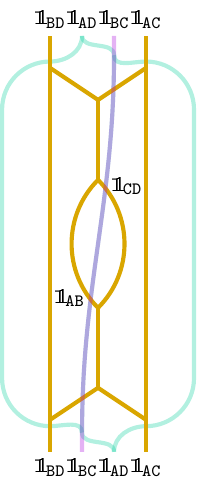}};
\endxy
\\
\xy
(0,0)*{\displaystyle\frac{1}{\scalar{CD,BC}}};
(0,-6.5)*{\scalar{AB,BC}};
\endxy
\xy
(0,0)*{\includegraphics[scale=1.15]{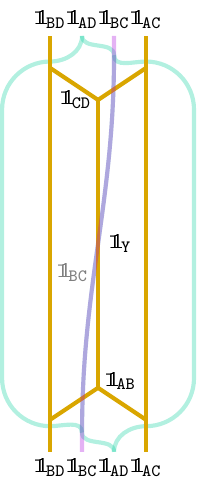}};
\endxy
= \xy
(0,0)*{\displaystyle\frac{\scalar{CD,AB}}{\scalar{BC,CD}}};
(0,-6.5)*{\scalar{AB,BC}};
\endxy
\xy
(0,0)*{\includegraphics[scale=1.15]{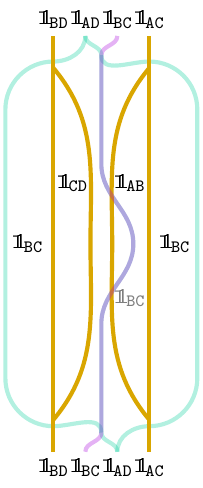}};
\endxy
=\frac{(-1)^{(a-b)(b-c)}}{\scalar{CD,AB}}
\xy
(0,0)*{\includegraphics[scale=1.15]{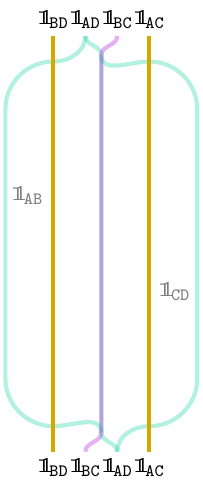}};
\endxy
\end{gather*}
In these steps have indicated the interaction of 
golden $\goldenbox$ facets with purple $\purplebox$ and cyan $\cyanbox$ ones, 
but we do not draw orientations.
 
The simplification use the \fcd relations 
from \fullref{subsec:special}, which also 
exhibit the last diagram as an identity foam in disguise.  
\end{proof}

We skip $\MM{11}$ as this just encodes an isotopy relation. For the remaining non-reversible movie moves, we denote the chain maps corresponding to cups, saddles and caps by $M_0$, $M_1$ and $M_2$.
 
\begin{lemma}\label{lemma:MM12}
The movie move $\MM{12}$ holds on simple resolutions.
\end{lemma}

\begin{proof}
The movie on left-hand side in \fullref{fig:movie3} is given on simple resolutions by:
\begin{align*}
\xy
\xymatrix{
\emptyset\quad\ar@<2pt>[r]^/-.25cm/{M_0}
&
\quad
\xy
(0,0)*{\includegraphics[scale=1.15]{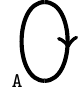}};
\endxy
\quad
\ar@<2pt>[r]^/-.05cm/{M_0}
\ar@<2pt>[l]^/.25cm/{M_2}
&
\quad
\xy
(0,0)*{\includegraphics[scale=1.15]{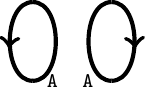}};
\endxy
\ar@<2pt>[l]^{\scalar{A} M_2}
}
\endxy
\end{align*}
From the right-hand side we get:
\begin{align*}
\xy
\xymatrix{
\emptyset\quad\ar@<2pt>[r]^/-.25cm/{M_0}
&
\quad
\xy
(0,0)*{\includegraphics[scale=1.15]{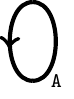}};
\endxy
\quad
\ar@<2pt>[r]^/-.05cm/{M_0}
\ar@<2pt>[l]^/.25cm/{M_2}
&
\quad
\xy
(0,0)*{\includegraphics[scale=1.15]{figs/fig-3-93}};
\endxy
\ar@<2pt>[l]^{\scalar{A}M_2}
}
\endxy
\end{align*}
These composites differ from the former ones only by isotopies.
\end{proof}

\begin{lemma}\label{lemma:MM13}
The movie move $\MM{13}$ holds on simple resolutions.
\end{lemma}

\begin{proof}
Again, we compare the movies on the 
left- and right-hand sides in \fullref{fig:movie3} on simple resolutions 
as follows:
\begin{gather*}
\xy
\xymatrix{\!\!
\xy
(0,0)*{\includegraphics[scale=1.15]{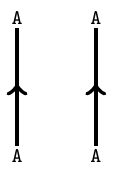}};
\endxy
\quad
\ar@<2pt>[r]^/-.1cm/{M_0}
&\quad
\xy
(0,0)*{\includegraphics[scale=1.15]{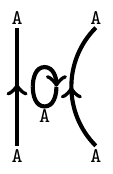}};
\endxy\quad
\ar@<2pt>[r]^/.15cm/{M_1}
\ar@<2pt>[l]^/.05cm/{\scalar{A}M_2}
&\quad
\xy
(0,0)*{\includegraphics[scale=1.15]{figs/fig-3-95}};
\endxy\!\!
\ar@<2pt>[l]^/-.15cm/{M_1}
}
\endxy
\\
\raisebox{0.001cm}{\xy
\xymatrix{\!\!
\xy
(0,0)*{\includegraphics[scale=1.15]{figs/fig-3-95}};
\endxy\quad
\ar@<2pt>[r]^/-.1cm/{M_0}
&\quad
\xy
(0,0)*{\includegraphics[scale=1.15]{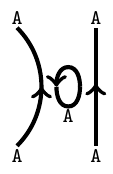}};
\endxy\quad
\ar@<2pt>[r]^/.15cm/{M_1}
\ar@<2pt>[l]^/.05cm/{\scalar{A}M_2}
&\quad
\xy
(0,0)*{\includegraphics[scale=1.15]{figs/fig-3-95}};
\endxy\!\!
\ar@<2pt>[l]^/-.15cm/{M_1}
}
\endxy}
\end{gather*}
These foams differ only by isotopies. 
\end{proof}

\begin{lemma}\label{lemma:MM14}
The movie move $\MM{14}$ holds on simple resolutions.
\end{lemma}

\begin{proof}
As above, we compare the movies on the left- and right-hand sides in Figure 12
on simple resolutions for $a\geq b$:
\[
\xy
\xymatrix{
\xy
(0,0)*{\includegraphics[scale=1.15]{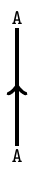}};
\endxy
\ar@<2pt>[r]^/-.1cm/{M_0}
&
\xy
(0,0)*{\includegraphics[scale=1.15]{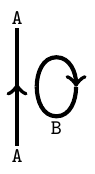}};
\endxy\,
\ar@<2pt>[r]^/-.1cm/{\epsilon F_2^+}
\ar@<2pt>[l]^/.1cm/{M_2}
&\,
\xy
(0,0)*{\includegraphics[scale=1.15]{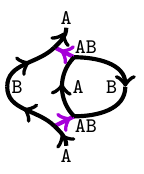}};
\endxy
\ar@<2pt>[l]^/.1cm/{G_2^+}
}
\endxy
\quad,\quad
\raisebox{0.1cm}{\xy
\xymatrix{
\xy
(0,.75)*{\includegraphics[scale=1.15]{figs/fig-3-98}};
\endxy
\ar@<2pt>[r]^/-.1cm/{M_0}
&\,
\xy
(0,0)*{\includegraphics[scale=1.15]{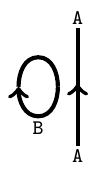}};
\endxy
\ar@<2pt>[r]^/-.1cm/{\epsilon F_2^-}
\ar@<2pt>[l]^/.1cm/{M_2}
&\,
\xy
(0,0)*{\includegraphics[scale=1.15]{figs/fig-3-100}};
\endxy
\ar@<2pt>[l]^/.1cm/{G_2^-}
}
\endxy}
\] 
The scalars are identical and the foams 
are isotopic, e.g. reading left-to right gives:
\[
\reflectbox{
\xy
(0,0)*{\includegraphics[scale=1.15]{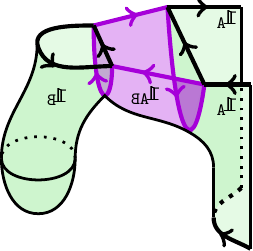}};
\endxy
}
\quad,\quad
\reflectbox{
\xy
(0,0)*{\includegraphics[scale=1.15]{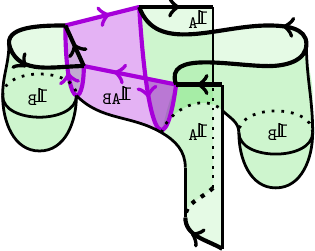}};
\endxy
}
\]
The case $b\geq a$ and a second variant with different relative orientations is proved analogously.
\end{proof}

\begin{lemma}\label{lemma:MM15}
The movie move $\MM{15}$ holds on simple resolutions.
\end{lemma}

\begin{proof}
In the left- and right-hand cases from \fullref{fig:movie3}, 
the movie takes the following form on the simple resolutions
for $a\geq b$:
\begin{align*}
&
\xy
\xymatrix{
\xy
(0,0)*{\includegraphics[scale=1.15]{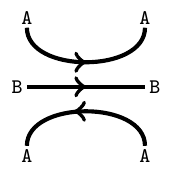}};
\endxy
\ar@<2pt>[r]^{F_2^+}
&
\xy
(0,0)*{\includegraphics[scale=1.15]{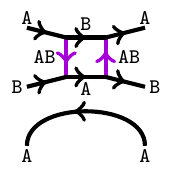}};
\endxy
\ar@<2pt>[r]^/.05cm/{M_1}
\ar@<2pt>[l]^{\epsilon G_2^+}
&\!
\xy
(0,-.75)*{\includegraphics[scale=1.15]{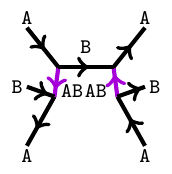}};
\endxy
\ar@<2pt>[l]^/-.05cm/{M_1}
}
\endxy\\
&
\raisebox{0.1cm}{\xy
\xymatrix{
\xy
(0,0)*{\includegraphics[scale=1.15]{figs/fig-3-104}};
\endxy
\ar@<2pt>[r]^{F_2^-}
&
\xy
(0,0)*{\includegraphics[scale=1.15]{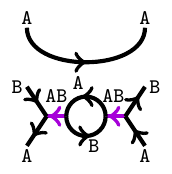}};
\endxy
\ar@<2pt>[r]^/.05cm/{M_1}
\ar@<2pt>[l]^{\epsilon G_2^-}
&\!
\xy
(0,-.75)*{\includegraphics[scale=1.15]{figs/fig-3-106}};
\endxy
\ar@<2pt>[l]^/-.05cm/{M_1}
}
\endxy}
\end{align*}
As before, it remains to compare the composite foams since 
the scalars are the same. For reading left-to-right we get the following two isotopic foams:
\[
\xy
(0,0)*{\includegraphics[scale=1.15]{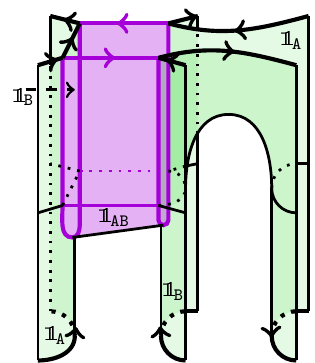}};
\endxy
\quad,\quad
\xy
(0,0)*{\includegraphics[scale=1.15]{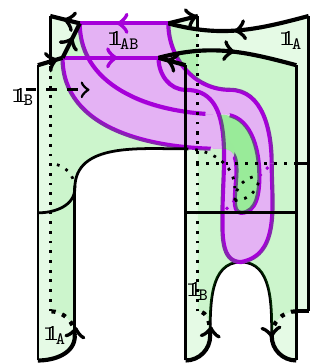}};
\endxy
\]
The other reading direction, 
the case $b\geq a$, as well as the variant 
obtained by switching orientations of the $a$-labeled strands  
work similar.
\end{proof}

\begin{remark}\label{remark:integrality-4}
(Integrality.) 
We know that also in the integral 
case, the chain maps on both sides of the 
movie moves differ by at most a scalar. That 
these scalars are all equal to one follows by 
extending scalars to $\C$, specializing via $\spec$ 
and using the results of this section. This is possible 
since the scaling for Reidemeister homotopy equivalences 
used here is integral by \fullref{lemma:integral}.
\end{remark}

%
\bibliographystyle{alphaurl}
\bibliography{functoriality}
\end{document}